\documentclass{article}
\usepackage[a4paper,top=3cm,bottom=2cm,left=3cm,right=3cm]{geometry}
\usepackage[utf8]{inputenc}

\usepackage{amsmath}
\usepackage{amssymb}
\usepackage{amsthm}
\usepackage{mathrsfs}
\usepackage{tikz-cd}
\usepackage{comment}
\usepackage{stmaryrd}
\usepackage{hyperref}

\newtheorem{proposition}{Proposition}[section]
\newtheorem{definition}[proposition]{Definition}
\newtheorem{lemma}[proposition]{Lemma}
\newtheorem{theorem}[proposition]{Theorem}

\newtheorem{example}[proposition]{Example}
\newtheorem{corollary}[proposition]{Corollary}

\usepackage{authblk}

\numberwithin{equation}{section}

\title{Braided Cartan Calculi and Submanifold Algebras}

\author{Thomas Weber\footnote{thomas.weber@unina.it}}
\affil{Università di Napoli “FEDERICO II”
and
I.N.F.N. Sezione di Napoli,\newline
Complesso MSA,
Via Cintia,
80126 Napoli,
Italy}

\date{January 17, 2020}

\begin{document}

\maketitle

\abstract{
We construct a noncommutative Cartan calculus on any braided commutative algebra
and study its applications in noncommutative geometry. The braided Lie derivative,
insertion and de Rham differential are introduced and related via graded braided commutators,
also incorporating the braided Schouten-Nijenhuis bracket. The resulting
braided Cartan calculus generalizes the Cartan calculus on smooth manifolds and 
the twisted Cartan calculus.
While it is a necessity of derivation based Cartan calculi on 
noncommutative algebras to employ central bimodules our approach
allows to consider bimodules over the full underlying algebra.
Furthermore, equivariant covariant derivatives
and metrics on braided commutative algebras are discussed. In particular, we prove
the existence and uniqueness of an equivariant Levi-Civita covariant derivative for any 
fixed non-degenerate equivariant metric. Operating in a symmetric braided monoidal
category we argue that Drinfel'd twist deformation corresponds to gauge
equivalences of braided Cartan calculi. The notions of equivariant covariant derivative 
and metric are compatible with the Drinfel'd functor as well.
Moreover, we project braided Cartan calculi to submanifold algebras and prove that
this process commutes with twist deformation.}

\tableofcontents

\section{Introduction}

In \cite{Woronowicz1989} Stanisław Lech Woronowicz generalized the notion of
Cartan calculus to quantum groups. The crucial ingredient is given by
the de Rham differential, which is understood as a linear map 
$\mathrm{d}\colon H\rightarrow\Gamma$ from a Hopf algebra $H$ to
a bicovariant $H$-bimodule $\Gamma$, generated by $\mathcal{A}$ and $\mathrm{d}$,
such that the Leibniz rule 
$\mathrm{d}(ab)=(\mathrm{d}a)b+a\mathrm{d}b$ holds for all $a,b\in H$.
It is proven that such a first order calculus admits an extension to the exterior algebra.
Noncommutative calculi based on derivations rather than generalizations of
differential forms are discussed by Michel Dubois-Violette, Peter Michor and Peter Schupp
in \cite{D-VM96,D-VM94,Schupp1994,Schupp1993}, though differential forms are included as dual
objects to derivations. The latter approaches are suitable for general noncommutative
algebras in the setting of \textit{noncommutative geometry} \cite{Connes94}. However,
bimodules have to be considered over the center of the algebra. 
In these notes we are proposing an intermediate procedure, sticking to derivation
based calculi while incorporating a Hopf algebra symmetry to avoid central bimodules.
It is motivated by twisted Cartan calculi, a particular class of noncommutative Cartan
calculi in the overlap of \textit{deformation quantization}
\cite{Bayen1978,Waldmann2016} and \textit{quantum groups} \cite{ES2010,Ma95}.
Drinfel'd twists \cite{Dr83} are tools to deform Hopf algebras as
well as the representation theory of the Hopf algebra in a compatible way.
They experienced a lot of attention in the field of deformation
quantization since a Drinfel'd twist induces a star product if the corresponding
symmetry acts on a smooth manifold by derivations (c.f. \cite{Aschieri2008}).
Explicit examples of
star products are quite rare, so this connection was very desirable.
However, this should be taken with a grain of salt since there are several situations
\cite{Thomas2016,dAWe17} in which deformation quantization can not be obtained via 
a twisting procedure. More generally,
it was pointed out in \cite{Aschieri2006} that a Drinfel'd twist
leads to a noncommutative calculus, the so-called \textit{twisted Cartan calculus}.
The mentioned article even provides twisted covariant derivatives and metrics,
generalizing classical Riemannian geometry.
The additional braided symmetries appearing in this work were the main motivation
for the author to consider noncommutative Cartan calculi only depending on a triangular
structure rather than on the Drinfel'd twist itself. The appropriate categorical
framework for this generalization is provided in \cite{Schenkel2015,Schenkel2016}:
the category of equivariant braided symmetric bimodules with respect to a triangular
Hopf algebra and a braided commutative algebra is symmetric braided and monoidal
with respect to the tensor product over the algebra. Generalizing the algebraic
construction of the Cartan calculus to this category we obtain the 
\textit{braided Cartan calculus}. Vector fields are represented by the braided Lie 
algebra of braided derivations, multivector fields become a braided Gerstenhaber
algebra, while differential forms constitute a braided Graßmann algebra. 
On the categorical level a Drinfel'd twist corresponds to a functor and its action can be
understood as a gauge equivalence on the symmetric braided monoidal category 
(see \cite{AsSh14,Ka95}). We prove that this \textit{Drinfel'd functor} respects the
braided Cartan calculus in the sense that it intertwines the braided Lie derivative,
insertion, de Rham differential and Schouten-Nijenhuis bracket. Note that both,
the classical Cartan calculus and the twisted Cartan calculus, can be regarded as
braided Cartan calculi. The first one with respect
to any cocommutative Hopf algebra with trivial triangular structure and the latter
with respect to the twisted Hopf algebra, triangular structure and algebra.
In the same spirit we generalize covariant derivatives and metrics to the braided
symmetric setting. Note however that for simplicity we regard them to be equivariant
in addition, a requirement which excludes some interesting examples already in the
twisted case. However, this assumption assures compatibility with the Drinfel'd functor.
As yet another application of the braided Cartan calculus we study the braided Cartan
calculus on submanifold algebras and prove that they are projected from the ambient
algebra in accordance to gauge equivalences.
It would be interesting to generalize the braided Cartan calculus
to the setting of \cite{Schenkel2015}, to Lie-Rinehart algebras \cite{Hue1998}
and furthermore to connect the braided Cartan calculus
to Hochschild cohomology and the Cartan calculus introduced by Boris Tsygan (see e.g.
\cite{Tsygan2000,Tsygan2012}).

The paper is organized as follows: in Section~\ref{section2} we recall basic
properties of triangular Hopf algebras and study the symmetric braided monoidal category
of equivariant braided symmetric bimodules of a braided commutative algebra.
The Drinfel'd functor leads to a braided monoidal equivalence of this category
and the one corresponding to the twisted algebra and triangular Hopf algebra.
Our main result is developed in Section~\ref{section3}:
we generalize the construction of the Cartan calculus 
of a commutative algebra to braided commutative algebras by incorporating a
braided symmetry. Starting from the braided Lie algebra of braided derivations we
build the braided Gerstenhaber algebra of braided multivector fields. The braided
Schouten-Nijenhuis bracket is obtained by extending the braided commutator.
The dual braided exterior algebra constitutes the braided differential forms.
Then, the braided Lie derivative, insertion and de Rham differential are defined,
resulting in the braided Cartan relations.
In the special case of a commutative algebra we regain the commutation
relations of the classical Cartan calculus. Connecting to Section~\ref{section2}
we introduce a twist deformation of the braided Cartan calculus and prove that
it is isomorphic to the braided Cartan calculus on the twisted algebra corresponding to
the twisted triangular structure. This shows that our construction respects gauge
equivalence classes. As an application, we introduce equivariant covariant derivatives
and metrics, give several constructions like extending them to braided multivector fields
and differential forms and proving the existence and uniqueness of an equivariant
Levi-Civita covariant derivative for every non-degenerate equivariant metric. 
The Drinfel'd functor respects the constructions. Finally in Section~\ref{section4} we
study braided Cartan calculi on submanifold algebras. We show how to project
the algebraic structure and that this procedure commutes with twist deformation.
An explicit example, given by twist quantization of quadric surfaces of
$\mathbb{R}^3$, is elaborated in \cite{GaetanoThomas19}.

Throughout these notes every module is considered over a commutative ring $\Bbbk$.
The category ${}_\Bbbk\mathcal{M}$ of $\Bbbk$-modules is monoidal with respect
to the tensor product $\otimes$.
If not stated otherwise every algebra is assumed to be
unital and associative.
A map $\Phi\colon V^\bullet\rightarrow W^\bullet$ between graded modules
$V^\bullet=\bigoplus_{k\in\mathbb{Z}}V^k$ and 
$W^\bullet=\bigoplus_{k\in\mathbb{Z}}W^k$ is said to
be homogeneous of degree $k\in\mathbb{Z}$ if
$\Phi(V^\ell)\subseteq W^{k+\ell}$. We often write
$\Phi\colon V^\bullet\rightarrow W^{\bullet+k}$ in this case.
The graded commutator of two homogeneous maps
$\Phi,\Psi\colon V^\bullet\rightarrow V^\bullet$ of degree $k$ and $\ell$
is defined by 
$[\Phi,\Psi]=\Phi\circ\Psi-(-1)^{k\ell}\Psi\circ\Phi$.

\section{Preliminaries on Quantum Groups}\label{section2}

In this introductory section we recall the notion of triangular
Hopf algebra together with its braided monoidal category of representations.
Afterwards we show how to twist the algebraic structure by a $2$-cocycle
and in which sense this induces an equivalence on the categorical level.
In the last subsection we discuss equivariant algebra bimodules and their twist
deformation. The previous braided monoidal equivalence can be refined to the
bimodules which inherit a braided symmetry in addition if the algebra is
braided commutative.
For more details on (triangular) Hopf algebras we refer to the textbooks
\cite{ChPr94,Ka95,Ma95,Mo93}. The more experienced readers are recommended
to \cite{AsSh14,Schenkel2015,Schenkel2016,GiZh98} for a prompt discussion of
what is covered in this section.

\subsection{Triangular Hopf Algebras and their Representations}

In a shortcut we introduce the category of algebras over a commutative ring
$\Bbbk$ along with their representations. Dualizing the definition we obtain
coalgebras, combining to the notion of bialgebra if the algebra and coalgebra
structures respect each other. From the categorical perspective bialgebras are
those algebras whose category of representations is monoidal with respect
to the usual associativity and unit constraints. Integrating a
braiding in this category induces universal $\mathcal{R}$-matrices on
the bialgebra, while an additional antipode corresponds to a
rigid (braided) monoidal category and accordingly to a (triangular) Hopf
algebra on the algebraic level.

A \textit{$\Bbbk$-algebra} is a $\Bbbk$-module $\mathcal{A}$ endowed with
$\Bbbk$-linear maps $\mu\colon\mathcal{A}\otimes\mathcal{A}\rightarrow\mathcal{A}$
and $\eta\colon\Bbbk\rightarrow\mathcal{A}$, called \textit{product} and
\textit{unit} of $\mathcal{A}$, such that the identities
\begin{equation}
    \mu\circ(\mu\otimes\mathrm{id})
    =\mu\circ(\mathrm{id}\otimes\mu)
    \colon\mathcal{A}^{\otimes 3}\rightarrow\mathcal{A}
\end{equation}
and
\begin{equation}\label{eq02}
    \mu\circ(\eta\otimes\mathrm{id})
    =\mathrm{id}
    =\mu\circ(\mathrm{id}\otimes\eta)
    \colon\mathcal{A}\rightarrow\mathcal{A}
\end{equation}
hold, where we used the $\Bbbk$-module isomorphisms $\Bbbk\otimes\mathcal{A}\cong
\mathcal{A}\cong\mathcal{A}\otimes\Bbbk$ in eq.(\ref{eq02}). These
are the well-known associativity and unit properties. A $\Bbbk$-algebra $\mathcal{A}$
is said to be \textit{commutative} if $\mu_{21}=\mu$, where
$\mu_{21}\colon\mathcal{A}\otimes\mathcal{A}\ni(a\otimes b)\mapsto\mu(b\otimes a)
\in\mathcal{A}$. In the following we often drop the symbol $\mu$ and simply write
$a\cdot b$ or $ab$ for the product of two elements $a,b\in\mathcal{A}$.
The $\Bbbk$-algebras form a category ${}_\Bbbk\mathcal{A}$
with morphisms being \textit{$\Bbbk$-algebra homomorphisms}, i.e.
$\Bbbk$-linear maps $\phi\colon\mathcal{A}\rightarrow\mathcal{A}'$ between
$\Bbbk$-algebras $(\mathcal{A},\mu,\eta)$ and $(\mathcal{A}',\mu',\eta')$
such that
\begin{equation}
    \phi\circ\mu=\mu'\circ(\phi\otimes\phi)
    \colon\mathcal{A}\otimes\mathcal{A}\rightarrow\mathcal{A}'
    \text{ and }
    \phi\circ\eta=\eta'\colon\Bbbk\rightarrow\mathcal{A}'.
\end{equation}
Dualizing this concept
we define a \textit{$\Bbbk$-coalgebra} to be a $\Bbbk$-module $\mathcal{C}$
together with $\Bbbk$-linear maps $\Delta\colon\mathcal{C}
\rightarrow\mathcal{C}\otimes\mathcal{C}$ and
$\epsilon\colon\mathcal{C}\rightarrow\Bbbk$ satisfying
\begin{equation}
    (\Delta\otimes\mathrm{id})\circ\Delta
    =(\mathrm{id}\otimes\Delta)\circ\Delta
    \colon\mathcal{C}\rightarrow\mathcal{C}^{\otimes 3}
\end{equation}
and
\begin{equation}
    (\epsilon\otimes\mathrm{id})\circ\Delta
    =\mathrm{id}
    =(\mathrm{id}\otimes\epsilon)\circ\Delta
    \colon\mathcal{C}\rightarrow\mathcal{C}.
\end{equation}
The maps $\Delta$ and $\epsilon$ are said to be the \textit{coproduct}
and \textit{counit}
of $\mathcal{C}$ with the properties of being coassociative and
counital, respectively. We frequently use Sweedler's sigma notation
$c_{(1)}\otimes c_{(2)}$ to denote the coproduct $\Delta(c)$ of an element
$c\in\mathcal{C}$, omitting a possibly finite sum of factorizing elements.
By the coassociativity of $\Delta$ we further define
\begin{equation}
    c_{(1)}\otimes c_{(2)}\otimes c_{(3)}
    :=c_{(1)(1)}\otimes c_{(1)(2)}\otimes c_{(2)}
    =c_{(1)}\otimes c_{(2)(1)}\otimes c_{(2)(2)}
\end{equation}
and similarly for higher coproducts. A $\Bbbk$-coalgebra $\mathcal{C}$ is said to be
\textit{cocommutative} if $\Delta_{21}=\Delta$, where $\Delta_{21}(c)
=c_{(2)}\otimes c_{(1)}$ for all $c\in\mathcal{C}$. A
\textit{$\Bbbk$-coalgebra homomorphism} is a $\Bbbk$-linear map
$\psi\colon\mathcal{C}\rightarrow\mathcal{C}'$ between
$\Bbbk$-coalgebras $(\mathcal{C},\Delta,\epsilon)$ and
$(\mathcal{C}',\Delta',\epsilon')$ obeying the relations
\begin{equation}
    \Delta'\circ\psi=(\psi\otimes\psi)\circ\Delta
    \colon\mathcal{C}\rightarrow\mathcal{C}'\otimes\mathcal{C}'
    \text{ and }
    \epsilon'\circ\psi=\epsilon\colon\mathcal{C}\rightarrow\Bbbk.
\end{equation}
The category of $\Bbbk$-comodules is denoted by ${}_\Bbbk\mathcal{C}$.
\begin{example}
We give some elementary examples and constructions of (co)algebras, focusing
on the ones we need in the rest of these notes.
\begin{enumerate}
\item[i.)] The tensor product $\mathcal{A}\otimes\mathcal{A}'$ of two
$\Bbbk$-algebras $(\mathcal{A},\mu,\eta)$ and $(\mathcal{A}',\mu',\eta')$
becomes a $\Bbbk$-algebra with product
$$
\mu_{\mathcal{A}\otimes\mathcal{A}'}
=(\mu\otimes\mu')\circ(\mathrm{id}\otimes\tau_{\mathcal{A}',\mathcal{A}}
\otimes\mathrm{id})
\colon(\mathcal{A}\otimes\mathcal{A}')\otimes(\mathcal{A}\otimes\mathcal{A}')
\rightarrow\mathcal{A}\otimes\mathcal{A}'
$$
and unit $\eta_{\mathcal{A}\otimes\mathcal{A}'}=\eta\otimes\eta'$,
where we use the $\Bbbk$-module isomorphism $\Bbbk\otimes\Bbbk\cong\Bbbk$ in
the latter definition and $\tau_{\mathcal{A}',\mathcal{A}}
\colon\mathcal{A}'\otimes\mathcal{A}\rightarrow\mathcal{A}\otimes\mathcal{A}'$
denotes the tensor flip isomorphism.
Dually, the tensor product $\mathcal{C}\otimes\mathcal{C}'$ of two
$\Bbbk$-coalgebras $(\mathcal{C},\Delta,\epsilon)$ and 
$(\mathcal{C}',\Delta',\epsilon')$ can be structured as a $\Bbbk$-coalgebra
with coproduct
$$
\Delta_{\mathcal{C}\otimes\mathcal{C}'}
=(\mathrm{id}\otimes\tau_{\mathcal{C},\mathcal{C}'}\otimes\mathrm{id})
\circ(\Delta\otimes\Delta')
\colon\mathcal{C}\otimes\mathcal{C}'
\rightarrow(\mathcal{C}\otimes\mathcal{C}')\otimes(\mathcal{C}\otimes\mathcal{C}')
$$
and counit $\epsilon_{\mathcal{C}\otimes\mathcal{C}'}
=\epsilon\otimes\epsilon'$.

\item[ii.)] Any commutative ring $\Bbbk$ is a $\Bbbk$-(co)algebra with
product and unit given by its ring multiplication and unit element,
while the coproduct and counit are defined by 
$\Delta(\lambda)=\lambda(1\otimes 1)$ and $\epsilon(\lambda)=\lambda$
for all $\lambda\in\Bbbk$.
\end{enumerate}
\end{example}
A $\Bbbk$-algebra $(\mathcal{A},\mu,\eta)$ which is also a $\Bbbk$-coalgebra
with coproduct $\Delta$ and counit $\epsilon$ is said to be a
\textit{$\Bbbk$-bialgebra} if $\Delta$ and $\epsilon$ are
$\Bbbk$-algebra homomorphisms and $\mu$ and $\eta$ are $\Bbbk$-coalgebra
homomorphisms. It is clear by the symmetry in the definition of algebra
and coalgebra that a $\Bbbk$-algebra and $\Bbbk$-coalgebra
is a $\Bbbk$-bialgebra if and only if its algebra structures are
$\Bbbk$-coalgebra homomorphisms if and only if its coalgebra structures are
$\Bbbk$-algebra homomorphisms.
A \textit{$\Bbbk$-bialgebra homomorphism} is a $\Bbbk$-algebra
homomorphism between $\Bbbk$-bialgebras which is also a $\Bbbk$-coalgebra
homomorphism.
\begin{definition}
A $\Bbbk$-bialgebra $(H,\mu,\eta,\Delta,\epsilon)$ is said to be triangular
if there is an invertible element $\mathcal{R}\in H\otimes H$, called universal
$\mathcal{R}$-matrix or triangular structure, with inverse given by
$\mathcal{R}_{21}=\tau_{H,H}(\mathcal{R})$, such that
\begin{equation}\label{eq03}
    \Delta_{21}(\xi)
    =\mathcal{R}\Delta(\xi)\mathcal{R}^{-1}
    \text{ for all }\xi\in H,
\end{equation}
and the hexagon relations
\begin{equation}
    (\Delta\otimes\mathrm{id})(\mathcal{R})
    =\mathcal{R}_{13}\mathcal{R}_{23}
    \text{ and }
    (\mathrm{id}\otimes\Delta)(\mathcal{R})
    =\mathcal{R}_{13}\mathcal{R}_{12}    
\end{equation}
are satisfied, where $\mathcal{R}_{12}=\mathcal{R}\otimes 1,
~\mathcal{R}_{23}=1\otimes\mathcal{R},
~\mathcal{R}_{13}=(\mathrm{id}\otimes\tau_{H,H})(\mathcal{R}_{12})\in H^{\otimes 3}$.
Property (\ref{eq03}) states that $H$ is quasi-cocommutative.
The $\Bbbk$-bialgebra $H$ is said to be a $\Bbbk$-Hopf algebra if there is a bijective 
$\Bbbk$-linear map $S\colon H\rightarrow H$, called antipode, such that
\begin{equation}
    \mu\circ(S\otimes\mathrm{id})\circ\Delta
    =\eta\circ\epsilon
    =\mu\circ(\mathrm{id}\otimes S)\circ\Delta
    \colon H\rightarrow H
\end{equation}
holds. A $\Bbbk$-bialgebra homomorphism between $\Bbbk$-Hopf algebras is
said to be a $\Bbbk$-Hopf algebra homomorphism if it intertwines the antipodes.
We denote the category of $\Bbbk$-Hopf algebras by
${}_\Bbbk\mathcal{H}$. A $\Bbbk$-Hopf algebra 
$(H,\mu,\eta,\Delta,\epsilon,S)$ is called triangular if its underlying
bialgebra structure is.
\end{definition}
In the following we often drop the reference to the commutative ring $\Bbbk$
and simply refer to Hopf algebras etc.
Remark that there are slightly weaker definitions of Hopf algebra, not
assuming the antipode to have an inverse (see \cite{Ka95,Ma95,Mo93}).
We follow the convention of \cite{ChPr94}, arguing that in all
examples which are relevant for us the antipode is invertible
and we do not want to state this as an additional condition throughout.
One can show that the antipode $S$ of a bialgebra
$(H,\mu,\eta,\Delta,\epsilon)$ is unique if it exists and that
it is an anti-bialgebra homomorphism in the sense that
\begin{equation}
    S(\xi\chi)=S(\chi)S(\xi),~
    S(1)=1,~
    S(\xi)_{(1)}\otimes S(\xi)_{(2)}=S(\xi_{(2)})\otimes S(\xi_{(1)})
    \text{ and }
    \epsilon\circ S=\epsilon
\end{equation}
for all $\xi,\chi\in H$. If $H$ is commutative or cocommutative it 
follows that $S^2=\mathrm{id}$. Moreover, any cocommutative
Hopf algebra is triangular with universal $\mathcal{R}$-matrix given by
$\mathcal{R}=1\otimes 1$. Any universal $\mathcal{R}$-matrix $\mathcal{R}$
satisfies the \textit{quantum Yang-Baxter equation}
$$
\mathcal{R}_{12}\mathcal{R}_{13}\mathcal{R}_{23}
=\mathcal{R}_{23}\mathcal{R}_{13}\mathcal{R}_{12}.
$$
Fix a triangular $\Bbbk$-bialgebra $(H,\mu,\eta,\Delta,\epsilon,\mathcal{R})$
for the moment. We motivate its definition by elaborating that the
representation theory of $H$ has interesting categorical properties.
Recall that a \textit{representation} of $H$ is nothing but a
\textit{left $H$-module}, i.e. a $\Bbbk$-module $\mathcal{M}$ together
with a $\Bbbk$-linear map $\lambda\colon H\otimes\mathcal{M}
\rightarrow\mathcal{M}$, called left $H$-module action or left $H$-module structure,
such that
\begin{equation}
    \lambda\circ(\mathrm{id}_H\otimes\lambda)
    =\lambda\circ(\mu\otimes\mathrm{id}_\mathcal{M})
    \colon H\otimes H\otimes\mathcal{M}\rightarrow\mathcal{M}
\end{equation}
and $\lambda\circ(\eta\otimes\mathrm{id}_\mathcal{M})
=\mathrm{id}_\mathcal{M}$ hold. A \textit{left $H$-module homomorphism}
is a $\Bbbk$-linear map $\Phi\colon\mathcal{M}\rightarrow\mathcal{M}'$
between left $H$-modules $(\mathcal{M},\lambda)$ and $(\mathcal{M}',\lambda')$
such that
\begin{equation}
    \Phi\circ\lambda
    =\lambda'\circ(\mathrm{id}_H\otimes\Phi)
    \colon H\otimes\mathcal{M}\rightarrow\mathcal{M}'.
\end{equation}
We sometimes refer to left $H$-module homomorphisms as
\textit{$H$-equivariant maps}.
This forms the category ${}_H\mathcal{M}$ of left $H$-modules.
In the following we often write $\xi\cdot m$ instead of $\lambda(\xi\otimes m)$
for a left $H$-module $(\mathcal{M},\lambda)$, where $\xi\in H$ and 
$m\in\mathcal{M}$. Note that until now we only used the algebra structure
of $H$ in the definition of ${}_H\mathcal{M}$. In other words, we can
consider the category of representations for any algebra.
However, since $\Delta$ and $\epsilon$ are
algebra homomorphisms we can define a left $H$-module action on
the tensor product of two left $H$-modules $(\mathcal{M},\lambda)$
and $(\mathcal{M}',\lambda')$ by
$$
\lambda_{\mathcal{M}\otimes\mathcal{M}'}
=(\lambda\otimes\lambda')\circ
(\mathrm{id}_H\otimes\tau_{H,\mathcal{M}}\otimes\mathrm{id}_{\mathcal{M}'})
\circ(\Delta\otimes\mathrm{id}_{\mathcal{M}\otimes\mathcal{M}'})
\colon H\otimes(\mathcal{M}\otimes\mathcal{M}')
\rightarrow\mathcal{M}\otimes\mathcal{M}'
$$
and a left $H$-module action on $\Bbbk$ by
$$
\lambda_\Bbbk=(\epsilon\otimes\mathrm{id}_\Bbbk)
\colon H\otimes\Bbbk\rightarrow\Bbbk\otimes\Bbbk\cong\Bbbk.
$$
Those actions respect the usual associativity and unit constraints of the
tensor product of $\Bbbk$-modules because $\Delta$ is coassociative and
$\epsilon$ satisfies the counit axiom. In other words,
$({}_H\mathcal{M},\otimes)$ is a monoidal category. The universal 
$\mathcal{R}$-matrix $\mathcal{R}$ induces a symmetric braiding on this category
by defining
\begin{equation}
    c^\mathcal{R}_{\mathcal{M},\mathcal{M}'}(m\otimes m')
    =\mathcal{R}^{-1}\cdot(m'\otimes m)
    \in\mathcal{M}'\otimes\mathcal{M}
    \text{ for all }m\in\mathcal{M},~m'\in\mathcal{M}'.
\end{equation}
In fact, the hexagon relations of $\mathcal{R}$ correspond to the hexagon
relations of $c^\mathcal{R}$ and $c^\mathcal{R}_{\mathcal{M},\mathcal{M}'}\circ
c^\mathcal{R}_{\mathcal{M}',\mathcal{M}}
=\mathrm{id}_{\mathcal{M}'\otimes\mathcal{M}}$
since $\mathcal{R}_{21}$ is the inverse of $\mathcal{R}$.
Conversely, any symmetric braiding $c$ on $({}_H\mathcal{M},\otimes)$ determines a
triangular structure $\mathcal{R}=\tau_{H,H}(c_{H,H}(1\otimes 1))\in H\otimes H$, where
$H$ acts on itself by left multiplication.
\begin{proposition}[\cite{Ka95}~Proposition~XIII.1.4.]
The representation theory ${}_H\mathcal{M}$ of a $\Bbbk$-bialgebra
is a monoidal category. It is braided symmetric if and only if
$H$ is triangular.
\end{proposition}
In the case of a Hopf algebra
$(H,\mu,\eta,\Delta,\epsilon,S)$ we receive an additional
rigidity property of its monoidal category in the sense that every
left $H$-module admits a left and right dual module. However, for
this we have to restrict our consideration to finitely generated
projective $\Bbbk$-modules ${}_\Bbbk\mathcal{M}^f$.
The antipode of $H$ can be used
to transfer the rigidity property from ${}_\Bbbk\mathcal{M}^f$
to ${}_H\mathcal{M}^f$. Denote the usual dual pairing of a
finitely generated projective $\Bbbk$-module $\mathcal{M}$ and its
dual module $\mathcal{M}^*$ by $\langle\cdot,\cdot\rangle\colon
\mathcal{M}^*\otimes\mathcal{M}\rightarrow\Bbbk$.
\begin{proposition}[\cite{ChPr94}~Example~5.1.4]
Let $H$ be a $\Bbbk$-Hopf algebra and consider the monoidal category
${}_H\mathcal{M}$ of left $H$-modules.
The monoidal subcategory ${}_H\mathcal{M}^f$ of
finitely generated projective left $H$-modules is rigid, where the
left and right dual $\mathcal{M}^*$ and ${}^*\mathcal{M}$ of an object
$\mathcal{M}$ in ${}_H\mathcal{M}^f$ are defined as the finitely generated
projective $\Bbbk$-module $\mathcal{M}^*$ with left $H$-module action
given by
$$
\langle\xi\cdot\alpha,m\rangle
=\langle\alpha,S(\xi)\cdot m\rangle
$$
and
$$
\langle\xi\cdot\alpha,m\rangle
=\langle\alpha,S^{-1}(\xi)\cdot m\rangle
$$
for all $\xi\in H$, $m\in\mathcal{M}$ and $\alpha\in\mathcal{M}^*$,
respectively. The forgetful functor
\begin{equation}
    F\colon{}_H\mathcal{M}^f\rightarrow {}_\Bbbk\mathcal{M}^f
\end{equation}
is monoidal.
\end{proposition}

\subsection{Drinfel'd Twist Deformation}

In this subsection we introduce Drinfel'd twists as an instrument to deform
(triangular) Hopf algebra structures. It turns out that the
representation theory of the deformed (triangular) Hopf algebra
is (braided) monoidally equivalent the representation theory of the undeformed
(triangular) Hopf algebra. The definition of
Drinfel'd twist originates from \cite{Dr83}, while the monoidal
equivalence was proven in \cite{Dr89}. We further refer to 
\cite{AsSh14,GiZh98} for a discussion of this topic.
Fix a Hopf algebra $(H,\mu,\eta,\Delta,\epsilon,S)$ in the following.
\begin{definition}
A (Drinfel'd) twist on $H$ is an invertible element
$\mathcal{F}\in H\otimes H$ satisfying the $2$-cocycle
condition
\begin{equation}\label{eq19}
    (\mathcal{F}\otimes 1)(\Delta\otimes\mathrm{id})(\mathcal{F})
    =(1\otimes\mathcal{F})(\mathrm{id}\otimes\Delta)(\mathcal{F})
\end{equation}
and the normalization condition
$
(\epsilon\otimes\mathrm{id})(\mathcal{F})
=1
=(\mathrm{id}\otimes\epsilon)(\mathcal{F}).
$
\end{definition}
There are several examples and constructions of Drinfel'd twists,
showing that this is a rich concept. We refer the interested reader to
\cite{Jonas2017,Pachol2017}.
It follows that the inverse $\mathcal{F}^{-1}$ of a twist $\mathcal{F}$ on
$H$ is \textit{normalized}, i.e. 
$(\epsilon\otimes\mathrm{id})(\mathcal{F}^{-1})
=1
=(\mathrm{id}\otimes\epsilon)(\mathcal{F}^{-1})$
and satisfies the so called \textit{inverse $2$-cocycle condition}
\begin{equation}
    (\Delta\otimes\mathrm{id})(\mathcal{F}^{-1})(\mathcal{F}^{-1}\otimes 1)
    =(\mathrm{id}\otimes\Delta)(\mathcal{F}^{-1})(1\otimes\mathcal{F}^{-1}).
\end{equation}
Any element $\mathcal{F}\in H\otimes H$ can be written as a finite sum of
factorizing elements $\mathcal{F}_1^i\otimes\mathcal{F}_2^i$,
$\mathcal{F}_1^i,\mathcal{F}_2^i\in H$.
In the following we usually omit this finite sum and simply write
$\mathcal{F}=\mathcal{F}_1\otimes\mathcal{F}$, which is called
\textit{leg notation}. Using this convention, the $2$-cocycle
(\ref{eq19}) condition reads
\begin{equation}
    \mathcal{F}_1\mathcal{F}'_{1(1)}\otimes\mathcal{F}_2\mathcal{F}'_{1(2)}
    \otimes\mathcal{F}'_2
    =\mathcal{F}'_1\otimes\mathcal{F}_1\mathcal{F}'_{2(1)}\otimes
    \mathcal{F}_2\mathcal{F}'_{2(2)},
\end{equation}
where we marked the second copy of $\mathcal{F}$ by $\mathcal{F}
=\mathcal{F}'_1\otimes\mathcal{F}'_2$ to distinguish the summations.
The following proposition (c.f. \cite{Ma95}~Theorem~2.3.4)
reveals the utility of Drinfel'd twists as they provide a construction of
(triangular) Hopf algebras from given ones.
\begin{proposition}
Consider a twist $\mathcal{F}$ on $H$. Then
$H_\mathcal{F}=(H,\mu,\eta,\Delta_\mathcal{F},\epsilon,S_\mathcal{F})$
is a Hopf algebra with coproduct and antipode given by
\begin{equation}
    \Delta_\mathcal{F}(\xi)=\mathcal{F}\Delta(\xi)\mathcal{F}^{-1}
    \text{ and }
    S_\mathcal{F}(\xi)=\beta S(\xi)\beta^{-1},
\end{equation}
respectively, for all $\xi\in H$, where
$\beta=\mathcal{F}_1S(\mathcal{F}_2)\in H$. If $H$ is triangular
with universal $\mathcal{R}$-matrix $\mathcal{R}$, so is $H_\mathcal{F}$
with universal $\mathcal{R}$-matrix 
$\mathcal{R}_\mathcal{F}=\mathcal{F}_{21}\mathcal{R}\mathcal{F}^{-1}$.
\end{proposition}
Let $\mathcal{F}$ be a twist on $H$ and consider the corresponding
monoidal category $({}_{H_\mathcal{F}}\mathcal{M},\otimes_\mathcal{F})$
of representations of $H_\mathcal{F}$. Since $H$ and $H_\mathcal{F}$ coincide as
algebras every left $H$-module is automatically a left $H_\mathcal{F}$-module
and vice versa. However, the actions on the tensor product of modules
differ in general. For this reason we denote the monoidal structure of
${}_{H_\mathcal{F}}\mathcal{M}$ by $\otimes_\mathcal{F}$. Namely,
for two left $H$-modules (or equivalently two left $H_\mathcal{F}$-modules)
$\mathcal{M}$ and $\mathcal{M}'$ the tensor product
$\mathcal{M}\otimes_\mathcal{F}\mathcal{M}'$ coincides with
$\mathcal{M}\otimes\mathcal{M}'$ as a $\Bbbk$-module but
$\mathcal{M}\otimes_\mathcal{F}\mathcal{M}'$ is a left $H_\mathcal{F}$-module
via
\begin{equation}
    \xi\cdot(m\otimes_\mathcal{F}m')
    =(\xi_{\widehat{(1)}}\cdot m)\otimes_\mathcal{F}(\xi_{\widehat{(2)}}\cdot m'),
\end{equation}
where $\Delta_\mathcal{F}(\xi)=\xi_{\widehat{(1)}}\otimes\xi_{\widehat{(2)}}$,
while $\mathcal{M}\otimes\mathcal{M}'$ is a left $H_\mathcal{F}$-module
via
$$
\xi\cdot(m\otimes m')
=(\xi_{(1)}\cdot m)\otimes(\xi_{(2)}\cdot m')
$$
for all $\xi\in H_\mathcal{F}$, $m\in\mathcal{M}$ and $m'\in\mathcal{M}'$.
We are able to compare those pictures via a left $H_\mathcal{F}$-module
isomorphism
\begin{equation}
    \varphi_{\mathcal{M},\mathcal{M}'}
    \colon\mathcal{M}\otimes_\mathcal{F}\mathcal{M}'
    \ni(m\otimes_\mathcal{F}m')
    \mapsto(\mathcal{F}^{-1}_1\cdot m)\otimes(\mathcal{F}^{-1}_2\cdot m')
    \in\mathcal{M}\otimes\mathcal{M}'.
\end{equation}
In fact, $\varphi_{\mathcal{M},\mathcal{M}'}$ intertwines the
left $H_\mathcal{F}$-module actions, since
\begin{align*}
    \varphi_{\mathcal{M},\mathcal{M}'}(\xi\cdot(m\otimes_\mathcal{F}m'))
    =((\xi_{(1)}\mathcal{F}_1^{-1})\cdot m)
    \otimes((\xi_{(2)}\mathcal{F}_2^{-1})\cdot m)
    =\xi\cdot\varphi_{\mathcal{M},\mathcal{M}'}(m\otimes_\mathcal{F}m')
\end{align*}
for all $\xi\in H$, $m\in\mathcal{M}$ and $m'\in\mathcal{M}'$ and
admits an inverse left $H_\mathcal{F}$-module homomorphism
$$
\varphi^{-1}_{\mathcal{M},\mathcal{M}'}
\colon\mathcal{M}\otimes\mathcal{M}'
\ni(m\otimes m')
\mapsto(\mathcal{F}_1\cdot m)\otimes_\mathcal{F}(\mathcal{F}_2\cdot m')
\in\mathcal{M}\otimes_\mathcal{F}\mathcal{M}'.
$$
The map $\varphi$ gives rise to a monoidal equivalence. We formulate this in the
following theorem (c.f. \cite{Ka95}~Lemma~XV.3.7.).
\begin{theorem}\label{thm01}
For any twist $\mathcal{F}$ on $H$ there is a monoidal equivalence
of the monoidal categories $({}_H\mathcal{M},\otimes)$ and
$({}_{H_\mathcal{F}}\mathcal{M},\otimes_\mathcal{F})$. If $H$
is triangular we obtain a braided monoidal equivalence between
braided monoidal categories
$({}_H\mathcal{M},\otimes,c^\mathcal{R})$ and
$({}_{H_\mathcal{F}}\mathcal{M},\otimes_\mathcal{F},c^{\mathcal{R}_\mathcal{F}})$.
\end{theorem}
In the light of this theorem Drinfel'd twists are sometimes referred to as
\textit{gauge transformations} or \textit{gauge equivalences} (see e.g.
\cite{Ka95}~Section~XV.3). This nomenclature is affirmed by the observation that
$1\otimes 1\in H\otimes H$ is a Drinfel'd twist on any Hopf algebra $H$ and if
$\mathcal{F}$ is a twist on $H$ and $\mathcal{F}'$ a twist on $H_\mathcal{F}$,
the product $\mathcal{F}'\mathcal{F}$ is a Drinfel'd twist on $H$ such that
$H_{\mathcal{F}'\mathcal{F}}=(H_\mathcal{F})_{\mathcal{F}'}$.

\subsection{Equivariant Hopf Algebra Module Algebra Representations}

For some applications the monoidal equivalence
${}_H\mathcal{M}\cong{}_{H_\mathcal{F}}\mathcal{M}$ of
Theorem~\ref{thm01} is too arbitrary. Motivated from differential geometry
we want to study equivariant module algebra bimodules instead, which generalize
equivariant vector bundles. However, the restriction of the monoidal equivalence
to those bimodules fails to be braided in general. To fix this
we have to restrict ourselves to braided commutative algebras and equivariant braided
symmetric algebra bimodules. Nonetheless, this setting is rich enough to
allow for several interesting examples, e.g. the braided multivector fields
and differential forms of a braided commutative algebra, as we see in
Sections~\ref{Sec3.1}. 

Fix a Hopf algebra $(H,\mu,\eta,\Delta,\epsilon,S)$ and 
consider a left $H$-module $(\mathcal{A},\lambda)$ which is an algebra
with product $\mu_\mathcal{A}$ and unit $\eta_\mathcal{A}$ in addition.
It is said to be a \textit{left $H$-module algebra} if the module action
respects the algebra structure, i.e. if
$$
\lambda\circ(\mathrm{id}_H\otimes\mu_\mathcal{A})
=\mu_\mathcal{A}\circ(\lambda\otimes\lambda)\circ
(\mathrm{id}_H\otimes\tau_{H,\mathcal{A}}\otimes\mathrm{id}_\mathcal{A})
\circ(\Delta\otimes\mathrm{id}_{\mathcal{A}\otimes\mathcal{A}})
\colon H\otimes\mathcal{A}\otimes\mathcal{A}\rightarrow\mathcal{A}
$$
and $\lambda\circ(\mathrm{id}_H\otimes\eta_\mathcal{A})
=\eta_\mathcal{A}\circ\epsilon\colon H\rightarrow\mathcal{A}$
hold. In the following we often write $\mu_\mathcal{A}(a\otimes b)=a\cdot b$
for $a,b\in\mathcal{A}$ and $\xi\rhd a$ for the module action of $\xi\in H$
on $a\in\mathcal{A}$. The units of $\mathcal{A}$ and $H$ are sometimes denoted
by $1_\mathcal{A}$ and $1_H$, respectively or simply by $1$. In this notation
the module algebra axioms read
\begin{equation}
    \xi\rhd(a\cdot b)=(\xi_{(1)}\rhd a)\cdot(\xi_{(2)}\rhd b)
    \text{ and }
    \xi\rhd 1_\mathcal{A}=\epsilon(\xi)1_\mathcal{A}
\end{equation}
for all $\xi\in H$ and $a,b\in\mathcal{A}$.
A \textit{left $H$-module algebra
homomorphism} is a left $H$-module homomorphism between left $H$-module
algebras which is also an algebra homomorphism. The category of
left $H$-module algebras is denoted by ${}_H\mathcal{A}$.
\begin{lemma}[\cite{AsSh14}~Theorem~3.4]\label{lemma01}
Let $\mathcal{F}$ be a twist on $H$ and consider a left $H$-module
algebra $(\mathcal{A},\cdot,1_\mathcal{A})$.
Then $\mathcal{A}_\mathcal{F}=
(\mathcal{A},\cdot_\mathcal{F},1_\mathcal{A})$
is a left $H_\mathcal{F}$-module algebra with respect to the same
left $H$-module action, where
\begin{equation}
    a\cdot_\mathcal{F}b
    =(\mathcal{F}_1^{-1}\rhd a)
    \cdot(\mathcal{F}_2^{-1}\rhd b)
\end{equation}
for all $a,b\in\mathcal{A}$.
\end{lemma}
Fix a left $H$-module algebra $\mathcal{A}$ in the following
and consider the category ${}_\mathcal{A}\mathcal{M}$ of left
$\mathcal{A}$-modules. In order to compare it to the representation
theory of the deformed algebra $\mathcal{A}_\mathcal{F}$
we have to incorporate an additional action of the Hopf algebra $H$
on the modules. To obtain interesting results
this action has to respect the $\mathcal{A}$-module structure.
Accordingly we consider the subcategory ${}_\mathcal{A}^H\mathcal{M}$ of
\textit{$H$-equivariant left $\mathcal{A}$-modules}. The objects of
${}_\mathcal{A}^H\mathcal{M}$ are left $H$-modules $\mathcal{M}$,
which are left $\mathcal{A}$-modules in addition such that
\begin{equation}
    \xi\rhd(a\cdot m)
    =(\xi_{(1)}\rhd a)\cdot(\xi_{(2)}\rhd m)
\end{equation}
for all $\xi\in H$, $a\in\mathcal{A}$ and $m\in\mathcal{M}$.
Morphisms are left $H$-module homomorphisms between
$H$-equivariant left $\mathcal{A}$-modules which are also
left $\mathcal{A}$-module homomorphisms.
\begin{lemma}
Let $\mathcal{F}$ be a twist on $H$ and $\mathcal{A}$ a left $H$-module
algebra. Then there is a functor
\begin{equation}
    \mathrm{Drin}_\mathcal{F}\colon
    {}_\mathcal{A}^H\mathcal{M}
    \rightarrow
    {}_{\mathcal{A}_\mathcal{F}}^{H_\mathcal{F}}\mathcal{M},
\end{equation}
called Drinfel'd functor,
which is the identity on morphisms and assigns to every
$H$-equivariant left $\mathcal{A}$-module $\mathcal{M}$
the same left $H$-module but with left $\mathcal{A}_\mathcal{F}$-module
structure given by
\begin{equation}
    a\cdot_\mathcal{F}m
    =(\mathcal{F}_1^{-1}\rhd a)\cdot(\mathcal{F}_2^{-1}\rhd m)
\end{equation}
for all $a\in\mathcal{A}$ and $m\in\mathcal{M}$.
\end{lemma}
\begin{proof}
In fact, the obtained
$\Bbbk$-module $\mathcal{M}_\mathcal{F}$ is an object in 
${}_{\mathcal{A}_\mathcal{F}}^{H_\mathcal{F}}\mathcal{M}$, since
\begin{align*}
    (a\cdot_\mathcal{F}b)\cdot_\mathcal{F}m
    =a\cdot_\mathcal{F}(b\cdot_\mathcal{F}m)
    \text{ and }
    \xi\rhd(a\cdot_\mathcal{F}m)
    =(\xi_{\widehat{(1)}}\rhd a)
    \cdot_\mathcal{F}(\xi_{\widehat{(2)}}\rhd m)
\end{align*}
follow for all $\xi\in H$, $a,b\in\mathcal{A}$ and $m\in\mathcal{M}$
in complete analogy to Lemma~\ref{lemma01}. Furthermore, any
morphisms $\phi\colon\mathcal{M}\rightarrow\mathcal{M}'$ in
${}_\mathcal{A}^H\mathcal{M}$ is automatically a morphism in
${}_{\mathcal{A}_\mathcal{F}}^{H_\mathcal{F}}\mathcal{M}$, where
left $H_\mathcal{F}$-linearity is trivially given and left 
$\mathcal{A}_\mathcal{F}$-linearity
follows since
\begin{align*}
    \phi(a\cdot_\mathcal{F}m)
    =&\phi((\mathcal{F}_1^{-1}\rhd a)\cdot(\mathcal{F}_2^{-1}\rhd m))
    =(\mathcal{F}_1^{-1}\rhd a)\cdot\phi(\mathcal{F}_2^{-1}\rhd m)\\
    =&(\mathcal{F}_1^{-1}\rhd a)\cdot(\mathcal{F}_2^{-1}\rhd\phi(m))
    =a\cdot_\mathcal{F}\phi(b)
\end{align*}
for all $a\in\mathcal{A}$ and $m\in\mathcal{M}$.
\end{proof}
One might ask if the monoidal equivalence of Theorem~\ref{thm01} 
restricts to ${}_\mathcal{A}^H\mathcal{M}$.
However, ${}_\mathcal{A}^H\mathcal{M}$ is not monoidal with respect to
the usual tensor product of $\Bbbk$-modules, since there is no coproduct
on $\mathcal{A}$ in general to distribute the left
$\mathcal{A}$-module action to the tensor factors. To obtain a monoidal
category we need two specifications: first we consider the
subcategory of $H$-equivariant $\mathcal{A}$-bimodules
${}_\mathcal{A}^H\mathcal{M}_\mathcal{A}$, i.e. there are commuting
left and right $\mathcal{A}$-actions which are equivariant with respect to
the left $H$-action. Secondly, we consider the tensor product
$\otimes_\mathcal{A}$ over $\mathcal{A}$, which is defined for
two objects $\mathcal{M}$ and $\mathcal{M}'$ by the quotient
$$
\mathcal{M}\otimes\mathcal{M}'/N_{\mathcal{M},\mathcal{M}'},
$$
where $N_{\mathcal{M},\mathcal{M}}
=\mathrm{im}(\rho_\mathcal{M}\otimes\mathrm{id}_{\mathcal{M}'}
-\mathrm{id}_\mathcal{M}\otimes\lambda_{\mathcal{M}'})$ and
$\lambda_{\mathcal{M}'}$ and $\rho_\mathcal{M}$ denote the left and
right $\mathcal{A}$-actions on $\mathcal{M}'$ and $\mathcal{M}$, respectively.
As a consequence one has
\begin{equation}
    (m\cdot a)\otimes_\mathcal{A}m'
    =m\otimes_\mathcal{A}(a\cdot m')
\end{equation}
for all $a\in\mathcal{A}$, $m\in\mathcal{M}$ and $m'\in\mathcal{M}'$. Then
$\mathcal{M}\otimes_\mathcal{A}\mathcal{M}'$ is an $H$-equivariant
$\mathcal{A}$-bimodule, with induced left $H$-action and left and right
$\mathcal{A}$-actions given by
\begin{equation}
    a\cdot(m\otimes_\mathcal{A}m')
    =(a\cdot m)\otimes_\mathcal{A}m'
    \text{ and }
    (m\otimes_\mathcal{A}m')\cdot a
    =m\otimes_\mathcal{A}(m'\cdot a)
\end{equation}
for all $a\in\mathcal{A}$, $m\in\mathcal{M}$ and $m'\in\mathcal{M}'$.
On morphisms $\phi\colon\mathcal{M}\rightarrow\mathcal{N}$ and 
$\psi\colon\mathcal{M}'\rightarrow\mathcal{N}'$ of
${}_\mathcal{A}^H\mathcal{M}_\mathcal{A}$ one defines 
$(\phi\otimes_\mathcal{A}\psi)(m\otimes_\mathcal{A}m')
=\phi(m)\otimes_\mathcal{A}\psi(m')$ for all
$m\in\mathcal{M}$ and $m'\in\mathcal{M}'$.
\begin{proposition}
The tuple
$({}_\mathcal{A}^H\mathcal{M}_\mathcal{A},\otimes_\mathcal{A})$
is a monoidal category and for a twist $\mathcal{F}$ on $H$
the monoidal equivalence of Theorem~\ref{thm01} descends to a monoidal
equivalence of
$({}_\mathcal{A}^H\mathcal{M}_\mathcal{A},\otimes_\mathcal{A})$
and
$({}_{\mathcal{A}_\mathcal{F}}^{H_\mathcal{F}}
\mathcal{M}_{\mathcal{A}_\mathcal{F}},\otimes_{\mathcal{A}_\mathcal{F}})$.
\end{proposition}
We refer to \cite{Schenkel2015}~Theorem~3.13 for a proof and more information.
In contrast to Theorem~\ref{thm01} we do not obtain a symmetric braided 
monoidal structure on ${}_\mathcal{A}^H\mathcal{M}_\mathcal{A}$
if $H$ is triangular in general. The $H$-equivariant $\mathcal{A}$-bimodules
are still too arbitrary. One has to demand more symmetry before. We do so
by considering a \textit{braided commutative} left $H$-module algebra
$\mathcal{A}$ for a triangular Hopf algebra $(H,\mathcal{R})$ instead of a
general left $\mathcal{A}$-module algebra. This means that
$b\cdot a=(\mathcal{R}_1^{-1}\rhd a)\cdot(\mathcal{R}_2^{-1}\rhd b)$ holds
for all $a,b\in\mathcal{A}$. On the level of $\mathcal{A}$-bimodules
we want to keep this symmetry: an \textit{$H$-equivariant braided symmetric
$\mathcal{A}$-bimodule} $\mathcal{M}$ for a braided commutative
left $H$-module algebra $\mathcal{A}$ is an $H$-equivariant 
$\mathcal{A}$-bimodule such that
$m\cdot a=(\mathcal{R}_1^{-1}\rhd a)\cdot(\mathcal{R}_2^{-1}\rhd m)$
for all $a\in\mathcal{A}$ and $m\in\mathcal{M}$. In other words,
the left and right $\mathcal{A}$-actions are related via the universal
$\mathcal{R}$-matrix, mirroring the braided commutativity of $\mathcal{A}$.
These bimodules form a category
${}_\mathcal{A}^H\mathcal{M}_\mathcal{A}^\mathcal{R}$ with
morphisms being the usual left $H$-linear and left and right $\mathcal{A}$-linear
maps. A proof of the following statement can be found in
\cite{Schenkel2015}~Theorem~5.21.
\begin{theorem}\label{thm02}
If $H$ is triangular and $\mathcal{A}$ is braided commutative we obtain a
braided monoidal equivalence
\begin{equation}
    ({}_\mathcal{A}^H\mathcal{M}_\mathcal{A}^\mathcal{R},
    \otimes_\mathcal{A},c^\mathcal{R})
    \cong({}_{\mathcal{A}_\mathcal{F}}^{H_\mathcal{F}}
    \mathcal{M}_{\mathcal{A}_\mathcal{F}}^{\mathcal{R}_\mathcal{F}},
    \otimes_{\mathcal{A}_\mathcal{F}},c^{\mathcal{R}_\mathcal{F}})
\end{equation}
between braided monoidal categories.
\end{theorem}

\section{Braided Commutative Geometry}\label{section3}

We enter the main section of these notes with the aim to construct a
noncommutative Cartan calculus for any braided commutative algebra. Since its
development is entirely parallel to the classical Cartan calculus on a
commutative algebra, with basically no choices on the way, it feels justified
to call it \textit{the} braided Cartan calculus on a fixed braided commutative
algebra. Before proving this result we recall the notion of multivector fields
and differential forms on a commutative algebra, also to indicate the naturality
of the generalization. The corresponding Graßmann and Gerstenhaber structures
are equivariant with respect to a cocommutative Hopf algebra if the commutative
algebra is a Hopf algebra module algebra in addition. More in general we give 
the definitions of braided Graßmann and Gerstenhaber algebra and provide braided
multivector fields and differential forms on a braided commutative algebra as 
examples. In the second subsection we introduce a differential on
braided differential forms via a braided version of the Chevalley-Eilenberg
formula. Remark that the differential is a graded braided derivation with respect to
the braided wedge product, however, since it is equivariant, it resembles a graded
(non-braided) derivation. Using graded braided
commutators the relations between the braided Lie derivative, insertion 
and differential are generalizing and entirely mirror the commutation
relations of the classical Cartan calculus. We end the second subsection by
applying the gauge equivalence given by the Drinfel'd functor to the braided
Cartan calculus and proving that the result is isomorphic to the braided Cartan
calculus on the twisted algebra with respect to the twisted triangular structure.
Some ramifications of this gauge equivalence, in particular for
the interpretation of the twisted Cartan calculus on a commutative algebra, are
discussed. As an application of the braided Cartan calculus
the third and last subsection deals with equivariant covariant 
derivatives and metrics.
The main results are the extension of an equivariant covariant derivative to braided
multivector fields and differential forms and the existence of a unique equivariant
Levi-Civita covariant derivative for a fixed non-degenerate equivariant metric.
We prove that the Drinfel'd functor is compatible with all constructions.

\subsection{Braided Graßmann and Gerstenhaber Algebras}\label{Sec3.1}

For the Cartan calculus on a commutative algebra $\mathcal{A}$
the two most important $\mathcal{A}$-bimodules are the
multivector fields $\mathfrak{X}^\bullet(\mathcal{A})$ and differential forms
$\Omega^\bullet(\mathcal{A})$. They are graded and possess a 
Graßmann structure. If $\mathcal{A}$ is a left $H$-module algebra
for a cocommutative Hopf algebra $H$, $\mathfrak{X}^\bullet(\mathcal{A})$
and $\Omega^\bullet(\mathcal{A})$ are $H$-equivariant symmetric
$\mathcal{A}$-bimodules
and the module actions respect the grading. Let us briefly recall the construction
of those modules and then generalize them to the category
${}_\mathcal{A}^H\mathcal{M}_\mathcal{A}^\mathcal{R}$ for a triangular Hopf
algebra $(H,\mathcal{R})$ and a braided commutative left $H$-module
algebra $\mathcal{A}$.

Fix a cocommutative Hopf algebra $H$ and a commutative left $H$-module
algebra $\mathcal{A}$ for the moment. The derivations
$\mathrm{Der}(\mathcal{A})$ of $\mathcal{A}$ are an $H$-equivariant symmetric
$\mathcal{A}$-bimodule with left $H$-action given by the \textit{adjoint action}
\begin{equation}
    (\xi\rhd X)(a)
    =\xi_{(1)}\rhd(X(S(\xi_{(2)})\rhd a))
\end{equation}
and left and right $\mathcal{A}$-module actions
$(a\cdot X)(b)=a\cdot X(b)=(X\cdot a)(b)$,
for all $\xi\in H$, $X\in\mathrm{Der}(\mathcal{A})$ and $a\in\mathcal{A}$.
In particular, the tensor algebra 
$$
\mathrm{T}^\bullet\mathrm{Der}(\mathcal{A})
=\mathcal{A}\oplus\mathrm{Der}(\mathcal{A})
\oplus(\mathrm{Der}(\mathcal{A})\otimes_\mathcal{A}\mathrm{Der}(\mathcal{A}))
\oplus\cdots
$$
of $\mathrm{Der}(\mathcal{A})$ with respect to the tensor product
$\otimes_\mathcal{A}$ over $\mathcal{A}$ is well-defined. It is an
$H$-equivariant symmetric $\mathcal{A}$-bimodule with module actions
defined on factorizing elements $X_1\otimes_\mathcal{A}\cdots
\otimes_\mathcal{A}X_k\in\mathrm{T}^k\mathrm{Der}(\mathcal{A})$ by
\begin{equation}
\begin{split}
    \xi\rhd(X_1\otimes_\mathcal{A}\cdots\otimes_\mathcal{A}X_k)
    =&(\xi_{(1)}\rhd X_1)\otimes_\mathcal{A}\cdots
    \otimes_\mathcal{A}(\xi_{(k)}\rhd X_k),\\
    a\cdot(X_1\otimes_\mathcal{A}\cdots\otimes_\mathcal{A}X_k)
    =&(a\cdot X_1)\otimes_\mathcal{A}\cdots\otimes_\mathcal{A}X_k,\\
    (X_1\otimes_\mathcal{A}\cdots\otimes_\mathcal{A}X_k)\cdot a
    =&X_1\otimes_\mathcal{A}\cdots\otimes_\mathcal{A}(X_k\cdot a)
\end{split}
\end{equation}
for all $\xi\in H$ and $a\in\mathcal{A}$. Furthermore, there is an ideal
$I$ in $\mathrm{T}^\bullet\mathrm{Der}(\mathcal{A})$ generated by elements
$X_1\otimes_\mathcal{A}\cdots\otimes_\mathcal{A}X_k\in
\mathrm{T}^k\mathrm{Der}(\mathcal{A})$ such that $X_i=X_j$ for a pair
$(i,j)$ such that $1\leq i<j\leq k$. The quotient
$\mathrm{T}^\bullet\mathrm{Der}(\mathcal{A})/I$ is the exterior algebra.
It is the Graßmann algebra $\mathfrak{X}^\bullet(\mathcal{A})$
of multivector fields on $\mathcal{A}$ and the induced product,
the wedge product, is denoted by $\wedge$.
Since $H$ is cocommutative and the $\mathcal{A}$-actions symmetric,
they respect the ideal $I$. Consequently, the induced actions on 
$\mathfrak{X}^\bullet(\mathcal{A})$ are well-defined, structuring the
multivector fields as an $H$-equivariant symmetric $\mathcal{A}$-bimodule
with the additional property that the module actions respect the grading.
Moreover, the usual
commutator of endomorphisms $[\cdot,\cdot]$ is a Lie bracket for the derivations
of $\mathcal{A}$. It extends uniquely to a Gerstenhaber bracket
$\llbracket\cdot,\cdot\rrbracket$ on $\mathfrak{X}^\bullet(\mathcal{A})$
by defining $\llbracket a,b\rrbracket=0$,
$\llbracket X,a\rrbracket=X(a)$ for all
$a,b\in\mathcal{A}$, $X\in\mathrm{Der}(\mathcal{A})$ and inductively declaring
the \textit{graded Leibniz rule}
\begin{equation}\label{eq01}
    \llbracket X,Y\wedge Z\rrbracket
    =\llbracket X,Y\rrbracket\wedge Z
    +(-1)^{(k-1)\ell}Y\wedge\llbracket X,Z\rrbracket
\end{equation}
for all $X\in\mathfrak{X}^k(\mathcal{A})$, $Y\in\mathfrak{X}^\ell(\mathcal{A})$
and $Z\in\mathfrak{X}^\bullet(\mathcal{A})$. In detail this means that
$\llbracket\cdot,\cdot\rrbracket\colon
\mathfrak{X}^k(\mathcal{A})\times\mathfrak{X}^\ell(\mathcal{A})
\rightarrow\mathfrak{X}^{k+\ell-1}(\mathcal{A})$ is a graded (with respect to
the degree shifted by $-1$) Lie bracket, i.e. it is \textit{graded skew-symmetric}
\begin{equation}
    \llbracket Y,X\rrbracket
    =-(-1)^{(k-1)(\ell-1)}\llbracket X,Y\rrbracket
\end{equation}
and satisfies the \textit{graded Jacobi identity}
\begin{equation}
    \llbracket X,\llbracket Y,Z\rrbracket\rrbracket
    =\llbracket\llbracket X,Y\rrbracket,Z\rrbracket
    +(-1)^{(k-1)(\ell-1)}\llbracket Y,\llbracket X,Z\rrbracket\rrbracket,
\end{equation}
where $X\in\mathfrak{X}^k(\mathcal{A})$, $Y\in\mathfrak{X}^\ell(\mathcal{A})$
and $Z\in\mathfrak{X}^\bullet(\mathcal{A})$,
such that the graded Leibniz rule (\ref{eq01}) holds in addition. 
Using the formula
\begin{equation}\label{eq04}
\begin{split}
    \llbracket X_1\wedge\cdots\wedge X_k,
    Y_1\wedge\cdots\wedge Y_\ell\rrbracket
    =&\sum_{i=1}^k\sum_{j=1}^\ell(-1)^{i+j}[X_i,Y_j]\wedge
    X_1\wedge\cdots\wedge\widehat{X_i}\wedge\cdots\wedge X_k\\
    &\wedge Y_1\wedge\cdots\wedge\widehat{Y_j}\wedge\cdots\wedge Y_\ell,
\end{split}
\end{equation}
which holds for all $X_1,\ldots,X_k,Y_1,\ldots,Y_\ell\in\mathfrak{X}^1(\mathcal{A})$,
it is easy to prove that the Gerstenhaber bracket $\llbracket\cdot,\cdot\rrbracket$ is
$H$-equivariant, i.e. that
\begin{equation}
    \xi\rhd\llbracket X,Y\rrbracket
    =\llbracket\xi_{(1)}\rhd X,\xi_{(2)}\rhd Y\rrbracket
\end{equation}
for all $\xi\in H$ and $X,Y\in\mathfrak{X}^\bullet(\mathcal{A})$.
Note that $\widehat{X_i}$ and $\widehat{Y_j}$ means that
$X_i$ and $Y_j$ are left out in the wedge product of eq.(\ref{eq04}).

Differential forms on $\mathcal{A}$ are defined in the following way:
consider $\mathrm{Hom}_\mathcal{A}(\mathrm{Der}(\mathcal{A})
,\mathcal{A})$, the $\Bbbk$-module of $\Bbbk$-linear and $\mathcal{A}$-linear maps
$\mathrm{Der}(\mathcal{A})\rightarrow\mathcal{A}$. It is an $H$-equivariant
symmetric $\mathcal{A}$-bimodule with respect to the adjoint $H$-action and
$(a\cdot\omega)(X)=a\cdot\omega(X)=(\omega\cdot a)(X)$
for all $a\in\mathcal{A}$,
$\omega\in\mathrm{Hom}_\mathcal{A}(\mathrm{Der}(\mathcal{A}),\mathcal{A})$
and $X\in\mathrm{Der}(\mathcal{A})$. The corresponding exterior algebra
is denoted by $\underline{\Omega}^\bullet(\mathcal{A})$.
One can define
a differential $\mathrm{d}$ of $\omega\in\underline{\Omega}^k(\mathcal{A})$ 
via
\begin{equation}
\begin{split}
    (\mathrm{d}\omega)&(X_1,\ldots,X_{k+1})
    =\sum_{i=1}^{k+1}(-1)^{i+1}
    X_i(\omega(X_1,\ldots,\widehat{X_i},\ldots,X_{k+1}))\\
    &+\sum_{i<j}(-1)^{i+j}
    \omega([X_i,X_j],X_1,\ldots,\widehat{X_i},\ldots,
    \widehat{X_j},\ldots,X_{k+1})
\end{split}
\end{equation}
for all $X_1,\ldots,X_{k+1}\in\mathrm{Der}(\mathcal{A})$.
This is known as the \textit{Chevalley-Eilenberg formula}.
Define now the differential forms
$\Omega^\bullet(\mathcal{A})$ on $\mathcal{A}$ to be the smallest
differential graded subalgebra of $\underline{\Omega}^\bullet(\mathcal{A})$
such that $\mathcal{A}\subseteq\Omega^\bullet(\mathcal{A})$ (compare to
\cite{D-VM96,D-VM94}).
In this case every element of $\Omega^k(\mathcal{A})$ can be written
as a finite sum of elements of the form 
$a_0\mathrm{d}a_1\wedge\ldots\wedge\mathrm{d}a_k$, where
$a_0,\ldots,a_k\in\mathcal{A}$.
The induced actions structure $(\Omega^\bullet(\mathcal{A}),\wedge)$
as an $H$-equivariant symmetric $\mathcal{A}$-bimodule and a Graßmann
algebra such that $\wedge$ is equivariant and $H\rhd\Omega^k(\mathcal{A})
\subseteq\Omega^k(\mathcal{A})$.
From the Chevalley-Eilenberg formula it follows that
$\mathrm{d}$ commutes with $\rhd$.
The insertion $\mathrm{i}\colon\mathfrak{X}^1(\mathcal{A})\otimes
\Omega^k(\mathcal{A})\rightarrow\Omega^{k-1}(\mathcal{A})$
of derivations $X\in\mathfrak{X}^1(\mathcal{A})$ into the first slot of a
differential form, i.e. $(\mathrm{i}_X(\omega))(X_1,\ldots,X_{k-1})
=\omega(X,X_1,\ldots,X_{k-1})$ for all
$\omega\in\Omega^k(\mathcal{A})$ and
$X_1,\ldots,X_{k-1}\in\mathfrak{X}^1(\mathcal{A})$ is $H$-equivariant.

We are ready to generalize the concepts of Graßmann and Gerstenhaber algebra
to the setting of equivariant braided symmetric bimodules. This is exemplified
by the example of braided multivector fields.
Fix a triangular Hopf algebra $(H,\mathcal{R})$ and a braided commutative
left $H$-module algebra $\mathcal{A}$.
A $\Bbbk$-linear endomorphism $X$ of $\mathcal{A}$
is said to be a \textit{braided derivation} if
\begin{equation}
    X(ab)=X(a)b+(\mathcal{R}_1^{-1}\rhd a)((\mathcal{R}_2^{-1}\rhd X)(b))
\end{equation}
for all $a,b\in\mathcal{A}$, where the left $H$-action on 
endomorphisms is given by the adjoint action.
\begin{lemma}
The braided derivations $\mathrm{Der}_\mathcal{R}(\mathcal{A})$ are an
$H$-equivariant braided symmetric $\mathcal{A}$-bimodule. Furthermore,
the braided commutator
\begin{equation}
    [X,Y]_\mathcal{R}=XY-(\mathcal{R}_1^{-1}\rhd Y)(\mathcal{R}_2^{-1}\rhd X),
\end{equation}
where $X,Y\in\mathrm{Der}_\mathcal{R}(\mathcal{A})$, structures
$\mathrm{Der}_\mathcal{R}(\mathcal{A})$ as a braided Lie algebra.
The latter means that $[\cdot,\cdot]_\mathcal{R}$ is braided
skew-symmetric, i.e.
$[Y,X]_\mathcal{R}
=-[\mathcal{R}_1^{-1}\rhd X,\mathcal{R}_2^{-1}\rhd Y]_\mathcal{R}$
and satisfies the braided Jacobi identity, i.e.
\begin{equation}
    [X,[Y,Z]_\mathcal{R}]_\mathcal{R}
    =[[X,Y]_\mathcal{R},Z]_\mathcal{R}
    +[\mathcal{R}_1^{-1}\rhd Y,
    [\mathcal{R}_2^{-1}\rhd X,Z]_\mathcal{R}]_\mathcal{R}
\end{equation}
for all $X,Y,Z\in\mathrm{Der}_\mathcal{R}(\mathcal{A})$.
\end{lemma}
This is an elementary consequence of the properties of the triangular structure.
In a next step we want to generalize the construction of multivector fields
of a commutative algebra (compare also to \cite{BZ2008}).
Since $\mathrm{Der}_\mathcal{R}(\mathcal{A})$ is an $\mathcal{A}$-bimodule
we can build the tensor algebra $\mathrm{T}^\bullet
\mathrm{Der}_\mathcal{R}(\mathcal{A})$ with respect to $\otimes_\mathcal{A}$
and with module actions on factorizing elements
$X_1\otimes_\mathcal{A}\cdots\otimes_\mathcal{A}X_k\in
\mathrm{T}^k\mathrm{Der}_\mathcal{R}(\mathcal{A})$ defined by
\begin{equation}\label{eq05}
\begin{split}
    \xi\rhd(X_1\otimes_\mathcal{A}\cdots\otimes_\mathcal{A}X_k)
    =&(\xi_{(1)}\rhd X_1)\otimes_\mathcal{A}\cdots
    \otimes_\mathcal{A}(\xi_{(k)}\rhd X_k),\\
    a\cdot(X_1\otimes_\mathcal{A}\cdots\otimes_\mathcal{A}X_k)
    =&(a\cdot X_1)\otimes_\mathcal{A}\cdots\otimes_\mathcal{A}X_k,\\
    (X_1\otimes_\mathcal{A}\cdots\otimes_\mathcal{A}X_k)\cdot a
    =&X_1\otimes_\mathcal{A}\cdots\otimes_\mathcal{A}(X_k\cdot a)
\end{split}
\end{equation}
for all $\xi\in H$ and $a\in\mathcal{A}$.
There is an ideal $I$ in 
$\mathrm{T}^\bullet\mathrm{Der}_\mathcal{R}(\mathcal{A})$
generated by elements
$X_1\otimes_\mathcal{A}\cdots\otimes_\mathcal{A}X_k\in
\mathrm{T}^k\mathrm{Der}_\mathcal{R}(\mathcal{A})$ which equal
\begin{equation}
\begin{split}
    X_1\otimes_\mathcal{A}&\cdots\otimes_\mathcal{A}X_{i-1}
    \otimes_\mathcal{A}\bigg(\mathcal{R}_1^{'-1}\rhd\bigg(
    (\mathcal{R}_1^{-1}\rhd X_j)\otimes_\mathcal{A}
    (\mathcal{R}_2^{-1}\rhd(X_{i+1}\otimes_\mathcal{A}
    \cdots\otimes_\mathcal{A}X_{j-1}))\bigg)\bigg)\\
    &\otimes_\mathcal{A}(\mathcal{R}_2^{'-1}\rhd X_i)
    \otimes_\mathcal{A}X_{j+1}\otimes_\mathcal{A}\cdots
    \otimes_\mathcal{A}X_k
\end{split}
\end{equation}
for a pair $(i,j)$ such that $1\leq i<j\leq k$. One can prove that
the module actions (\ref{eq05}) respect $I$ (see \cite{ThomasPhDThesis}).
This induces an $H$-equivariant graded
associative braided commutative product $\wedge_\mathcal{R}$ on the
quotient, declaring the \textit{braided multivector fields}
$(\mathfrak{X}^\bullet_\mathcal{R}(\mathcal{A}),\wedge_\mathcal{R})$
on $\mathcal{A}$. In general, the associative unital graded algebra and
$H$-equivariant braided symmetric $\mathcal{A}$-bimodule
$(\Lambda^\bullet\mathcal{M},\wedge_\mathcal{R})$
associated to an $H$-equivariant braided symmetric $\mathcal{A}$-bimodule
$\mathcal{M}$ in this way is said to be the \textit{braided Graßmann algebra}
or \textit{braided exterior algebra} corresponding to $\mathcal{M}$.
Coming back to the example of braided multivector fields we can use
the braided commutator of vector fields to obtain additional structure
on the braided Graßmann algebra. Namely,
we are defining a $\Bbbk$-bilinear operation
$\llbracket\cdot,\cdot\rrbracket_\mathcal{R}\colon
\mathfrak{X}^k_\mathcal{R}(\mathcal{A})\times
\mathfrak{X}^\ell_\mathcal{R}(\mathcal{A})
\rightarrow\mathfrak{X}^{k+\ell-1}_\mathcal{R}(\mathcal{A})$
in the following way. If $a,b\in\mathcal{A}$ we set 
$\llbracket a,b\rrbracket_\mathcal{R}=0$. For $a\in\mathcal{A}$ and a
factorizing element
$X=X_1\wedge_\mathcal{R}\cdots\wedge_\mathcal{R}X_k
\in\mathfrak{X}^k_\mathcal{R}(\mathcal{A})$ where $k>0$ we define
\begin{equation}
\begin{split}
    \llbracket X,a\rrbracket_\mathcal{R}
    =&\sum_{i=1}^k(-1)^{k-i}
    X_1\wedge_\mathcal{R}\cdots\wedge_\mathcal{R}X_{i-1}
    \wedge_\mathcal{R}(X_i(\mathcal{R}_1^{-1}\rhd a))\\
    &\wedge_\mathcal{R}\bigg(\mathcal{R}_2^{-1}\rhd\bigg(
    X_{i+1}\wedge_\mathcal{R}\cdots\wedge_\mathcal{R}X_k
    \bigg)\bigg)
\end{split}
\end{equation}
and
\begin{equation}
\begin{split}
    \llbracket a,X\rrbracket_\mathcal{R}
    =&\sum_{i=1}^k(-1)^i\bigg(\mathcal{R}_{1(1)}^{-1}\rhd\bigg(
    X_1\wedge_\mathcal{R}\cdots\wedge_\mathcal{R}X_{i-1}\bigg)\bigg)
    \wedge_\mathcal{R}((\mathcal{R}_{1(2)}^{-1}\rhd X_i)
    (\mathcal{R}_2^{-1}\rhd a))\\
    &\wedge_\mathcal{R}
    X_{i+1}\wedge_\mathcal{R}\cdots\wedge_\mathcal{R}X_k.
\end{split}
\end{equation}
Furthermore, on factorizing elements
$X=X_1\wedge_\mathcal{R}\cdots\wedge_\mathcal{R}X_k
\in\mathfrak{X}^k_\mathcal{R}(\mathcal{A})$ and
$Y=Y_1\wedge_\mathcal{R}\cdots\wedge_\mathcal{R}Y_\ell
\in\mathfrak{X}^\ell_\mathcal{R}(\mathcal{A})$, where
$k,\ell>0$, we define
\begin{equation}
\begin{split}
    \llbracket X&,Y\rrbracket_\mathcal{R}
    =\sum_{i=1}^k\sum_{j=1}^\ell(-1)^{i+j}
    [\mathcal{R}_1^{-1}\rhd X_i,\mathcal{R}_1^{'-1}\rhd Y_j]_\mathcal{R}\\
    &\wedge_\mathcal{R}
    \bigg(\mathcal{R}_2^{'-1}\rhd\bigg(
    \bigg(\mathcal{R}_2^{-1}\rhd(X_1\wedge_\mathcal{R}\cdots
    \wedge_\mathcal{R}X_{i-1})\bigg)
    \wedge_\mathcal{R}\widehat{X_i}\wedge_\mathcal{R}X_{i+1}
    \wedge_\mathcal{R}\cdots\wedge_\mathcal{R}X_k\\
    &\wedge_\mathcal{R}Y_1\wedge_\mathcal{R}\cdots\wedge_\mathcal{R}Y_{j-1}\bigg)\bigg)
    \wedge_\mathcal{R}\widehat{Y_j}\wedge_\mathcal{R}Y_{j+1}
    \wedge_\mathcal{R}\cdots\wedge_\mathcal{R}Y_\ell,
\end{split}
\end{equation}
where $[\cdot,\cdot]_\mathcal{R}$ denotes the braided commutator
and $\widehat{X_i}$ and $\widehat{Y_j}$ means that $X_i$ and
$Y_j$ are omitted in above product.
The operation $\llbracket\cdot,\cdot\rrbracket_\mathcal{R}$
is said to be the \textit{braided Schouten-Nijenhuis bracket}.
\begin{proposition}
The braided multivector fields $(\mathfrak{X}^\bullet_\mathcal{R}(\mathcal{A}),
\wedge_\mathcal{R},\llbracket\cdot,\cdot\rrbracket_\mathcal{R})$ on $\mathcal{A}$
are an associative unital graded algebra and an $H$-equivariant braided symmetric
$\mathcal{A}$-bimodule equipped with an $H$-equivariant graded (with degree
shifted by $-1$) braided Lie bracket
$\llbracket\cdot,\cdot\rrbracket_\mathcal{R}\colon
\mathfrak{X}^k_\mathcal{R}(\mathcal{A})
\otimes\mathfrak{X}^\ell_\mathcal{R}(\mathcal{A})
\rightarrow\mathfrak{X}^{k+\ell-1}_\mathcal{R}(\mathcal{A})$, which means that
$\llbracket\cdot,\cdot\rrbracket_\mathcal{R}$ is graded braided skewsymmetric, i.e.
\begin{equation}
    \llbracket Y,X\rrbracket_\mathcal{R}
    =-(-1)^{(k-1)\cdot(\ell-1)}\llbracket\mathcal{R}_1^{-1}\rhd X,
    \mathcal{R}_2^{-1}\rhd Y\rrbracket_\mathcal{R},
\end{equation}
and satisfies the graded braided Jacobi identity
\begin{equation}
    \llbracket X,\llbracket Y,Z\rrbracket_\mathcal{R}\rrbracket_\mathcal{R}
    =\llbracket\llbracket X,Y\rrbracket_\mathcal{R},Z\rrbracket_\mathcal{R}
    +(-1)^{(k-1)\cdot(\ell-1)}\llbracket\mathcal{R}_1^{-1}\rhd Y,
    \llbracket\mathcal{R}_2^{-1}\rhd X,Z\rrbracket_\mathcal{R}
    \rrbracket_\mathcal{R},
\end{equation}
such that the graded braided Leibniz rule
\begin{equation}
    \llbracket X,Y\wedge_\mathcal{R}Z\rrbracket_\mathcal{R}
    =\llbracket X,Y\rrbracket_\mathcal{R}\wedge_\mathcal{R}Z
    +(-1)^{(k-1)\cdot\ell}(\mathcal{R}_1^{-1}\rhd Y)\wedge_\mathcal{R}
    \llbracket\mathcal{R}_2^{-1}\rhd X,Z\rrbracket_\mathcal{R}
\end{equation}
holds in addition,
where $X\in\mathfrak{X}^k_\mathcal{R}(\mathcal{A})$,
$Y\in\mathfrak{X}^\ell_\mathcal{R}(\mathcal{A})$
and
$Z\in\mathfrak{X}^\bullet_\mathcal{R}(\mathcal{A})$.
\end{proposition}
More in general we make the following definition.
\begin{definition}[Braided Gerstenhaber algebra]
An associative unital graded algebra and $H$-equivariant braided symmetric
$\mathcal{A}$-bimodule $(\mathfrak{G}^\bullet,\wedge_\mathcal{R})$
is said to be a braided Gerstenhaber algebra if the module actions respect
the degree and if there is an $H$-equivariant
graded (with degree shifted by $-1$) braided Lie bracket
satisfying a graded braided Leibniz rule
with respect to $\wedge_\mathcal{R}$.
\end{definition}
Let $\mathfrak{G}^\bullet$ be a braided Gerstenhaber algebra.
It follows that $\mathfrak{G}^0$ is a braided commutative
left $H$-module algebra and $\mathfrak{G}^1$ is a braided Lie algebra.
Moreover, $\mathfrak{G}^1$ is an $H$-equivariant
braided symmetric $\mathfrak{G}^0$-bimodule and
$\mathfrak{G}^k$ is an $H$-equivariant braided symmetric
$\mathfrak{G}^1$-bimodule.
This means that for any $X\in\mathfrak{G}^1$ we can define the 
\textit{braided Lie derivative}
$\mathscr{L}_X^\mathcal{R}=\llbracket X,\cdot\rrbracket_\mathcal{R}
\colon\mathfrak{G}^k\rightarrow\mathfrak{G}^k$
which is a braided derivation, i.e.
\begin{equation}
    \mathscr{L}_X^\mathcal{R}(Y\wedge_\mathcal{R}Z)
    =\mathscr{L}_X^\mathcal{R}Y\wedge_\mathcal{R}Z
    +(\mathcal{R}_1^{-1}\rhd Y)\wedge_\mathcal{R}
    (\mathcal{R}_2^{-1}\rhd\mathscr{L}_X^\mathcal{R})Z
\end{equation}
for all $X\in\mathfrak{G}^1$ and $Y,Z\in\mathfrak{G}^\bullet$.
It furthermore satisfies
$\mathscr{L}^\mathcal{R}_{[X,Y]_\mathcal{R}}
=\mathscr{L}^\mathcal{R}_X\mathscr{L}^\mathcal{R}_Y
-\mathscr{L}^\mathcal{R}_{\mathcal{R}_1^{-1}\rhd Y}
\mathscr{L}^\mathcal{R}_{\mathcal{R}_2^{-1}\rhd X}$ for all
$X,Y\in\mathfrak{G}^1$.
On the other hand one can start with a braided commutative
left $H$-module algebra $\mathcal{A}$ and construct the braided Gerstenhaber
algebra of its braided multivector fields, as discussed before.
Note that the braided Schouten-Nijenhuis bracket
$\llbracket\cdot,\cdot\rrbracket_\mathcal{R}$ is the unique
braided Gerstenhaber bracket on $(\mathfrak{X}^\bullet_\mathcal{R}(\mathcal{A}),
\wedge_\mathcal{R})$ such that 
\begin{equation}\label{eq09}
    \llbracket X,a\rrbracket_\mathcal{R}=X(a)
    \text{ and }
    \llbracket X,Y\rrbracket_\mathcal{R}=[X,Y]_\mathcal{R}
\end{equation}
hold for all $a\in\mathcal{A}$ and $X,Y\in\mathfrak{X}^1_\mathcal{R}(\mathcal{A})$.

Dually we consider $\Bbbk$-linear maps
$\omega\colon\mathrm{Der}_\mathcal{R}(\mathcal{A})\rightarrow\mathcal{A}$
such that $\omega(X\cdot a)=\omega(X)\cdot a$ for all
$X\in\mathrm{Der}_\mathcal{R}(\mathcal{A})$ and $a\in\mathcal{A}$ and denote
the corresponding $\Bbbk$-module by $\underline{\Omega}^1_\mathcal{R}(\mathcal{A})$.
We structure $\underline{\Omega}^1_\mathcal{R}(\mathcal{A})$ as a
braided symmetric $\mathcal{A}$-bimodule with left and right $\mathcal{A}$-actions
defined by
\begin{equation}
    (a\cdot\omega)(X)=a\cdot\omega(X)
    \text{ and }
    (\omega\cdot a)(X)=\omega(\mathcal{R}_1^{-1}\rhd X)
    \cdot(\mathcal{R}_2^{-1}\rhd a),
\end{equation}
respectively, and left $H$-action
$
(\xi\rhd\omega)(X)=\xi_{(1)}\rhd(\omega(S(\xi_{(2)})\rhd X)),
$
the adjoint action, for all $\xi\in H$, $a\in\mathcal{A}$,
$\omega\in\underline{\Omega}^1_\mathcal{R}(\mathcal{A})$
and $X\in\mathrm{Der}_\mathcal{R}(\mathcal{A})$.
It follows that $\omega(a\cdot X)=(\mathcal{R}_1^{-1}\rhd a)\cdot
(\mathcal{R}_2^{-1}\rhd\omega)(X)$ and
$\xi\rhd(\omega(X))=(\xi_{(1)}\rhd\omega)(\xi_{(2)}\rhd X)$ for all $\xi\in H$,
$\omega\in\underline{\Omega}^1_\mathcal{R}(\mathcal{A})$, $a\in\mathcal{A}$
and $X\in\mathrm{Der}_\mathcal{R}(\mathcal{A})$.
There is an $H$-equivariant insertion 
$\mathrm{i}^\mathcal{R}\colon\mathfrak{X}^1_\mathcal{R}(\mathcal{A})\otimes
\underline{\Omega}^1_\mathcal{R}(\mathcal{A})
\rightarrow\mathcal{A}$, defined for any $X\in\mathrm{Der}_\mathcal{R}(\mathcal{A})$
and $\omega\in\underline{\Omega}^1_\mathcal{R}(\mathcal{A})$ by
$
\mathrm{i}^\mathcal{R}_X\omega
=(\mathcal{R}_1^{-1}\rhd\omega)(\mathcal{R}_2^{-1}\rhd X).
$
In fact,
\begin{align*}
    \xi\rhd(\mathrm{i}^\mathcal{R}_X\omega)
    =&\xi\rhd((\mathcal{R}_1^{-1}\rhd\omega)(\mathcal{R}_2^{-1}\rhd X))
    =((\xi_{(1)}\mathcal{R}_1^{-1})\rhd\omega)
    ((\xi_{(2)}\mathcal{R}_2^{-1})\rhd X)\\
    =&((\mathcal{R}_1^{-1}\xi_{(2)})\rhd\omega)
    ((\mathcal{R}_2^{-1}\xi_{(1)})\rhd X)
    =\mathrm{i}^\mathcal{R}_{\xi_{(1)}\rhd X}(\xi_{(2)}\rhd\omega)
\end{align*}
for all $\xi\in H$, $X\in\mathrm{Der}_\mathcal{R}(\mathcal{A})$ and
$\omega\in\underline{\Omega}^1_\mathcal{R}(\mathcal{A})$.
It follows that the braided exterior algebra 
$\underline{\Omega}^\bullet_\mathcal{R}(\mathcal{A})$ of 
$\underline{\Omega}^1_\mathcal{R}(\mathcal{A})$ is an $H$-equivariant
braided symmetric $\mathcal{A}$-bimodule.
In the following lines we show that it is also
compatible with the braided evaluation.
For $\omega,\eta\in\underline{\Omega}^1_\mathcal{R}(\mathcal{A})$
we define a $\Bbbk$-bilinear map
$\omega\wedge_\mathcal{R}\eta
\colon\mathrm{Der}(\mathcal{A})\times\mathrm{Der}(\mathcal{A})
\rightarrow\mathcal{A}$ by
$$
(\omega\wedge_\mathcal{R}\eta)(X,Y)
=(\omega(\mathcal{R}_1^{-1}\rhd X))((\mathcal{R}_2^{-1}\rhd\eta)(Y))
-(\omega(\mathcal{R}_1^{-1}\rhd Y))((\mathcal{R}_{2(1)}^{-1}\rhd\eta)
(\mathcal{R}_{2(2)}^{-1}\rhd X))
$$
for all $X,Y\in\mathrm{Der}_\mathcal{R}(\mathcal{A})$. One proves that
$$
-(\omega\wedge_\mathcal{R}\eta)
(\mathcal{R}_1^{-1}\rhd Y,\mathcal{R}_2^{-1}\rhd X)
=(\omega\wedge_\mathcal{R}\eta)(X,Y)
=-((\mathcal{R}_1^{-1}\rhd\eta)\wedge_\mathcal{R}
(\mathcal{R}_2^{-1}\rhd\omega))(X,Y)
$$
and that
\begin{equation}
\begin{split}
    (\omega\wedge_\mathcal{R}\eta)(X,Y\cdot a)
    =&((\omega\wedge_\mathcal{R}\eta)(X,Y))\cdot a,\\
    (\omega\wedge_\mathcal{R}\eta)(a\cdot X,Y)
    =&(\mathcal{R}_1^{-1}\rhd a)\cdot
    ((\mathcal{R}_2^{-1}\rhd(\omega\wedge_\mathcal{R}\eta))(X,Y)),\\
    \xi\rhd((\omega\wedge_\mathcal{R}\eta)(X,Y))
    =&((\xi_{(1)}\rhd\omega)\wedge_\mathcal{R}(\xi_{(2)}\rhd\eta))
    (\xi_{(3)}\rhd X,\xi_{(4)}\rhd Y)
\end{split}
\end{equation}
hold for all $\xi\in H$, $\omega,\eta\in\underline{\Omega}^1_\mathcal{R}
(\mathcal{A})$, $a\in\mathcal{A}$ and
$X,Y\in\mathrm{Der}_\mathcal{R}(\mathcal{A})$.
The evaluations of the $H$-action and $\mathcal{A}$-module actions read
\begin{equation}
\begin{split}
    (\xi\rhd(\omega\wedge_\mathcal{R}\eta))(X,Y)
    =&\xi_{(1)}\rhd((\omega\wedge_\mathcal{R}\eta)
    (S(\xi_{(3)})\rhd X,S(\xi_{(2)})\rhd Y)),\\
    (a\cdot(\omega\wedge_\mathcal{R}\eta))(X,Y)
    =&a\cdot((\omega\wedge_\mathcal{R}\eta)(X,Y)),\\
    ((\omega\wedge_\mathcal{R}\eta)\cdot a)(X,Y)
    =&((\omega\wedge_\mathcal{R}\eta)
    (\mathcal{R}_{1(1)}^{-1}\rhd X,\mathcal{R}_{1(2)}^{-1}\rhd Y))
    \cdot(\mathcal{R}_2^{-1}\rhd a).
\end{split}
\end{equation}
Inductively one defines the evaluation of higher wedge products.
Explicitly, the evaluated module actions on factorizing elements
$\omega_1\wedge_\mathcal{R}\ldots\wedge_\mathcal{R}\omega_k
\in\underline{\Omega}^k_\mathcal{R}(\mathcal{A})$ read
\begin{equation}
\begin{split}
    (\xi\rhd&(\omega_1\wedge_\mathcal{R}\ldots\wedge_\mathcal{R}\omega_k))
    (X_1,\ldots,X_k)\\
    =&\xi_{(1)}\rhd((\omega_1\wedge_\mathcal{R}\ldots\wedge_\mathcal{R}\omega_k)
    (S(\xi_{(k+1)})\rhd X_1,\ldots,S(\xi_{(2)})\rhd X_k)),
\end{split}
\end{equation}
\begin{equation}
    (a\cdot(\omega_1\wedge_\mathcal{R}\ldots\wedge_\mathcal{R}\omega_k))
    (X_1,\ldots,X_k)
    =a\cdot((\omega_1\wedge_\mathcal{R}\ldots\wedge_\mathcal{R}\omega_k)
    (X_1,\ldots,X_k))
\end{equation}
and
\begin{equation}
\begin{split}
    ((\omega_1&\wedge_\mathcal{R}\ldots\wedge_\mathcal{R}\omega_k)\cdot a)
    (X_1,\ldots,X_k)\\
    =&((\omega_1\wedge_\mathcal{R}\ldots\wedge_\mathcal{R}\omega_k)
    (\mathcal{R}_{1(1)}^{-1}\rhd X_1,\ldots,\mathcal{R}_{1(k)}^{-1}\rhd X_k))
    \cdot(\mathcal{R}_2^{-1}\rhd a).
\end{split}
\end{equation}
for all $X_1,\ldots,X_k\in\mathrm{Der}_\mathcal{R}(\mathcal{A})$,
$a\in\mathcal{A}$ and $\xi\in H$.
It is useful to further define the insertion $\mathrm{i}^\mathcal{R}_X\colon
\underline{\Omega}^\bullet_\mathcal{R}(\mathcal{A})
\rightarrow\underline{\Omega}^{\bullet-1}_\mathcal{R}(\mathcal{A})$ of an element
$X\in\mathrm{Der}_\mathcal{R}(\mathcal{A})$ into the \textit{last} slot 
an element $\omega\in\underline{\Omega}^k_\mathcal{R}(\mathcal{A})$ by
\begin{equation}
    \mathrm{i}^\mathcal{R}_X\omega
    =(-1)^{k-1}(\mathcal{R}_1^{-1}\rhd\omega)(\cdot,\ldots,\cdot,
    \mathcal{R}_2^{-1}\rhd X).
\end{equation}
Inductively we set 
\begin{equation}
    \mathrm{i}^\mathcal{R}_{X\wedge_\mathcal{R}Y}
    =\mathrm{i}^\mathcal{R}_{X}
    \mathrm{i}^\mathcal{R}_{Y}
\end{equation}
for all $X,Y\in\mathfrak{X}^\bullet_\mathcal{R}(\mathcal{A})$.
\begin{lemma}
$(\underline{\Omega}^\bullet_\mathcal{R}(\mathcal{A}),\wedge_\mathcal{R})$
is a graded braided commutative associative unital algebra
and an $H$-equivariant braided symmetric $\mathcal{A}$-bimodule. The insertion 
\begin{equation}
    \mathrm{i}^\mathcal{R}\colon\mathfrak{X}^\bullet_\mathcal{R}(\mathcal{A})\otimes
    \underline{\Omega}^\bullet_\mathcal{R}(\mathcal{A})
    \rightarrow\underline{\Omega}^{\bullet}_\mathcal{R}(\mathcal{A})
\end{equation}
of braided multivector fields
is $H$-equivariant such that $\mathrm{i}^\mathcal{R}_X$ is a
right $\mathcal{A}$-linear and braided left
$\mathcal{A}$-linear homogeneous map of degree $-k$
for all $X\in\mathfrak{X}^k_\mathcal{R}(\mathcal{A})$. Furthermore,
$\mathrm{i}^\mathcal{R}_X$ is
left $\mathcal{A}$-linear and braided right $\mathcal{A}$-linear in $X$.
If $k=1$ $\mathrm{i}^\mathcal{R}_X$ is a graded braided derivation
of degree $-1$.
\end{lemma}
\begin{proof}
Fix $a,b\in\mathcal{A}$, $X\in\mathrm{Der}_\mathcal{R}(\mathcal{A})$, $\xi\in H$ and
$\omega\in\underline{\Omega}^1_\mathcal{R}(\mathcal{A})$.
First of all, the left and right $\mathcal{A}$ and left $H$-module actions
are well-defined on $\underline{\Omega}^1_\mathcal{R}(\mathcal{A})$,
since $(b\cdot\omega)(X\cdot a)=b\cdot(\omega(X\cdot a))
=((b\cdot\omega)(X))\cdot a$,
\begin{align*}
    (\omega\cdot b)(X\cdot a)
    =&\omega((\mathcal{R}_{1(1)}^{-1}\rhd X)\cdot(\mathcal{R}_{1(2)}^{-1}\rhd a))
    \cdot(\mathcal{R}_2^{-1}\rhd b)\\
    =&\omega(\mathcal{R}_{1(1)}^{-1}\rhd X)
    \cdot((\mathcal{R}_1^{'-1}\mathcal{R}_2^{-1})\rhd b)
    \cdot((\mathcal{R}_2^{'-1}\mathcal{R}_{1(2)}^{-1})\rhd a)\\
    =&\omega(\mathcal{R}_{1}^{-1}\rhd X)
    \cdot(\mathcal{R}_2^{-1}\rhd b)\cdot a\\
    =&((\omega\cdot b)(X))\cdot a
\end{align*}
and
\begin{align*}
    (\xi\rhd\omega)(X\cdot a)
    =&\xi_{(1)}\rhd(\omega((S(\xi_{(2)})_{(1)}\rhd X)
    \cdot(S(\xi_{(2)})_{(2)}\rhd a)))\\
    =&\xi_{(1)}\rhd(\omega(S(\xi_{(3)})\rhd X)
    \cdot(S(\xi_{(2)})\rhd a))\\
    =&(\xi_{(1)}\rhd\omega(S(\xi_{(4)})\rhd X))
    \cdot((\xi_{(2)}S(\xi_{(3)}))\rhd a)\\
    =&((\xi_{(1)}\rhd\omega)((\xi_{(2)}S(\xi_{(3)}))\rhd X))\cdot a\\
    =&((\xi\rhd\omega)(X))\cdot a
\end{align*}
hold by the hexagon relations and the bialgebra anti-homomorphism properties of
$S$. The $\mathcal{A}$-bimodule is $H$-equivariant, since
\begin{align*}
    (\xi\rhd(a\cdot\omega\cdot b))(X)
    =&\xi_{(1)}\rhd((a\cdot\omega\cdot b)(S(\xi_{(2)})\rhd X))\\
    =&(\xi_{(1)}\rhd a)\cdot
    (\xi_{(2)}\rhd(\omega((\mathcal{R}_1^{-1}S(\xi_{(4)}))\rhd X)))\cdot
    ((\xi_{(3)}\mathcal{R}_2^{-1})\rhd b)\\
    =&(\xi_{(1)}\rhd a)\cdot
    ((\xi_{(2)}\rhd\omega)((\xi_{(3)}\mathcal{R}_1^{-1}S(\xi_{(5)}))\rhd X))\cdot
    ((\xi_{(4)}\mathcal{R}_2^{-1})\rhd b)\\
    =&(\xi_{(1)}\rhd a)\cdot
    ((\xi_{(2)}\rhd\omega)(\mathcal{R}_1^{-1}\rhd X))\cdot
    ((\mathcal{R}_2^{-1}\xi_{(3)})\rhd b)\\
    =&((\xi_{(1)}\rhd a)\cdot(\xi_{(2)}\rhd\omega)\cdot(\xi_{(3)}\rhd b))(X)
\end{align*}
and it is braided symmetric because
\begin{align*}
    ((\mathcal{R}_1^{-1}\rhd\omega)\cdot(\mathcal{R}_2^{-1}\rhd a))(X)
    =&((\mathcal{R}_1^{-1}\rhd\omega)
    (\mathcal{R}_1^{'-1}\rhd X))
    \cdot((\mathcal{R}_2^{'-1}\mathcal{R}_2^{-1})\rhd a)\\
    =&((\mathcal{R}_1^{''-1}\mathcal{R}_2^{'-1}\mathcal{R}_2^{-1})\rhd a)
    \cdot(\mathcal{R}_2^{''-1}\rhd((\mathcal{R}_1^{-1}\rhd\omega)
    (\mathcal{R}_1^{'-1}\rhd X)))\\
    =&(a\cdot\omega)(X).
\end{align*}
These properties extend to the braided Graßmann algebra
$\underline{\Omega}^\bullet_\mathcal{R}(\mathcal{A})$, giving an associative
graded braided commutative product $\wedge_\mathcal{R}$.
We further prove that $\mathrm{i}^\mathcal{R}_X$
is a graded braided derivation of the wedge product for
$X\in\mathrm{Der}_\mathcal{R}
(\mathcal{A})$. Let $\omega,\eta\in\underline{\Omega}^1_\mathcal{R}(\mathcal{A})$.
Then
\begin{align*}
    \mathrm{i}^\mathcal{R}_X(\omega\wedge_\mathcal{R}\eta)
    =&(-1)^{2-1}((\mathcal{R}_{1(1)}^{-1}\rhd\omega)
    \wedge_\mathcal{R}(\mathcal{R}_{1(2)}^{-1}\rhd\eta))
    (\cdot,\mathcal{R}_2^{-1}\rhd X)\\
    =&-(\mathcal{R}_{1(1)}^{-1}\rhd\omega)
    (\mathcal{R}_{1(2)}^{-1}\rhd\eta)(\mathcal{R}_2^{-1}\rhd X)\\
    &+(\mathcal{R}_{1(1)}^{-1}\rhd\omega)
    ((\mathcal{R}_1^{'-1}\mathcal{R}_2^{-1})\rhd X)
    ((\mathcal{R}_2^{'-1}\mathcal{R}_{1(2)}^{-1})\rhd\eta)\\
    =&\mathrm{i}^\mathcal{R}_X(\omega)\wedge_\mathcal{R}\eta
    +(-1)^{1\cdot 1}(\mathcal{R}_1^{-1}\rhd\omega)\wedge_\mathcal{R}
    \mathrm{i}^\mathcal{R}_{\mathcal{R}_2^{-1}\rhd X}\eta.
\end{align*}
In particular this implies $\xi\rhd(\mathrm{i}^\mathcal{R}_X
(\omega\wedge_\mathcal{R}\eta))
=\mathrm{i}^\mathcal{R}_{\xi_{(1)}\rhd
X}((\xi_{(2)}\rhd\omega)\wedge_\mathcal{R}(\xi_{(3)}\rhd\eta))$
for all $\xi\in H$.
Inductively, one shows
\begin{align*}
    \mathrm{i}^\mathcal{R}_X(\omega\wedge_\mathcal{R}\eta)
    =(\mathrm{i}^\mathcal{R}_X\omega)\wedge_\mathcal{R}\eta
    +(-1)^k(\mathcal{R}_1^{-1}\rhd\omega)\wedge_\mathcal{R}
    \mathrm{i}^\mathcal{R}_{\mathcal{R}_2^{-1}\rhd X}\eta
\end{align*}
and $\xi\rhd(\mathrm{i}^\mathcal{R}_X\eta)
=\mathrm{i}^\mathcal{R}_{\xi_{(1)}\rhd X}(\xi_{(2)}\rhd\eta)$
for all $\xi\in H$, $X\in\mathrm{Der}_\mathcal{R}(\mathcal{A})$, $\omega\in
\underline{\Omega}^k_\mathcal{R}(\mathcal{A})$ and $\eta\in
\underline{\Omega}^\bullet_\mathcal{R}(\mathcal{A})$.
For factorizing elements $X_1\wedge_\mathcal{R}
X_2\in\mathfrak{X}^2_\mathcal{R}(\mathcal{A})$ this implies
\begin{align*}
    \xi\rhd\mathrm{i}^\mathcal{R}_{X_1\wedge_\mathcal{R}X_2}\omega
    =&\xi\rhd(\mathrm{i}^\mathcal{R}_{X_1}
    \mathrm{i}^\mathcal{R}_{X_2}\omega)
    =\mathrm{i}^\mathcal{R}_{\xi_{(1)}\rhd X_1}
    \mathrm{i}^\mathcal{R}_{\xi_{(2)}\rhd X_2}(\xi_{(3)}\rhd\omega)
    =\mathrm{i}^\mathcal{R}_{\xi_{(1)}\rhd(X_1\wedge_\mathcal{R}X_2)}
    (\xi_{(2)}\rhd\omega)
\end{align*}
for all $\xi\in H$ and $\omega\in\underline{\Omega}^\bullet_\mathcal{R}(\mathcal{A})$
and inductively one obtains $\xi\rhd(\mathrm{i}^\mathcal{R}_X\omega)
=\mathrm{i}^\mathcal{R}_{\xi_{(1)}\rhd X}(\xi_{(2)}\rhd\omega)$ for any
$X\in\underline{\Omega}^\bullet_\mathcal{R}(\mathcal{A})$. It is easy to verify
that $\mathrm{i}^\mathcal{R}$ sarisfies the linearity properties
\begin{equation}
\begin{split}
    \mathrm{i}^\mathcal{R}_{a\cdot X}\omega
    =&a\cdot(\mathrm{i}^\mathcal{R}_X\omega),~
    \mathrm{i}^\mathcal{R}_{X\cdot a}\omega
    =(\mathrm{i}^\mathcal{R}_X(\mathcal{R}_1^{-1}\rhd\omega))
    \cdot(\mathcal{R}_2^{-1}\rhd a),\\
    \mathrm{i}^\mathcal{R}_X(\omega\cdot a)
    =&(\mathrm{i}^\mathcal{R}_X\omega)\cdot a,~
    \mathrm{i}^\mathcal{R}_X(a\cdot\omega)
    =(\mathcal{R}_1^{-1}\rhd a)\cdot
    (\mathrm{i}^\mathcal{R}_{\mathcal{R}_2^{-1}\rhd X}\omega)
\end{split}
\end{equation}
for all $X\in\mathfrak{X}^\bullet_\mathcal{R}(\mathcal{A})$,
$a\in\mathcal{A}$ and $\omega\in\underline{\Omega}^\bullet_\mathcal{R}
(\mathcal{A})$. This concludes the proof of the lemma.
\end{proof}

\subsection{Braided Cartan Calculi and Gauge Equivalence}

In the following pages we construct a noncommutative Cartan calculus
for any braided symmetric algebra. The development is entirely parallel to
the Cartan calculus of a commutative algebra, however in a symmetric
braided monoidal category. In particular, we are not constrained to use
the center of the algebra. Afterwards we define a twist deformation
of any braided Cartan calculus and show that it is isomorphic to the braided
Cartan calculus of the twist deformed algebra with respect to the
twisted triangular structure.

One defines a differential
$\mathrm{d}\colon\underline{\Omega}_\mathcal{R}^\bullet(\mathcal{A})
\rightarrow\underline{\Omega}_\mathcal{R}^{\bullet+1}(\mathcal{A})$
on $a\in\mathcal{A}$ by $\mathrm{i}^\mathcal{R}_X(\mathrm{d}a)=X(a)$ for all
$X\in\mathrm{Der}_\mathcal{R}(\mathcal{A})$, on 
$\omega\in\underline{\Omega}^1_\mathcal{R}(\mathcal{A})$ by
\begin{equation}
    (\mathrm{d}\omega)(X,Y)
    =(\mathcal{R}_1^{-1}\rhd X)((\mathcal{R}_2^{-1}\rhd\omega)(Y))
    -(\mathcal{R}_1^{-1}\rhd Y)(\mathcal{R}_2^{-1}\rhd(\omega(X)))
    -\omega([X,Y]_\mathcal{R})    
\end{equation}
for all $X,Y\in\mathrm{Der}_\mathcal{R}(\mathcal{A})$
and extends $\mathrm{d}$
to higher wedge powers by demanding it to be a
graded derivation with respect to $\wedge_\mathcal{R}$, i.e.
\begin{equation}
    \mathrm{d}(\omega_1\wedge_\mathcal{R}\omega_2)
    =(\mathrm{d}\omega_1)\wedge_\mathcal{R}\omega_2
    +(-1)^k\omega_1\wedge_\mathcal{R}(\mathrm{d}\omega_2)
\end{equation}
for $\omega_1\in\underline{\Omega}_\mathcal{R}^k(\mathcal{A})$
and $\omega_2\in\underline{\Omega}_\mathcal{R}^\bullet(\mathcal{A})$.
Alternatively we can directly define
$\mathrm{d}\omega\in\underline{\Omega}^{k+1}_\mathcal{R}(\mathcal{A})$ for any
$\omega\in\underline{\Omega}^{k}_\mathcal{R}(\mathcal{A})$ by
\begin{equation}\label{eq06}
\begin{split}
    (\mathrm{d}\omega)(X_0,\ldots,X_k)
    =&\sum_{i=0}^k(-1)^i
    (\mathcal{R}_1^{-1}\rhd X_i)
    \bigg((\mathcal{R}_{2(1)}^{-1}\rhd\omega)\bigg(\\
    &\mathcal{R}_{2(2)}^{-1}\rhd X_0,\ldots,
    \mathcal{R}_{2(i+1)}^{-1}\rhd X_{i-1},
    \widehat{X_i},X_{i+1},\ldots,X_k\bigg)\bigg)\\
    &+\sum_{i<j}(-1)^{i+j}
    \omega\bigg([\mathcal{R}_1^{-1}\rhd X_i,
    \mathcal{R}_1^{'-1}\rhd X_j]_\mathcal{R},\\
    &(\mathcal{R}_{2(1)}^{'-1}\mathcal{R}_{2(1)}^{-1})\rhd X_{0},
    \ldots,(\mathcal{R}_{2(i)}^{'-1}\mathcal{R}_{2(i)}^{-1})\rhd X_{i-1},
    \widehat{X_i},\\
    &\mathcal{R}_{2(i+1)}^{'-1}\rhd X_{i+1},\ldots,
    \mathcal{R}_{2(j-1)}^{'-1}\rhd X_{j-1},\widehat{X_j},
    X_{j+1},\ldots,X_k\bigg)
\end{split}
\end{equation}
for all $X_0,\ldots,X_k\in\mathrm{Der}_\mathcal{R}(\mathcal{A})$.
It is sufficient to prove $\mathrm{d}^2=0$ on $\underline{\Omega}^k_\mathcal{R}
(\mathcal{A})$ for $k<2$, since $\mathrm{d}^2$ is a graded braided derivation.
The computations can be found in \cite{ThomasPhDThesis}.
Define now the \textit{braided differential forms}
$\Omega^\bullet_\mathcal{R}(\mathcal{A})$ on $\mathcal{A}$ to be the smallest
differential graded subalgebra of
$\underline{\Omega}^\bullet_\mathcal{R}(\mathcal{A})$
such that $\mathcal{A}\subseteq\Omega^\bullet_\mathcal{R}(\mathcal{A})$.
Every element of $\Omega^k_\mathcal{R}(\mathcal{A})$
can be written as a finite sum of elements of the form 
$a_0\mathrm{d}a_1\wedge_\mathcal{R}\ldots\wedge_\mathcal{R}\mathrm{d}a_k$, where
$a_0,\ldots,a_k\in\mathcal{A}$.
Using eq.(\ref{eq06}) and the fact that the braided
commutator is $H$-equivariant it immediately follows that $\mathrm{d}$
commutes with the left $H$-module action. In other words, $\mathrm{d}$
is equivariant with respect to the adjoint action, implying
\begin{equation}
    (\xi\rhd\mathrm{d})\omega
    =\xi_{(1)}\rhd(\mathrm{d}(S(\xi_{(2)})\rhd\omega))
    =(\xi_{(1)}S(\xi_{(2)}))\rhd(\mathrm{d}\omega)
    =\epsilon(\xi)\mathrm{d}\omega
\end{equation}
for all $\xi\in H$ and $\omega\in\Omega^\bullet_\mathcal{R}(\mathcal{A})$.
Recall that the \textit{graded braided commutator}
of two homogeneous maps
$\Phi,\Psi\colon\mathfrak{G}^\bullet\rightarrow\mathfrak{G}^\bullet$
of degree $k$ and $\ell$ between braided Graßmann algebras is defined by
\begin{equation}
    [\Phi,\Psi]_\mathcal{R}
    =\Phi\circ\Psi-(-1)^{k\ell}(\mathcal{R}_1^{-1}\rhd\Psi)
    \circ(\mathcal{R}_2^{-1}\rhd\Phi).
\end{equation}
If $\Phi$ or $\Psi$ is equivariant, the graded braided
commutator coincides with the graded commutator. Furthermore,
if $\Phi,\Psi\colon\mathfrak{X}^\bullet_\mathcal{R}(\mathcal{A})
\otimes\mathfrak{G}^\bullet\rightarrow\mathfrak{G}^\bullet$
are $H$-equivariant maps such that
$\Phi_X,\Psi_Y\colon\mathfrak{G}^\bullet\rightarrow\mathfrak{G}^\bullet$
are homogeneous of degree $k$ and $\ell$
for any $X\in\mathfrak{X}^k_\mathcal{R}(\mathcal{A})$ and
$Y\in\mathfrak{X}^\ell_\mathcal{R}(\mathcal{A})$, respectively, 
the graded braided commutator of $\Phi_X$ and $\Psi_Y$ reads
\begin{equation}
    [\Phi_X,\Psi_Y]_\mathcal{R}
    =\Phi_X\circ\Psi_Y
    -(-1)^{k\ell}\Psi_{\mathcal{R}_1^{-1}\rhd Y}\circ\Phi_{\mathcal{R}_2^{-1}\rhd X}.
\end{equation}
For any $X\in\mathfrak{X}^\bullet_\mathcal{R}(\mathcal{A})$ we define the
\textit{braided Lie derivative} $\mathscr{L}^\mathcal{R}\colon
\mathfrak{X}^\bullet_\mathcal{R}(\mathcal{A})\otimes
\Omega^\bullet_\mathcal{R}(\mathcal{A})
\rightarrow\Omega^{\bullet}_\mathcal{R}(\mathcal{A})$ by
$\mathscr{L}^\mathcal{R}_X=[\mathrm{i}^\mathcal{R}_X,\mathrm{d}]_\mathcal{R}$.
It is $H$-equivariant and
if $X\in\mathfrak{X}^k_\mathcal{R}(\mathcal{A})$, $\mathscr{L}^\mathcal{R}_X$
is a homogeneous map of degree $-(k-1)$.
For $k=1$ we obtain a braided derivation $\mathscr{L}^\mathcal{R}_X$ of 
$\Omega^\bullet_\mathcal{R}(\mathcal{A})$.
\begin{lemma}\label{lemma02}
One has
\begin{equation}
    \mathscr{L}^\mathcal{R}_a\omega=-(\mathrm{d}a)\wedge_\mathcal{R}\omega
    \text{ and }
    \mathscr{L}^\mathcal{R}_{X\wedge_\mathcal{R}Y}
    =\mathrm{i}^\mathcal{R}_X\mathscr{L}^\mathcal{R}_Y
    +(-1)^\ell\mathscr{L}^\mathcal{R}_X\mathrm{i}^\mathcal{R}_Y
\end{equation}
for all $a\in\mathcal{A}$, 
$\omega\in\Omega^\bullet_\mathcal{R}(\mathcal{A})$,
$X\in\mathfrak{X}^\bullet_\mathcal{R}(\mathcal{A})$ and
$Y\in\mathfrak{X}^\ell_\mathcal{R}(\mathcal{A})$. If
$X,Y\in\mathfrak{X}^1_\mathcal{R}(\mathcal{A})$
\begin{equation}
    [\mathscr{L}^\mathcal{R}_X,\mathrm{i}^\mathcal{R}_Y]_\mathcal{R}
    =\mathrm{i}^\mathcal{R}_{[X,Y]_\mathcal{R}}
\end{equation}
holds.
\end{lemma}
\begin{proof}
By the definition of the braided Lie derivative
\begin{align*}
    \mathscr{L}^\mathcal{R}_a\omega
    =&\mathrm{i}^\mathcal{R}_a\mathrm{d}\omega
    -(-1)^{0\cdot 1}\mathrm{d}(\mathrm{i}^\mathcal{R}_a\omega)
    =a\wedge_\mathcal{R}\mathrm{d}\omega
    -((\mathrm{d}a)\wedge_\mathcal{R}\omega
    +(-1)^{0}a\wedge_\mathcal{R}\mathrm{d}\omega)
    =-(\mathrm{d}a)\wedge_\mathcal{R}\omega
\end{align*}
follows. From the graded braided Leibniz rule of the graded braided commutator
we obtain
\begin{align*}
    \mathscr{L}^\mathcal{R}_{X\wedge_\mathcal{R}Y}
    =&[\mathrm{i}^\mathcal{R}_{X\wedge_\mathcal{R}Y},\mathrm{d}]_\mathcal{R}
    =[\mathrm{i}^\mathcal{R}_{X}\mathrm{i}^\mathcal{R}_{Y},
    \mathrm{d}]_\mathcal{R}
    =\mathrm{i}^\mathcal{R}_{X}[\mathrm{i}^\mathcal{R}_{Y},
    \mathrm{d}]_\mathcal{R}
    +(-1)^{-1\cdot\ell}[\mathrm{i}^\mathcal{R}_{X},
    \mathrm{d}]_\mathcal{R}\mathrm{i}^\mathcal{R}_{Y}\\
    =&\mathrm{i}^\mathcal{R}_X\mathscr{L}^\mathcal{R}_Y
    +(-1)^\ell\mathscr{L}^\mathcal{R}_X\mathrm{i}^\mathcal{R}_Y.
\end{align*}
The missing formula trivially holds on braided differential forms of degree $0$,
while for $\omega\in\Omega^1_\mathcal{R}(\mathcal{A})$ one obtains
\begin{align*}
    [\mathscr{L}^\mathcal{R}_X,\mathrm{i}^\mathcal{R}_Y]_\mathcal{R}\omega
    =&\mathscr{L}^\mathcal{R}_X\mathrm{i}^\mathcal{R}_Y\omega
    -(-1)^{0\cdot 1}\mathrm{i}^\mathcal{R}_{\mathcal{R}_1^{-1}\rhd Y}
    \mathscr{L}^\mathcal{R}_{\mathcal{R}_2^{-1}\rhd X}\omega\\
    =&(\mathrm{i}^\mathcal{R}_X\mathrm{d}+\mathrm{d}\mathrm{i}^\mathcal{R}_X)
    \mathrm{i}^\mathcal{R}_Y\omega
    -\mathrm{i}^\mathcal{R}_{\mathcal{R}_1^{-1}\rhd Y}
    (\mathrm{i}^\mathcal{R}_{\mathcal{R}_2^{-1}\rhd X}\mathrm{d}
    +\mathrm{d}\mathrm{i}^\mathcal{R}_{\mathcal{R}_2^{-1}\rhd X})\omega\\
    =&X(\mathrm{i}^\mathcal{R}_Y\omega)+0
    +(\mathrm{d}((\mathcal{R}_1^{''-1}\mathcal{R}_1^{'-1})\rhd\omega))
    ((\mathcal{R}_2^{''-1}\mathcal{R}_1^{-1})\rhd Y,
    (\mathcal{R}_2^{'-1}\mathcal{R}_2^{-1})\rhd X)\\
    &-(\mathcal{R}_1^{-1}\rhd Y)
    (\mathrm{i}^\mathcal{R}_{\mathcal{R}_2^{-1}\rhd X}\omega)\\
    =&\mathrm{i}^\mathcal{R}_{[X,Y]}\omega
\end{align*}
for all $X,Y\in\mathfrak{X}^1_\mathcal{R}(\mathcal{A})$. Since
$[\mathscr{L}^\mathcal{R}_X,\mathrm{i}^\mathcal{R}_Y]_\mathcal{R}$
is a graded braided derivation this is all we have to prove.
\end{proof}
We are prepared to prove the main theorem of this section.
It assigns to any braided commutative left $H$-module algebra
$\mathcal{A}$ a noncommutative Cartan calculus, which we call
\textit{the braided Cartan calculus} of $\mathcal{A}$ in the
following.
\begin{theorem}[Braided Cartan calculus]
Let $\mathcal{A}$ be a braided commutative left $H$-module algebra
and consider the braided differential forms 
$(\Omega^\bullet_\mathcal{R}(\mathcal{A}),\wedge_\mathcal{R},\mathrm{d})$ 
and braided multivector fields
$(\mathfrak{X}^\bullet_\mathcal{R}(\mathcal{A}),\wedge_\mathcal{R},
\llbracket\cdot,\cdot\rrbracket_\mathcal{R})$ on $\mathcal{A}$.
The homogeneous maps
\begin{equation}
    \mathscr{L}^\mathcal{R}_X\colon\Omega^\bullet_\mathcal{R}(\mathcal{A})
    \rightarrow\Omega^{\bullet-(k-1)}_\mathcal{R}(\mathcal{A})
    \text{ and }
    \mathrm{i}^\mathcal{R}_X\colon\Omega^\bullet_\mathcal{R}(\mathcal{A})
    \rightarrow\Omega^{\bullet-k}_\mathcal{R}(\mathcal{A}),
\end{equation}
where $X\in\mathfrak{X}^k_\mathcal{R}(\mathcal{A})$,
and $\mathrm{d}\colon\Omega_\mathcal{R}^\bullet(\mathcal{A})
\rightarrow\Omega_\mathcal{R}^{\bullet+1}(\mathcal{A})$
satisfy
\begin{equation}
\begin{split}
        [\mathscr{L}^\mathcal{R}_X,\mathscr{L}^\mathcal{R}_Y]_\mathcal{R}
        =&\mathscr{L}^\mathcal{R}_{\llbracket X,Y\rrbracket_\mathcal{R}},\\
        [\mathscr{L}^\mathcal{R}_X,\mathrm{i}^\mathcal{R}_Y]_\mathcal{R}
        =&\mathrm{i}^\mathcal{R}_{\llbracket X,Y\rrbracket_\mathcal{R}},\\
        [\mathscr{L}^\mathcal{R}_X,\mathrm{d}]_\mathcal{R}=&0,
\end{split}
\hspace{1cm}
\begin{split}
        [\mathrm{i}^\mathcal{R}_X,\mathrm{i}^\mathcal{R}_Y]_\mathcal{R}=&0,\\
        [\mathrm{i}^\mathcal{R}_X,\mathrm{d}]_\mathcal{R}
        =&\mathscr{L}^\mathcal{R}_X,\\
        [\mathrm{d},\mathrm{d}]_\mathcal{R}=&0,
\end{split}
\end{equation}
for all $X,Y\in\mathfrak{X}^\bullet_\mathcal{R}(\mathcal{A})$.
\end{theorem}
\begin{proof}
We are going to prove the above formulas in reversed order. Since
$\mathrm{d}$ is a differential it follows that
$[\mathrm{d},\mathrm{d}]_\mathcal{R}=2\mathrm{d}^2=0$. Recall that
there is no braiding appearing here since $\mathrm{d}$ is equivariant.
By the definition of the braided Lie derivative 
$[\mathrm{i}^\mathcal{R}_X,\mathrm{d}]_\mathcal{R}
=\mathscr{L}^\mathcal{R}_X$ holds for all $X\in\mathfrak{X}^\bullet_\mathcal{R}
(\mathcal{A})$. Let $X\in\mathfrak{X}^k_\mathcal{R}(\mathcal{A})$
and $Y\in\mathfrak{X}^\ell_\mathcal{R}(\mathcal{A})$. Then
$$
[\mathrm{i}^\mathcal{R}_X,\mathrm{i}^\mathcal{R}_Y]_\mathcal{R}
=\mathrm{i}^\mathcal{R}_X\mathrm{i}^\mathcal{R}_Y
-(-1)^{k\ell}\mathrm{i}^\mathcal{R}_{\mathcal{R}_1^{-1}\rhd Y}
\mathrm{i}^\mathcal{R}_{\mathcal{R}_2^{-1}\rhd X}
=\mathrm{i}^\mathcal{R}_{X\wedge_\mathcal{R}Y
-(-1)^{k\ell}(\mathcal{R}_1^{-1}\rhd Y)
\wedge_\mathcal{R}(\mathcal{R}_2^{-1}\rhd X)}
=0
$$
follows from the definition of
$\mathrm{i}^\mathcal{R}_{X\wedge_\mathcal{R}Y}
=\mathrm{i}^\mathcal{R}_X\mathrm{i}^\mathcal{R}_Y$.
Using the graded braided Jacobi identity of the graded braided commutator we obtain
\begin{align*}
    [[\mathrm{i}^\mathcal{R}_X,\mathrm{d}]_\mathcal{R},\mathrm{d}]_\mathcal{R}
    =[\mathrm{i}^\mathcal{R}_X,[\mathrm{d},\mathrm{d}]_\mathcal{R}]_\mathcal{R}
    +(-1)^{1\cdot 1}
    [[\mathrm{i}^\mathcal{R}_X,\mathrm{d}]_\mathcal{R},\mathrm{d}]_\mathcal{R}
    =-[[\mathrm{i}^\mathcal{R}_X,
    \mathrm{d}]_\mathcal{R},\mathrm{d}]_\mathcal{R}
\end{align*}
for all $X\in\mathfrak{X}^\bullet_\mathcal{R}(\mathcal{A})$, which implies
$[\mathscr{L}^\mathcal{R}_X,\mathrm{d}]_\mathcal{R}=0$.
Again, there is no braiding appearing since $\mathrm{d}$ is equivariant.
Recall that the braided Schouten-Nijenhuis bracket of a homogeneous element
$Y=Y_1\wedge_\mathcal{R}\cdots\wedge_\mathcal{R}Y_\ell\in
\mathfrak{X}^\ell_\mathcal{R}(\mathcal{A})$ with $a\in\mathcal{A}$ and
$X\in\mathfrak{X}^1_\mathcal{R}(\mathcal{A})$ read
$$
\llbracket a,Y\rrbracket_\mathcal{R}
=\sum_{j=1}^\ell(-1)^{j+1}(\mathcal{R}_{1(1)}^{-1}\rhd Y_1)\wedge_\mathcal{R}
\cdots\wedge_\mathcal{R}(\mathcal{R}_{1(j-1)}^{-1}\rhd Y_{j-1})
\wedge_\mathcal{R}\llbracket\mathcal{R}_2^{-1}\rhd a,Y_j\rrbracket_\mathcal{R}
\wedge_\mathcal{R}Y_{j+1}\wedge_\mathcal{R}\cdots\wedge_\mathcal{R}Y_\ell
$$
and
$$
\llbracket X,Y\rrbracket_\mathcal{R}
=\sum_{j=1}^\ell(\mathcal{R}_{1(1)}^{-1}\rhd Y_1)\wedge_\mathcal{R}
\cdots\wedge_\mathcal{R}(\mathcal{R}_{1(j-1)}^{-1}\rhd Y_{j-1})
\wedge_\mathcal{R}[\mathcal{R}_2^{-1}\rhd X,Y_j]_\mathcal{R}
\wedge_\mathcal{R}Y_{j+1}\wedge_\mathcal{R}\cdots\wedge_\mathcal{R}Y_\ell,
$$
respectively. If $\ell=1$ we obtain
\begin{align*}
    [\mathscr{L}^\mathcal{R}_a,\mathrm{i}^\mathcal{R}_Y]_\mathcal{R}\omega
    =&(\mathscr{L}^\mathcal{R}_a\mathrm{i}^\mathcal{R}_Y
    -(-1)^{(-1)\cdot 1}\mathrm{i}^\mathcal{R}_{\mathcal{R}_1^{-1}\rhd Y}
    \mathscr{L}^\mathcal{R}_{\mathcal{R}_2^{-1}\rhd a})\omega\\
    =&-\mathrm{d}a\wedge_\mathcal{R}\mathrm{i}^\mathcal{R}_Y\omega
    -\mathrm{i}^\mathcal{R}_{\mathcal{R}_1^{-1}\rhd Y}(
    \mathrm{d}(\mathcal{R}_2^{-1}\rhd a)\wedge_\mathcal{R}\omega)\\
    =&-\mathrm{d}a\wedge_\mathcal{R}\mathrm{i}^\mathcal{R}_Y\omega
    -(\mathcal{R}_1^{-1}\rhd Y)(\mathcal{R}_2^{-1}\rhd a)\cdot\omega
    +\mathrm{d}((\mathcal{R}_1^{'-1}\mathcal{R}_2^{-1})\rhd a)\wedge_\mathcal{R}
    \mathrm{i}^\mathcal{R}_{(\mathcal{R}_2^{'-1}\mathcal{R}_1^{-1})\rhd Y}\omega\\
    =&\mathrm{i}^\mathcal{R}_{\llbracket a,Y\rrbracket_\mathcal{R}}\omega
\end{align*}
for all $\omega\in\Omega^\bullet_\mathcal{R}(\mathcal{A})$
by Lemma~\ref{lemma02}. Using the graded braided Leibniz rule
this extends to any $\ell>1$, namely
\begin{align*}
    [\mathscr{L}^\mathcal{R}_a,
    \mathrm{i}^\mathcal{R}_{Y_1\wedge_\mathcal{R}\cdots
    \wedge_\mathcal{R}Y_\ell}]_\mathcal{R}
    =&[\mathscr{L}^\mathcal{R}_a,\mathrm{i}^\mathcal{R}_{Y_1}]_\mathcal{R}
    \mathrm{i}^\mathcal{R}_{Y_2\wedge_\mathcal{R}\cdots\wedge_\mathcal{R}Y_\ell}
    +(-1)^{(-1)\cdot 1}\mathrm{i}^\mathcal{R}_{\mathcal{R}_1^{-1}\rhd Y_1}
    [\mathscr{L}^\mathcal{R}_{\mathcal{R}_2^{-1}\rhd a},
    \mathrm{i}^\mathcal{R}_{Y_2\wedge_\mathcal{R}\cdots\wedge_\mathcal{R}Y_\ell}]\\
    =&\mathrm{i}^\mathcal{R}_{\llbracket a,Y_1\rrbracket_\mathcal{R}
    \wedge_\mathcal{R}Y_2\wedge_\mathcal{R}\cdots\wedge_\mathcal{R}Y_\ell}
    -\mathrm{i}^\mathcal{R}_{\mathcal{R}_1^{-1}\rhd Y_1}
    [\mathscr{L}^\mathcal{R}_{\mathcal{R}_2^{-1}\rhd a},
    \mathrm{i}^\mathcal{R}_{Y_2\wedge_\mathcal{R}\cdots\wedge_\mathcal{R}Y_\ell}]\\
    =&\cdots
    =\mathrm{i}^\mathcal{R}_{\llbracket a,Y\rrbracket_\mathcal{R}}.
\end{align*}
Again by Lemma~\ref{lemma02} we know that $[\mathscr{L}^\mathcal{R}_X,
\mathrm{i}^\mathcal{R}_Y]_\mathcal{R}=\mathrm{i}^\mathcal{R}_{[X,Y]_\mathcal{R}}$
holds for $\ell=1$ and $X\in\mathfrak{X}^1_\mathcal{R}(\mathcal{A})$.
Using the graded braided Leibniz rule this extends to all 
$Y\in\mathfrak{X}^\bullet_\mathcal{R}(\mathcal{A})$.
Assume now that $[\mathscr{L}^\mathcal{R}_X,\mathrm{i}^\mathcal{R}_Z]_\mathcal{R}
=\mathrm{i}^\mathcal{R}_{\llbracket X,Z\rrbracket_\mathcal{R}}$
holds for all $X\in\mathfrak{X}^k_\mathcal{R}(\mathcal{A})$ and
$Z\in\mathfrak{X}^\bullet_\mathcal{R}(\mathcal{A})$ for a fixed
$k>0$. Then, for all $X\in\mathfrak{X}^k_\mathcal{R}(\mathcal{A})$,
$Y\in\mathfrak{X}^1_\mathcal{R}(\mathcal{A})$ and $Z\in
\mathfrak{X}^m_\mathcal{R}(\mathcal{A})$ it follows that
\begin{align*}
    [\mathscr{L}^\mathcal{R}_{X\wedge_\mathcal{R}Y},\mathrm{i}^\mathcal{R}_Z
    ]_\mathcal{R}
    =&[\mathrm{i}^\mathcal{R}_X\mathscr{L}^\mathcal{R}_Y
    -\mathscr{L}^\mathcal{R}_X\mathrm{i}^\mathcal{R}_Y,
    \mathrm{i}^\mathcal{R}_Z]_\mathcal{R}\\
    =&\mathrm{i}^\mathcal{R}_X[\mathscr{L}^\mathcal{R}_Y,
    \mathrm{i}^\mathcal{R}_Z]_\mathcal{R}
    +[\mathrm{i}^\mathcal{R}_X,
    \mathrm{i}^\mathcal{R}_{\mathcal{R}_1^{-1}\rhd Z}]_\mathcal{R}
    \mathscr{L}^\mathcal{R}_{\mathcal{R}_2^{-1}\rhd Y}\\
    &-\mathscr{L}^\mathcal{R}_X[\mathrm{i}^\mathcal{R}_Y,
    \mathrm{i}^\mathcal{R}_Z]_\mathcal{R}
    -(-1)^m[\mathscr{L}^\mathcal{R}_X,
    \mathrm{i}^\mathcal{R}_{\mathcal{R}_1^{-1}\rhd Z}]_\mathcal{R}
    \mathrm{i}^\mathcal{R}_{\mathcal{R}_2^{-1}\rhd Y}\\
    =&\mathrm{i}^\mathcal{R}_X[\mathscr{L}^\mathcal{R}_Y,
    \mathrm{i}^\mathcal{R}_Z]_\mathcal{R}
    -(-1)^m[\mathscr{L}^\mathcal{R}_X,
    \mathrm{i}^\mathcal{R}_{\mathcal{R}_1^{-1}\rhd Z}]_\mathcal{R}
    \mathrm{i}^\mathcal{R}_{\mathcal{R}_2^{-1}\rhd Y}\\
    =&\mathrm{i}^\mathcal{R}_{X}\mathrm{i}^\mathcal{R}_{
    \llbracket Y,Z\rrbracket_\mathcal{R}}
    -(-1)^{m}\mathrm{i}^\mathcal{R}_{\llbracket X,
    \mathcal{R}_1^{-1}\rhd Z\rrbracket_\mathcal{R}}
    \mathrm{i}^\mathcal{R}_{(\mathcal{R}_2^{-1}\rhd Y)}\\
    =&\mathrm{i}^\mathcal{R}_{X\wedge_\mathcal{R}
    \llbracket Y,Z\rrbracket_\mathcal{R}}
    +(-1)^{m-1}\mathrm{i}^\mathcal{R}_{\llbracket X,
    \mathcal{R}_1^{-1}\rhd Z\rrbracket_\mathcal{R}\wedge_\mathcal{R}
    (\mathcal{R}_2^{-1}\rhd Y)}\\
    =&\mathrm{i}^\mathcal{R}_{\llbracket X\wedge_\mathcal{R}Y,
    Z\rrbracket_\mathcal{R}}
\end{align*}
for all $X\in\mathfrak{X}^k_\mathcal{R}(\mathcal{A})$,
$Y\in\mathfrak{X}^1_\mathcal{R}(\mathcal{A})$ and
$Z\in\mathfrak{X}^m_\mathcal{R}(\mathcal{A})$
using Lemma~\ref{lemma02}. By induction
$[\mathscr{L}^\mathcal{R}_X,\mathrm{i}^\mathcal{R}_Y]_\mathcal{R}
=\mathrm{i}^\mathcal{R}_{\llbracket X,Y\rrbracket_\mathcal{R}}$
for all $X,Y\in\mathfrak{X}^\bullet_\mathcal{R}(\mathcal{A})$.
The remaining formula is verified via
\begin{align*}
    [\mathscr{L}^\mathcal{R}_X,\mathscr{L}^\mathcal{R}_Y]_\mathcal{R}
    =&[\mathscr{L}^\mathcal{R}_X,[\mathrm{i}^\mathcal{R}_Y,
    \mathrm{d}]_\mathcal{R}]_\mathcal{R}\\
    =&[[\mathscr{L}^\mathcal{R}_X,\mathrm{i}^\mathcal{R}_Y]_\mathcal{R},
    \mathrm{d}]_\mathcal{R}
    +(-1)^{(k-1)\ell}[\mathrm{i}^\mathcal{R}_{\mathcal{R}_1^{-1}\rhd Y},
    [\mathscr{L}^\mathcal{R}_{\mathcal{R}_2^{-1}\rhd X},
    \mathrm{d}]_\mathcal{R}]_\mathcal{R}\\
    =&[\mathrm{i}^\mathcal{R}_{\llbracket X,Y\rrbracket_\mathcal{R}},
    \mathrm{d}]_\mathcal{R}+0\\
    =&\mathscr{L}^\mathcal{R}_{\llbracket X,Y\rrbracket_\mathcal{R}}
\end{align*}
for all $X\in\mathfrak{X}^k_\mathcal{R}(\mathcal{A})$ and
$Y\in\mathfrak{X}^\ell_\mathcal{R}(\mathcal{A})$.
This concludes the proof of the theorem.
\end{proof}
In particular, the Cartan calculus on a commutative algebra is a braided
Cartan calculus with respect to the trivial triangular structure and a
(possibly trivial) action of a cocommutative Hopf algebra. We discuss a
further class of examples which is to some extent already present in the
literature, see \cite{Aschieri2006} for $\mathcal{R}=1\otimes 1$ and
\cite{Schenkel2016}~Proposition~3.22. for the first order calculus in the case
of a quasi-triangular Hopf algebra and non-associative algebras.
Fix a triangular Hopf algebra $(H,\mathcal{R})$, a braided commutative
left $H$-module algebra $\mathcal{A}$ and a Drinfel'd twist
$\mathcal{F}$ on $H$ in the following. Recall from Theorem~\ref{thm02}
that the Drinfel'd functor
\begin{equation}
    \mathrm{Drin}_\mathcal{F}\colon
    ({}_\mathcal{A}^H\mathcal{M}_\mathcal{A}^\mathcal{R},\otimes_\mathcal{A},
    c^\mathcal{R})\rightarrow
    ({}_{\mathcal{A}_\mathcal{F}}^{H_\mathcal{F}}\mathcal{M}_{
    \mathcal{A}_\mathcal{F}}^{\mathcal{R}_\mathcal{F}},\otimes_{\mathcal{A}_\mathcal{F}},
    c^{\mathcal{R}_\mathcal{F}})
\end{equation}
is a braided monoidal equivalence of braided monoidal categories
with braided monoidal natural transformation given on objects
$\mathcal{M}$ and $\mathcal{M}'$ of
${}_\mathcal{A}^H\mathcal{M}^\mathcal{R}_\mathcal{A}$ by
$$
\varphi_{\mathcal{M},\mathcal{M}'}\colon
\mathcal{M}_\mathcal{F}\otimes_{\mathcal{A}_\mathcal{F}}\mathcal{M}'_\mathcal{F}
\ni(m\otimes_{\mathcal{A}_\mathcal{F}}m')
\mapsto(\mathcal{F}_1^{-1}\rhd m)\otimes_\mathcal{A}(\mathcal{F}_2^{-1}\rhd m')
\in(\mathcal{M}\otimes_\mathcal{A}\mathcal{M}')_\mathcal{F}.
$$
For
$X\in\mathrm{Der}_\mathcal{R}(\mathcal{A})_\mathcal{F}$ we define a $\Bbbk$-linear map
$X^\mathcal{F}\colon\mathcal{A}\rightarrow\mathcal{A}$ by
\begin{equation}\label{eq07}
    X^\mathcal{F}(a)=(\mathcal{F}_1^{-1}\rhd X)(\mathcal{F}_2^{-1}\rhd a)
    \text{ for all }a\in\mathcal{A}.
\end{equation}
This declares an isomorphism 
$(\mathfrak{X}^1_\mathcal{R}(\mathcal{A}))_\mathcal{F}
\ni X\mapsto X^\mathcal{F}
\in\mathfrak{X}^1_{\mathcal{R}_\mathcal{F}}(\mathcal{A}_\mathcal{F})$
of $H_\mathcal{F}$-equivariant braided symmetric
$\mathcal{A}_\mathcal{F}$-modules. In particular,
\begin{equation}\label{eq08}
    \xi\rhd X^\mathcal{F}=(\xi\rhd X)^\mathcal{F},~
    a\cdot_{\mathcal{R}_\mathcal{F}}X^\mathcal{F}
    =(a\cdot_\mathcal{F}X)^\mathcal{F},~
    X^\mathcal{F}\cdot_{\mathcal{R}_\mathcal{F}}a
    =(X\cdot_\mathcal{F}a)^\mathcal{F}
\end{equation}
for all $\xi\in H$, $a\in\mathcal{A}$ and
$X\in\mathfrak{X}^1_\mathcal{R}(\mathcal{A})_\mathcal{F}$, where we denoted the
$\mathcal{A}_\mathcal{F}$-module actions on 
$\mathfrak{X}^1_{\mathcal{R}_\mathcal{F}}(\mathcal{A}_\mathcal{F})$ by
$\cdot_{\mathcal{R}_\mathcal{F}}$.
We define the \textit{twisted wedge product}
\begin{equation}
    \wedge_\mathcal{F}
    =\mathrm{Drin}_\mathcal{F}(\wedge_\mathcal{R})
    \circ\varphi_{\mathfrak{X}^1_\mathcal{R}(\mathcal{A}),
    \mathfrak{X}^1_\mathcal{R}(\mathcal{A})}\colon
    \mathfrak{X}^1_\mathcal{R}(\mathcal{A})_\mathcal{F}
    \otimes_{\mathcal{A}_\mathcal{F}}
    \mathfrak{X}^1_\mathcal{R}(\mathcal{A})_\mathcal{F}
    \rightarrow\mathfrak{X}^2_\mathcal{R}(\mathcal{A})_\mathcal{F}
\end{equation}
and extend the isomorphism (\ref{eq07})
to higher wedge powers as a homomorphism of the twisted wedge product, i.e.
\begin{equation}\label{eq10}
    (X\wedge_\mathcal{F}Y)^\mathcal{F}
    =X^\mathcal{F}\wedge_{\mathcal{R}_\mathcal{F}}Y^\mathcal{F}
\end{equation}
for all $X,Y\in\mathfrak{X}^\bullet_\mathcal{R}(\mathcal{A})_\mathcal{F}$,
where $\wedge_\mathcal{F}=\mathrm{Drin}_\mathcal{F}(\wedge_\mathcal{R})
\circ\varphi_{\mathfrak{X}^\bullet_\mathcal{R}(\mathcal{A}),
\mathfrak{X}^\bullet_\mathcal{R}(\mathcal{A})}$.
Inductively this leads to an isomorphism 
$
\mathfrak{X}^\bullet_\mathcal{R}(\mathcal{A})_\mathcal{F}\rightarrow
\mathfrak{X}^\bullet_{\mathcal{R}_\mathcal{F}}(\mathcal{A}_\mathcal{F})
$
of $H_\mathcal{F}$-equivariant braided symmetric
$\mathcal{A}_\mathcal{F}$-bimodules.
Also the \textit{twisted Schouten-Nijenhuis bracket}
\begin{equation}
    \llbracket\cdot,\cdot\rrbracket_\mathcal{F}\colon
    \mathrm{Drin}_\mathcal{F}(\llbracket\cdot,\cdot\rrbracket_\mathcal{R})\circ
    \varphi_{\mathfrak{X}^\bullet_\mathcal{R}(\mathcal{A}),
    \mathfrak{X}^\bullet_\mathcal{R}(\mathcal{A})}
    \colon\mathfrak{X}^\bullet_\mathcal{R}(\mathcal{A})_\mathcal{F}
    \otimes_{\mathcal{A}_\mathcal{F}}
    \mathfrak{X}^\bullet_\mathcal{R}(\mathcal{A})_\mathcal{F}
    \rightarrow\mathfrak{X}^\bullet_\mathcal{R}(\mathcal{A})_\mathcal{F}
\end{equation}
can be defined.
On elements
$X,Y\in\mathfrak{X}^\bullet_\mathcal{R}(\mathcal{A})_\mathcal{F}$ the twisted
operations read
\begin{equation}
    X\wedge_\mathcal{F}Y
    =(\mathcal{F}_1^{-1}\rhd X)\wedge_\mathcal{R}(\mathcal{F}_2^{-1}\rhd Y)
    \text{ and }
    \llbracket X,Y\rrbracket_\mathcal{F}
    =\llbracket\mathcal{F}_1^{-1}\rhd X,
    \mathcal{F}_2^{-1}\rhd Y\rrbracket_\mathcal{R},
\end{equation}
respectively.
Similarly we define an isomorphism
${}^\mathcal{F}\colon\Omega^\bullet_\mathcal{R}(\mathcal{A})_\mathcal{F}\rightarrow
\Omega^\bullet_{\mathcal{R}_\mathcal{F}}(\mathcal{A}_\mathcal{F})$ of
$H_\mathcal{F}$-equivariant braided symmetric $\mathcal{A}_\mathcal{F}$-bimodules,
the \textit{twisted Lie derivative} and \textit{twisted insertion}
\begin{equation}
\begin{split}
    \mathscr{L}^\mathcal{F}&\colon
    \mathfrak{X}^\bullet_\mathcal{R}(\mathcal{A})_\mathcal{F}
    \otimes_\mathcal{F}\Omega^\bullet_\mathcal{R}(\mathcal{A})_\mathcal{F}
    \rightarrow\Omega^\bullet_\mathcal{R}(\mathcal{A})_\mathcal{F},~\\
    \mathrm{i}^\mathcal{F}&\colon
    \mathfrak{X}^\bullet_\mathcal{R}(\mathcal{A})_\mathcal{F}
    \otimes_\mathcal{F}\Omega^\bullet_\mathcal{R}(\mathcal{A})_\mathcal{F}
\rightarrow\Omega^\bullet_\mathcal{R}(\mathcal{A})_\mathcal{F},
\end{split}
\end{equation}
while the de Rham differential becomes
$\mathrm{d}\colon\Omega^\bullet_\mathcal{R}(\mathcal{A})_\mathcal{F}
\rightarrow\Omega^{\bullet+1}_\mathcal{R}(\mathcal{A})_\mathcal{F}$ after utilizing
the Drinfel'd functor (see \cite{ThomasPhDThesis} for more details). On elements $X\in\mathfrak{X}^\bullet_\mathcal{R}
(\mathcal{A})_\mathcal{F}$ and $\omega\in\Omega^\bullet_\mathcal{R}
(\mathcal{A})_\mathcal{F}$ 
we obtain
\begin{equation}
    \mathscr{L}^\mathcal{F}_X\omega
    =\mathscr{L}^\mathcal{R}_{\mathcal{F}_1^{-1}\rhd X}
    (\mathcal{F}_2^{-1}\rhd\omega)
    \text{ and }
    \mathrm{i}^\mathcal{F}_X\omega
    =\mathrm{i}^\mathcal{R}_{\mathcal{F}_1^{-1}\rhd X}
    (\mathcal{F}_2^{-1}\rhd\omega),
\end{equation}
while the de Rham differential remains undeformed. We refer to
\begin{equation}
    (\Omega^\bullet_\mathcal{R}(\mathcal{A})_\mathcal{F},\wedge_\mathcal{F},
    \mathscr{L}^\mathcal{F},\mathrm{i}^\mathcal{F},\mathrm{d})
    \text{ and }
    (\mathfrak{X}^\bullet_\mathcal{R}
    (\mathcal{A})_\mathcal{F},\wedge_\mathcal{F},
    \llbracket\cdot,\cdot\rrbracket_\mathcal{F})
\end{equation}
as the \textit{twisted Cartan calculus} with respect to $\mathcal{F}$ and 
$\mathcal{R}$.
\begin{proposition}\label{proposition01}
This assignment
\begin{equation}
    {}^\mathcal{F}\colon
    (\mathfrak{X}^\bullet_\mathcal{R}
    (\mathcal{A})_\mathcal{F},\wedge_\mathcal{F},
    \llbracket\cdot,\cdot\rrbracket_\mathcal{F})
    \rightarrow
    (\mathfrak{X}^\bullet_{\mathcal{R}_\mathcal{F}}(\mathcal{A}_\mathcal{F}),
    \wedge_{\mathcal{R}_\mathcal{F}},
    \llbracket\cdot,\cdot\rrbracket_{\mathcal{R}_\mathcal{F}}),
\end{equation}
defined by eq.(\ref{eq07}) and eq.(\ref{eq10}),
is an isomorphism of braided Gerstenhaber algebras and
the twisted Cartan calculus with respect to $\mathcal{R}$ and
$\mathcal{F}$ is isomorphic to the braided Cartan calculus on
$\mathcal{A}_\mathcal{F}$ with respect
to $\mathcal{R}_\mathcal{F}$ via the isomorphism ${}^\mathcal{F}$.
In particular
\begin{equation}
    (\llbracket X,Y\rrbracket_\mathcal{F})^\mathcal{F}
    =\llbracket X^\mathcal{F},Y^\mathcal{F}\rrbracket_{\mathcal{R}_\mathcal{F}},~
    (\mathscr{L}^\mathcal{F}_X\omega)^\mathcal{F}
    =\mathscr{L}^{\mathcal{R}_\mathcal{F}}_{X^\mathcal{F}}\omega^\mathcal{F},~
    (\mathrm{i}^\mathcal{F}_X\omega)^\mathcal{F}
    =\mathrm{i}^{\mathcal{R}_\mathcal{F}}_{X^\mathcal{F}}\omega^\mathcal{F},~
    (\mathrm{d}\omega)^\mathcal{F}
    =\mathrm{d}\omega^\mathcal{F}
\end{equation}
for all $X,Y\in\mathfrak{X}^\bullet_\mathcal{R}(\mathcal{A})_\mathcal{F}$
and $\omega\in\Omega^\bullet_\mathcal{R}(\mathcal{A})_\mathcal{F}$.
\end{proposition}
\begin{proof}
By the inverse $2$-cocycle property the twisted concatenation of
$X,Y\in\mathrm{Der}(\mathcal{A})_\mathcal{F}$ equals
\begin{align*}
    (X\cdot_\mathcal{F}Y)^\mathcal{F}(a)
    =((\mathcal{F}_{1(1)}^{-1}\mathcal{F}_1^{'-1})\rhd X)
    ((\mathcal{F}_{1(2)}^{-1}\mathcal{F}_2^{'-1})\rhd Y)
    (\mathcal{F}_2^{-1}\rhd a)
    =(X^\mathcal{F}\cdot_{\mathcal{R}_\mathcal{F}}Y^\mathcal{F})(a)
\end{align*}
for all $a\in\mathcal{A}$, where $\cdot_{\mathcal{R}_\mathcal{F}}$ denotes
the concatenation of endomorphisms of $\mathcal{A}_\mathcal{F}$. Then
\begin{align*}
    ([X,Y]_\mathcal{F})^\mathcal{F}
    =&([\mathcal{F}_1^{-1}\rhd X,
    \mathcal{F}_2^{-1}\rhd Y]_\mathcal{R})^\mathcal{F}\\
    =&((\mathcal{F}_1^{-1}\rhd X)
    \cdot_\mathcal{R}(\mathcal{F}_2^{-1}\rhd Y))^\mathcal{F}
    -(((\mathcal{R}_1^{-1}\mathcal{F}_2^{-1})\rhd Y)
    \cdot_\mathcal{R}((\mathcal{R}_2^{-1}\mathcal{F}_1^{-1})\rhd X))^\mathcal{F}\\
    =&(X\cdot_\mathcal{F}Y)^\mathcal{F}
    -((\mathcal{R}_{\mathcal{F}1}^{-1}\rhd Y)
    \cdot_\mathcal{F}(\mathcal{R}_{\mathcal{F}2}^{-1}\rhd X))^\mathcal{F}\\
    =&X^\mathcal{F}\cdot_{\mathcal{R}_\mathcal{F}}Y^\mathcal{F}
    -(\mathcal{R}_{\mathcal{F}1}^{-1}\rhd Y^\mathcal{F})
    \cdot_{\mathcal{R}_\mathcal{F}}
    (\mathcal{R}_{\mathcal{F}2}^{-1}\rhd X^\mathcal{F})\\
    =&[X^\mathcal{F},Y^\mathcal{F}]_{\mathcal{R}_\mathcal{F}}
\end{align*}
where we also employed (\ref{eq08}). Using formula (\ref{eq09}) our previous
computations together with eq.(\ref{eq10}) imply 
$(\llbracket X,Y\rrbracket_\mathcal{F})^\mathcal{F}
=\llbracket X^\mathcal{F},Y^\mathcal{F}\rrbracket_{\mathcal{R}_\mathcal{F}}$
for all $X,Y\in\mathfrak{X}^\bullet_\mathcal{R}(\mathcal{A})_\mathcal{F}$.
The other equations follow similarly. We refer to \cite{ThomasPhDThesis}
for a full proof.
\end{proof}
In other words, the above proposition shows that the twisted Cartan calculus
is gauge equivalent to the untwisted Cartan calculus. Since
the construction of the braided Cartan calculus is determined by the triangular
structure and the twisted Cartan calculus is braided with respect to the
twisted triangular structure our construction respects the
gauge equivalence. In this light twist deformations seem trivial.
On the other hand, there are situations where it is worth to distinguish the
braided Cartan calculus and its twist deformations. Imagine for example a
commutative left $H$-module algebra $\mathcal{A}$ for a cocommutative Hopf algebra
$H$. For a nontrivial twist $\mathcal{F}$ on $H$ the twisted Cartan
calculus is noncommutative while the untwisted one is commutative.
In this sense one might consider the twisted Cartan calculus as a quantization
of the untwisted one even if both are gauge equivalent. This might be interpreted as
a quantization which is in $1$-$1$-correspondence to its classical counterpart.

\subsection{Equivariant Covariant Derivatives and Metrics}

Having the braided Cartan calculus at hand we wonder if other concepts of
differential geometry generalize to this setting. 
Focusing on the algebraic properties of covariant derivatives, namely
function linearity in the first argument and a Leibniz rule in the second
argument, we introduce equivariant covariant derivatives on equivariant
braided symmetric bimodules. Note that there are several notions of covariant
derivatives on noncommutative algebras (see e.g.
\cite{Arnlind2019,Aschieri2010,AsSh14,Aschieri2006,Schenkel2015,Jyotishman2019,D-VM96,GaetanoThomas19,LaMa2012,Peterka2017}).
In particular one has to distinguish between left and right covariant derivatives.
In the spirit of these notes we demand the covariant derivative to be
equivariant in addition, for which the definitions of left and right
covariant derivatives coincide. Curvature and Torsion of equivariant
covariant derivatives are discussed and we extend an equivariant
covariant derivative on the algebra to braided multivector fields and
differential forms. We furthermore give a generalization of metrics 
to the braided commutative setting and prove that there exists a unique
equivariant Levi-Civita covariant derivative for every non-degenerate
equivariant metric.
Fix in the following a triangular Hopf algebra $(H,\mathcal{R})$ and a
braided commutative left $H$-module algebra $\mathcal{A}$.
\begin{definition}[Equivariant covariant derivative]
Consider an $H$-equivariant braided symmetric
$\mathcal{A}$-bimodule $\mathcal{M}$. An $H$-equivariant map
$
\nabla^\mathcal{R}\colon\mathfrak{X}^1_\mathcal{R}(\mathcal{A})
\otimes\mathcal{M}\rightarrow\mathcal{M}
$
is said to be an equivariant covariant derivative on $\mathcal{M}$
with respect to $\mathcal{R}$, if for all
$a\in\mathcal{A}$, $X\in\mathfrak{X}^1_\mathcal{R}(\mathcal{A})$ and
$s\in\mathcal{M}$ one has
\begin{equation}\label{eq11}
    \nabla^\mathcal{R}_{a\cdot X}s
    =a\cdot(\nabla^\mathcal{R}_Xs)
\end{equation}
and
\begin{equation}\label{eq12}
    \nabla^\mathcal{R}_X(a\cdot s)
    =(\mathscr{L}^\mathcal{R}_Xa)\cdot s
    +(\mathcal{R}_1^{-1}\rhd a)\cdot
    (\nabla^\mathcal{R}_{\mathcal{R}_2^{-1}\rhd X}s).
\end{equation}
\end{definition}
Note that $H$-equivariance of a $\Bbbk$-linear map
$\nabla^\mathcal{R}\colon\mathfrak{X}^1_\mathcal{R}(\mathcal{A})
\otimes\mathcal{M}\rightarrow\mathcal{M}$ reads 
$\xi\rhd(\nabla^\mathcal{R}_Xs)
=\nabla^\mathcal{R}_{\xi_{(1)}\rhd X}(\xi_{(2)}\rhd s)$ for all
$\xi\in H$, $X\in\mathfrak{X}^1_\mathcal{R}(\mathcal{A})$ and $s\in\mathcal{M}$.
The \textit{curvature} of an equivariant covariant derivative $\nabla^\mathcal{R}$
on $\mathcal{M}$ is defined by
\begin{equation}
    R^{\nabla^\mathcal{R}}(X,Y)
    =\nabla^\mathcal{R}_X\nabla^\mathcal{R}_Y
    -\nabla^\mathcal{R}_{\mathcal{R}_1^{-1}\rhd Y}
    \nabla^\mathcal{R}_{\mathcal{R}_2^{-1}\rhd X}
    -\nabla^\mathcal{R}_{[X,Y]_\mathcal{R}}
\end{equation}
for $X,Y\in\mathfrak{X}^1_\mathcal{R}(\mathcal{A})$.
If $\mathcal{M}=\mathfrak{X}^1_\mathcal{R}(\mathcal{A})$
we can further define the \textit{torsion} of $\nabla^\mathcal{R}$ by
\begin{equation}
    \mathrm{Tor}^{\nabla^\mathcal{R}}(X,Y)
    =\nabla^\mathcal{R}_XY
    -\nabla^\mathcal{R}_{\mathcal{R}_1^{-1}\rhd Y}(\mathcal{R}_2^{-1}\rhd X)
    -[X,Y]_\mathcal{R},
\end{equation}
for all $X,Y\in\mathfrak{X}^1_\mathcal{R}(\mathcal{A})$. An equivariant
covariant derivative $\nabla^\mathcal{R}$ is \textit{flat} if
$R^{\nabla^\mathcal{R}}=0$ and \textit{torsion-free} if
$\mathrm{Tor}^{\nabla^\mathcal{R}}=0$.
While eq.(\ref{eq11}) and eq.(\ref{eq12}) usually only refer to a left covariant
derivative we prove in the following lemma (c.f. \cite{ThomasPhDThesis})
that in the equivariant setup the
notions of left and right covariant derivatives are equivalent.
\begin{lemma}
Let $\nabla^\mathcal{R}$ be a covariant derivative on an $H$-equivariant
braided symmetric $\mathcal{A}$-bimodule $\mathcal{M}$. Then
for all
$a\in\mathcal{A}$, $X\in\mathfrak{X}^1_\mathcal{R}(\mathcal{A})$ and
$s\in\mathcal{M}$,
\begin{equation}\label{eq13}
    \nabla^\mathcal{R}_{X\cdot a}s
    =(\nabla^\mathcal{R}_X(\mathcal{R}_1^{-1}\rhd s))
    \cdot(\mathcal{R}_2^{-1}\rhd a)
\end{equation}
and
\begin{equation}\label{eq14}
    \nabla^\mathcal{R}_X(s\cdot a)
    =(\nabla^\mathcal{R}_Xs)\cdot a
    +(\mathcal{R}_1^{-1}\rhd s)\cdot
    (\mathscr{L}^\mathcal{R}_{\mathcal{R}_2^{-1}\rhd X}a)
\end{equation}
hold. On the other hand, every $H$-equivariant map
$\nabla^\mathcal{R}\colon\mathfrak{X}^1_\mathcal{R}(\mathcal{A})\otimes
\mathcal{M}\rightarrow\mathcal{M}$ satisfying eq.(\ref{eq13}) and
eq.(\ref{eq14}) is an equivariant covariant derivative on $\mathcal{M}$.
\end{lemma}
There are natural extensions of an equivariant covariant derivative
$\nabla^\mathcal{R}\colon\mathfrak{X}^1_\mathcal{R}(\mathcal{A})\otimes
\mathfrak{X}^1_\mathcal{R}(\mathcal{A})\rightarrow
\mathfrak{X}^1_\mathcal{R}(\mathcal{A})$ to
braided multivector fields and differential forms in analogy
to differential geometry. We define the \textit{braided dual
pairing} $\langle\cdot,\cdot\rangle_\mathcal{R}\colon
\Omega^1_\mathcal{R}(\mathcal{A})\otimes
\mathfrak{X}^1_\mathcal{R}(\mathcal{R})\rightarrow\mathcal{A}$ by
$\langle\omega,X\rangle_\mathcal{R}=\omega(X)$
for all $\omega\in\Omega^1_\mathcal{R}
(\mathcal{A})$ and $X\in\mathfrak{X}^1_\mathcal{R}(\mathcal{A})$. It is
$H$-equivariant,
left $\mathcal{A}$-linear in the first and right $\mathcal{A}$-linear in the
second argument.
\begin{proposition}\label{prop02}
An equivariant covariant derivative $\nabla^\mathcal{R}$ on
$\mathfrak{X}^1_\mathcal{R}(\mathcal{A})$ induces an equivariant covariant derivative
$\tilde{\nabla}^\mathcal{R}$ on $\Omega^1_\mathcal{R}(\mathcal{A})$ via
\begin{equation}
    \langle\tilde{\nabla}^\mathcal{R}_X\omega,Y\rangle_\mathcal{R}
    =\mathscr{L}^\mathcal{R}_X\langle\omega,Y\rangle_\mathcal{R}
    -\langle\mathcal{R}_1^{-1}\rhd\omega,
    \nabla^\mathcal{R}_{\mathcal{R}_2^{-1}\rhd X}Y\rangle_\mathcal{R}
\end{equation}
for all $X,Y\in\mathfrak{X}^1_\mathcal{R}(\mathcal{A})$ and
$\omega\in\Omega^1_\mathcal{R}(\mathcal{A})$. Moreover, $\nabla^\mathcal{R}$
and $\tilde{\nabla}^\mathcal{R}$ can be extended
as braided derivations to equivariant covariant
derivatives on $\mathfrak{X}^\bullet_\mathcal{R}(\mathcal{A})$ and
$\Omega^\bullet_\mathcal{R}(\mathcal{A})$, respectively.
\end{proposition}
\begin{proof}
Let $X,Y\in\mathfrak{X}^1_\mathcal{R}(\mathcal{A})$,
$\omega\in\Omega^1_\mathcal{R}(\mathcal{A})$ and
$a\in\mathcal{A}$. Then $\tilde{\nabla}^\mathcal{R}_X\omega
\in\Omega^1_\mathcal{R}(\mathcal{A})$ is well-defined, since
\begin{align*}
    \langle\tilde{\nabla}^\mathcal{R}_X\omega,Y\cdot a\rangle_\mathcal{R}
    =&\mathscr{L}^\mathcal{R}_X\langle\omega,Y\cdot a\rangle_\mathcal{R}
    -\langle\mathcal{R}_1^{-1}\rhd\omega,
    \nabla^\mathcal{R}_{\mathcal{R}_2^{-1}\rhd X}(Y\cdot a)\rangle_\mathcal{R}\\
    =&(\mathscr{L}^\mathcal{R}_X\langle\omega,Y\rangle_\mathcal{R})\cdot a
    +(\mathcal{R}_1^{-1}\rhd\langle\omega,Y\rangle_\mathcal{R})
    \cdot\mathscr{L}^\mathcal{R}_{\mathcal{R}_2^{-1}\rhd X}a\\
    &-\langle\mathcal{R}_1^{-1}\rhd\omega,
    (\nabla^\mathcal{R}_{\mathcal{R}_2^{-1}\rhd X}Y)\cdot a
    +(\mathcal{R}_1^{'-1}\rhd Y)\cdot
    \mathscr{L}^\mathcal{R}_{\mathcal{R}_2^{'-1}
    \mathcal{R}_2^{-1}\rhd X}a\rangle_\mathcal{R}\\
    =&\langle\tilde{\nabla}^\mathcal{R}_X\omega,Y\rangle_\mathcal{R}\cdot a.
\end{align*}
Similarly one proves that $\tilde{\nabla}^\mathcal{R}$ is left $\mathcal{A}$-linear
in the first argument and satisfies the braided Leibniz rule in the second
argument. For another $\eta\in\Omega^1_\mathcal{R}(\mathcal{A})$ one verifies that
\begin{equation}
    \tilde{\nabla}^\mathcal{R}_X(\omega\wedge_\mathcal{R}\eta)
    =\tilde{\nabla}^\mathcal{R}_X\omega\wedge_\mathcal{R}\eta
    +(\mathcal{R}_1^{-1}\rhd\omega)\wedge_\mathcal{R}
    \tilde{\nabla}^\mathcal{R}_{\mathcal{R}_2^{-1}\rhd X}\eta
\end{equation}
defines an equivariant
covariant derivative on $\Omega^2_\mathcal{R}(\mathcal{A})$ and inductively
$\tilde{\nabla}^\mathcal{R}$ extends as a braided derivation of $\wedge_\mathcal{R}$
to $\Omega^\bullet_\mathcal{R}(\mathcal{A})$. The extension of $\nabla^\mathcal{R}$
to braided multivector fields is entirely similar.
\end{proof}
Let $\nabla^\mathcal{R}\colon\mathfrak{X}^1_\mathcal{R}(\mathcal{A})
\otimes\mathcal{M}\rightarrow\mathcal{M}$ be an equivariant covariant derivative
with respect to $\mathcal{R}$ on an $H$-equivariant braided symmetric
$\mathcal{A}$-bimodule $\mathcal{M}$.
For any twist $\mathcal{F}$ on $H$ we define the
\textit{twisted equivariant covariant derivative}
\begin{equation}
    \nabla^\mathcal{F}=
    \mathrm{Drin}_\mathcal{F}(\nabla^\mathcal{R})\circ
    \varphi_{\mathfrak{X}^1_\mathcal{R}(\mathcal{A}),
    \mathcal{M}}\colon\mathfrak{X}^1_\mathcal{R}(\mathcal{A})_\mathcal{F}
    \otimes_\mathcal{F}\mathcal{M}_\mathcal{F}\rightarrow\mathcal{M}_\mathcal{F},
\end{equation}
which reads
\begin{equation}
    \nabla^\mathcal{F}_Xs
    =\nabla^\mathcal{R}_{\mathcal{F}_1^{-1}\rhd X}(\mathcal{F}_2^{-1}\rhd s).
\end{equation}
on elements $X\in\mathfrak{X}^1_\mathcal{R}(\mathcal{A})_\mathcal{F}$ and
$s\in\mathcal{M}_\mathcal{F}$.
\begin{proposition}
The twisted equivariant covariant derivative is an equivariant covariant
derivative with respect to the twisted triangular structure, where we
identify $\mathfrak{X}^1_\mathcal{R}(\mathcal{A})_\mathcal{F}$
with $\mathfrak{X}^1_{\mathcal{R}_\mathcal{F}}(\mathcal{A}_\mathcal{F})$
according to Proposition~\ref{proposition01}.
\end{proposition}
\begin{proof}
Let $\xi\in H$, $a\in\mathcal{A}$, $X\in\mathfrak{X}^1_\mathcal{R}(\mathcal{A})_\mathcal{F}$
and $s\in\mathcal{M}_\mathcal{F}$. Then
\begin{align*}
    \xi\rhd(\nabla^\mathcal{F}_Xs)
    =\nabla^\mathcal{R}_{(\xi_{(1)}\mathcal{F}_1^{-1})\rhd X}
    ((\xi_{(2)}\mathcal{F}_2^{-1})\rhd s)
    =\nabla^\mathcal{F}_{\xi_{\widehat{(1)}}\rhd X}(\xi_{\widehat{(2)}}\rhd s)
\end{align*}
shows that $\nabla^\mathcal{F}$ is $H_\mathcal{F}$-equivariant, while
\begin{align*}
    \nabla^\mathcal{F}_{a\cdot_\mathcal{F}X}s
    =&((\mathcal{F}_{1(1)}^{-1}\mathcal{F}_1^{'-1})\rhd a)
    \cdot(\nabla^\mathcal{R}_{(\mathcal{F}_{1(2)}^{-1}\mathcal{F}_2^{'-1})\rhd X}
    (\mathcal{F}_2^{-1}\rhd s))\\
    =&(\mathcal{F}_{1}^{-1}\rhd a)
    \cdot(\nabla^\mathcal{R}_{(\mathcal{F}_{2(1)}^{-1}\mathcal{F}_1^{'-1})\rhd X}
    ((\mathcal{F}_{2(2)}^{-1}\mathcal{F}_2^{'-1})\rhd s))\\
    =&(\mathcal{F}_{1}^{-1}\rhd a)
    \cdot(\mathcal{F}_2^{-1}\rhd(
    \nabla^\mathcal{R}_{\mathcal{F}_1^{'-1}\rhd X}
    (\mathcal{F}_2^{'-1}\rhd s)))\\
    =&a\cdot_\mathcal{F}(\nabla^\mathcal{F}_Xs)
\end{align*}
and
\begin{align*}
    \nabla^\mathcal{F}_X(a\cdot_\mathcal{F}s)
    =&\nabla^\mathcal{R}_{\mathcal{F}_1^{-1}\rhd X}(
    ((\mathcal{F}_{2(1)}^{-1}\mathcal{F}_1^{'-1})\rhd a)
    \cdot((\mathcal{F}_{2(1)}^{-1}\mathcal{F}_1^{'-1})\rhd a))\\
    =&(\mathscr{L}^\mathcal{R}_{\mathcal{F}_1^{-1}\rhd X}
    ((\mathcal{F}_{2(1)}^{-1}\mathcal{F}_1^{'-1})\rhd a))
    \cdot((\mathcal{F}_{2(2)}^{-1}\mathcal{F}_2^{'-1})\rhd s)\\
    &+((\mathcal{R}_1^{-1}\mathcal{F}_{2(1)}^{-1}\mathcal{F}_1^{'-1})\rhd a)
    \cdot(\nabla^\mathcal{R}_{(\mathcal{R}_2^{-1}\mathcal{F}_1^{-1})\rhd X}
    ((\mathcal{F}_{2(2)}^{-1}\mathcal{F}_2^{'-1})\rhd s))\\
    =&(\mathcal{F}_1^{-1}\rhd(
    \mathscr{L}^\mathcal{R}_{\mathcal{F}_1^{'-1}\rhd X}
    (\mathcal{F}_2^{'-1}\rhd a)))
    \cdot(\mathcal{F}_{2}^{-1}\rhd s)\\
    &+((\mathcal{F}_{2(1)}^{-1}\mathcal{R}_1^{-1}\mathcal{F}_2^{'-1})\rhd a)
    \cdot(\nabla^\mathcal{R}_{(\mathcal{F}_{2(2)}^{-1}\mathcal{R}_2^{-1}
    \mathcal{F}_1^{'-1})\rhd X}
    (\mathcal{F}_{2}^{-1}\rhd s))\\
    =&(\mathscr{L}^\mathcal{F}_Xa)\cdot_\mathcal{F}s
    +((\mathcal{F}_{1(1)}^{-1}\mathcal{F}_1^{-1}\mathcal{R}_{\mathcal{F}1}^{-1})
    \rhd a)\cdot
    (\nabla^\mathcal{R}_{(\mathcal{F}_{1(2)}^{-1}
    \mathcal{F}_2^{-1}\mathcal{R}_{\mathcal{F}2}^{-1})
    \rhd X}(\mathcal{F}_2^{-1}\rhd s))\\
    =&(\mathscr{L}^\mathcal{F}_Xa)\cdot_\mathcal{F}s
    +(\mathcal{R}_{\mathcal{F}1}^{-1}\rhd a)\cdot_\mathcal{F}
    (\nabla^\mathcal{F}_{\mathcal{R}_{\mathcal{F}2}\rhd X}s)
\end{align*}
are the correct linearity properties, proving that $\nabla^\mathcal{F}$
is an equivariant covariant derivative with respect to $\mathcal{R}_\mathcal{F}$.
\end{proof}
In Riemannian geometry covariant derivatives are always considered together
with a Riemannian metric. We want to generalize them to braided commutative
algebras: a $\Bbbk$-linear map
${\bf g}\colon\mathfrak{X}^1_\mathcal{R}(\mathcal{A})
\otimes_\mathcal{A}\mathfrak{X}^1_\mathcal{R}(\mathcal{A})
\rightarrow\mathcal{A}$, which is left $\mathcal{A}$-linear in the first argument
and $H$-equivariant,
is said to be an \textit{equivariant metric} if it is \textit{braided symmetric},
i.e. if ${\bf g}(Y,X)
={\bf g}(\mathcal{R}_1^{-1}\rhd X,\mathcal{R}_2^{-1}\rhd Y)$ for all
$X,Y\in\mathfrak{X}^1_\mathcal{R}(\mathcal{A})$.
It follows that ${\bf g}$ is braided right $\mathcal{A}$-linear in the first argument
as well as right $\mathcal{A}$-linear and braided left $\mathcal{A}$-linear in the
second argument. An equivariant metric
is said to be \textit{non-degenerate}
if ${\bf g}(X,Y)=0$ for all $Y\in\mathfrak{X}^1_\mathcal{R}(\mathcal{A})$ implies
$X=0$, it is said to be \textit{strongly non-degenerate}
if ${\bf g}(X,X)\neq 0$ for all $X\neq 0$ and it is said to be \textit{Riemannian} if
it is strongly non-degenerate and there is a partial order $\geq$ on $\mathcal{A}$
such that ${\bf g}(X,X)\geq 0$ for all $X\neq 0$ in addition. Note that strongly non-degeneracy implies non-degeneracy.
An equivariant covariant derivative $\nabla^\mathcal{R}\colon
\mathfrak{X}^1_\mathcal{R}(\mathcal{A})\otimes\mathfrak{X}^1_\mathcal{R}(\mathcal{A})
\rightarrow\mathfrak{X}^1_\mathcal{R}(\mathcal{A})$ on $\mathcal{A}$ is said to be a
\textit{metric equivariant covariant derivative} with respect to an equivariant metric
${\bf g}$, if
\begin{equation}\label{eq15}
    \mathscr{L}^\mathcal{R}_X({\bf g}(Y,Z))
    ={\bf g}(\nabla^\mathcal{R}_XY,Z)
    +{\bf g}(\mathcal{R}_1^{-1}\rhd Y,\nabla^\mathcal{R}_{\mathcal{R}_2^{-1}\rhd X}Z)
\end{equation}
holds for all $X,Y,Z\in\mathfrak{X}^1_\mathcal{R}(\mathcal{A})$.
Note that equivariance 
$\xi\rhd{\bf g}(X,Y)={\bf g}(\xi_{(1)}\rhd X,\xi_{(2)}\rhd Y)$ for all
$\xi\in H$ and $X,Y\in\mathfrak{X}^1_\mathcal{R}(\mathcal{A})$,
of a metric is a quite strong requirement. Similar approaches
which omit this condition are e.g. \cite{Aschieri2010,Aschieri2006,GaetanoThomas19}.
\begin{lemma}\label{lemma04}
Let ${\bf g}$ be a non-degenerate equivariant metric on $\mathcal{A}$.
Then there is a unique
torsion-free metric equivariant covariant derivative on $\mathcal{A}$.
\end{lemma}
\begin{proof}
Fix an equivariant metric ${\bf g}$ on $\mathcal{A}$. Any equivariant covariant derivative
$\nabla^\mathcal{R}$ on $\mathcal{A}$, which is torsion free and metric with
respect to ${\bf g}$, satisfies
\begin{equation}
\begin{split}
    2{\bf g}(\nabla^\mathcal{R}_XY,Z)
    =&X({\bf g}(Y,Z))
    +(\mathcal{R}_{1(1)}^{-1}\rhd Y)({\bf g}(\mathcal{R}_{1(2)}^{-1}\rhd Z,
    \mathcal{R}_2^{-1}\rhd X))\\
    &-(\mathcal{R}_{1}^{-1}\rhd Z)({\bf g}(\mathcal{R}_{2(1)}^{-1}\rhd X,
    \mathcal{R}_{2(2)}^{-1}\rhd Y))\\
    &-{\bf g}(X,[Y,Z]_\mathcal{R})
    +{\bf g}(\mathcal{R}_{1(1)}^{-1}\rhd Y,
    [\mathcal{R}_{1(2)}^{-1}\rhd Z,\mathcal{R}_2^{-1}\rhd X]_\mathcal{R})\\
    &+{\bf g}(\mathcal{R}_{1}^{-1}\rhd Z,
    [\mathcal{R}_{2(1)}^{-1}\rhd X,\mathcal{R}_{2(2)}^{-1}\rhd Y]_\mathcal{R})
\end{split}
\end{equation}
for all $X,Y,Z\in\mathfrak{X}^1_\mathcal{R}(\mathcal{A})$. In particular, this
shows the uniqueness of a torsion-free equivariant covariant derivative which is metric
with respect to ${\bf g}$, if ${\bf g}$ is non-degenerate. It remains to prove that
a $\Bbbk$-bilinear map $\nabla^\mathcal{R}$ determined by the above formula
is a metric torsion-free equivariant covariant derivative. This follows by the
(braided) linearity properties of ${\bf g}$ and the braided Leibniz rule. A full proof
can be found in \cite{ThomasPhDThesis}.
\end{proof}
The unique torsion-free metric equivariant covariant derivative on
$(\mathcal{A},{\bf g})$ is said to be
the \textit{equivariant Levi-Civita covariant derivative}. We want to remark that
Lemma~\ref{lemma04} admits a generalization in the sense that for any value
of the torsion there exists a unique metric equivariant covariant derivative.
As a last observation of this section we prove that
the twist deformation of an equivariant metric is an equivariant metric on the
twisted algebra and the assignment $\mathrm{LC}\colon{\bf g}
\mapsto\nabla^{\mathrm{LC}}$, attributing to a non-degenerate equivariant metric
its equivariant Levi-Civita covariant derivative, respects the Drinfel'd functor.
\begin{corollary}
Let ${\bf g}$ be an equivariant metric on $\mathcal{A}$. Then, the
twisted equivariant metric ${\bf g}_\mathcal{F}$, which is defined by
\begin{equation}
    {\bf g}_\mathcal{F}(X,Y)
    ={\bf g}(\mathcal{F}_1^{-1}\rhd X,\mathcal{F}_2^{-1}\rhd Y)
\end{equation}
for all $X,Y\in\mathfrak{X}^1_\mathcal{R}(\mathcal{A})$, is an equivariant metric
with respect to $\mathcal{R}_\mathcal{F}$ on $\mathcal{A}_\mathcal{F}$.
Moreover, assuming that ${\bf g}$ and ${\bf g}_\mathcal{F}$ are non-degenerate,
twisting the equivariant Levi-Civita covariant derivative with
respect to ${\bf g}$ leads to the equivariant Levi-Civita covariant derivative with
respect to ${\bf g}_\mathcal{F}$.
\end{corollary}
\begin{proof}
One immediately verifies that ${\bf g}_\mathcal{F}$ is an
$H_\mathcal{F}$-equivariant left $\mathcal{A}_\mathcal{F}$-linear map 
$\mathfrak{X}^1_\mathcal{R}(\mathcal{A})_\mathcal{F}
\otimes_{\mathcal{A}_\mathcal{F}}
\mathfrak{X}^1_\mathcal{R}(\mathcal{A})_\mathcal{F}
\rightarrow\mathcal{A}_\mathcal{F}$ which is braided symmetric with respect to
$\mathcal{R}_\mathcal{F}$. Via the identification
$\mathfrak{X}^1_\mathcal{R}(\mathcal{A})_\mathcal{F}\cong
\mathfrak{X}^1_{\mathcal{R}_\mathcal{F}}(\mathcal{A}_\mathcal{F})$ of
Proposition~\ref{proposition01} ${\bf g}_\mathcal{F}$ becomes an equivariant
metric on $\mathcal{A}_\mathcal{F}$.
Let $X,Y,Z\in\mathfrak{X}_\mathcal{R}^1(\mathcal{A})$ and denote the equivariant
Levi-Civita covariant derivative of ${\bf g}$ by $\nabla^\mathcal{R}$.
From eq.(\ref{eq15}) it follows that the twisted equivariant covariant derivative
$\nabla^\mathcal{F}$ satisfies
\begin{equation}
    \mathscr{L}^\mathcal{F}_X({\bf g}_\mathcal{F}(Y,Z))
    ={\bf g}_\mathcal{F}(\nabla^\mathcal{F}_XY,Z)
    +{\bf g}_\mathcal{F}(\mathcal{R}_{\mathcal{F}1}^{-1}\rhd Y,
    \nabla^\mathcal{F}_{\mathcal{R}_{\mathcal{F}2}^{-1}\rhd X}Z).
\end{equation}
Since $\mathrm{Tor}^{\nabla^\mathcal{R}}=0$ we obtain
$\mathrm{Tor}^{\nabla^\mathcal{F}}=0$, proving that $\nabla^\mathcal{F}$
is the unique equivariant Levi-Civita covariant derivative corresponding to
${\bf g}_\mathcal{F}$ if the latter is non-degenerate.
\end{proof}

\section{Submanifolds in Braided Commutative Geometry}\label{section4}

In this section we show that the braided Cartan calculus 
is compatible with the concept of submanifold algebras if the triangular
Hopf algebra respects the corresponding submanifold ideal. This can be
understood as a construction of new examples of braided Cartan calculi
from known ones. The projection to submanifold algebras respects
Drinfel'd twist gauge equivalence classes, which is an interesting
supplement to Proposition~\ref{proposition01}. 
The second subsection is devoted to
the study of equivariant covariant derivatives on submanifold algebras. Depending
on the choice of a strongly non-degenerate equivariant metric one is able to
project equivariant covariant derivatives as well as curvature and torsion if
the submanifold algebra obeys two mild axioms.
Furthermore, the notion of twisted equivariant covariant derivative and metric
are compatible with the projections. 
While the main Section~\ref{section3} stands out with quite an amount
of details, we are relatively short-spoken in the present section. The interested
reader is relegated to \cite{ThomasPhDThesis} for a more circumstantial 
discussion.
A different approach to Riemannian geometry on noncommutative submanifolds,
based on the choice of a finite-dimensional
Lie subalgebra $\mathfrak{g}$ of $\mathrm{Der}(\mathcal{A})$
and a vector space homomorphism $\mathfrak{g}\rightarrow\mathcal{M}$
into a right $\mathcal{A}$-module $\mathcal{M}$, is considered in
\cite{Arnlind2019}.
Yet another approach to noncommutative (fuzzy)
submanifolds $S$ of $\mathbb{R}^n$,
based on the imposition of an energy cutoff on a quantum particle
in $\mathbb{R}^n$, subject to a confining potential with a very sharp
minimum on $S$,
has been recently proposed and applied to spheres in \cite{FiorePisacane2018}.

\subsection{Braided Cartan Calculi on Submanifolds}

In noncommutative geometry there is a well-known notion of submanifold
ideal (c.f. \cite{Masson1995}) generalizing the concept of closed embedded
smooth submanifolds. We refer to \cite{Francesco2019} for a recent discussion of 
submanifold algebras. In braided commutative geometry the submanifold ideals
have to be respected by the Hopf algebra action in order to inherit a braided
symmetry on the quotient algebra. We continue by describing the braided Cartan
calculus on the submanifold algebra in this situation. 
It is given by the projection of the braided Cartan calculus of the ambient
algebra. Moreover, the Drinfel'd functor intertwines the submanifold algebra
projections. The following
discussion is also motivated by \cite{Fiore1997,Fiore2010}.

Fix a triangular Hopf algebra $(H,\mathcal{R})$ and a
braided commutative left $H$-module algebra $\mathcal{A}$.
For any algebra ideal $\mathcal{C}\subseteq\mathcal{A}$ the coset space
$\mathcal{A}/\mathcal{C}$ becomes an algebra with unit and product induced
from $\mathcal{A}$. The elements of $\mathcal{A}/\mathcal{C}$ are equivalence
classes of elements in $\mathcal{A}$, where $a,b\in\mathcal{A}$ are identified
if and only if there exists an element $c\in\mathcal{C}$ such that $a=b+c$. The
corresponding surjective projection is denoted by
$\mathrm{pr}\colon\mathcal{A}\ni a\mapsto a+\mathcal{C}\in\mathcal{A}/\mathcal{C}$.
If the left $H$-module action respects $\mathcal{C}$, i.e. if
$H\rhd\mathcal{C}\subseteq\mathcal{C}$, the quotient $\mathcal{A}/\mathcal{C}$
is a braided commutative left $H$-module algebra with respect to 
$\mathcal{R}$ and the
left $H$-action defined by $\xi\rhd\mathrm{pr}(a)=\mathrm{pr}(\xi\rhd a)$ for
all $a\in\mathcal{A}$. Braided vector fields on the braided commutative
algebra $\mathcal{A}/\mathcal{C}$ can be obtained as projections from a
certain class of braided vector fields on $\mathcal{A}$.
A braided derivation $X\in\mathrm{Der}_\mathcal{R}
(\mathcal{A})$ is said to be \textit{tangent to $\mathcal{C}$} if
$X(\mathcal{C})\subseteq\mathcal{C}$. The $\Bbbk$-module of all braided derivations
of $\mathcal{A}$ which are tangent to $\mathcal{C}$ is denoted by
$\mathfrak{X}^1_t(\mathcal{A})$. It is a braided Lie subalgebra and an
$H$-equivariant braided symmetric $\mathcal{A}$-sub-bimodule of
$\mathfrak{X}^1_\mathcal{R}(\mathcal{A})$. Consider the $\Bbbk$-linear map
\begin{equation}\label{eq16}
    \mathrm{pr}\colon\mathfrak{X}^1_t(\mathcal{A})\rightarrow
    \mathrm{Der}_\mathcal{R}(\mathcal{A}/\mathcal{C}),
\end{equation}
defined for any $X\in\mathfrak{X}^1_t(\mathcal{A})$ by
$\mathrm{pr}(X)(\mathrm{pr}(a))=\mathrm{pr}(X(a))$
for all $a\in\mathcal{A}$.
Note that there are left and right 
$\mathcal{A}/\mathcal{C}$-actions and
an $H$-action on the image
of (\ref{eq16}) defined by
\begin{equation}
    \xi\rhd\mathrm{pr}(X)
    =\mathrm{pr}(\xi\rhd X),~
    \mathrm{pr}(a)\cdot\mathrm{pr}(X)
    =\mathrm{pr}(a\cdot X),~
    \mathrm{pr}(X)\cdot\mathrm{pr}(a)
    =\mathrm{pr}(X\cdot a)
\end{equation}
for all $\xi\in H$, $a\in\mathcal{A}$ and
$X\in\mathfrak{X}^1_t(\mathcal{A})$. Those structure the image
$\mathrm{pr}(\mathfrak{X}^1_t(\mathcal{A}))
\subseteq\mathfrak{X}^1_\mathcal{R}(\mathcal{A})$ as an $H$-equivariant
braided symmetric $\mathcal{A}/\mathcal{C}$-sub-bimodule and braided Lie subalgebra.
On the other hand $\mathrm{pr}(\mathfrak{X}^1_t(\mathcal{A}))$ can be viewed
as an $H$-equivariant braided symmetric $\mathcal{A}$-bimodule with
$\mathcal{A}$-actions $a\cdot\mathrm{pr}(X)=\mathrm{pr}(a)\cdot\mathrm{pr}(X)$
and $\mathrm{pr}(X)\cdot a=\mathrm{pr}(X)\cdot\mathrm{pr}(a)$ for all
$a\in\mathcal{A}$ and $X\in\mathfrak{X}^1_t(\mathcal{A})$.
With respect to the latter structures
(\ref{eq16}) becomes a homomorphism of $H$-equivariant braided symmetric
$\mathcal{A}$-bimodules and braided Lie algebras. It follows that the kernel
$\mathfrak{X}^1_0(\mathcal{A})$ of (\ref{eq16}) is an $H$-equivariant braided
symmetric $\mathcal{A}$-sub-bimodule and a braided Lie ideal of
$\mathfrak{X}^1_t(\mathcal{A})$.
\begin{definition}
An algebra ideal $\mathcal{C}\subseteq\mathcal{A}$
is said to be a \textit{submanifold ideal} and the quotient $\mathcal{A}/\mathcal{C}$
is said to be a \textit{submanifold algebra}
if there is a short exact sequence
\begin{equation}\label{eq17}
    0\rightarrow\mathfrak{X}^1_0(\mathcal{A})\rightarrow\mathfrak{X}^1_t(\mathcal{A})
    \xrightarrow{\mathrm{pr}}\mathrm{Der}_\mathcal{R}(\mathcal{A}/\mathcal{C})
    \rightarrow 0
\end{equation}
of $H$-equivariant braided symmetric $\mathcal{A}$-bimodules
and braided Lie algebras.
\end{definition}
Fix a submanifold ideal $\mathcal{C}$ of $\mathcal{A}$ in the following. The
short exact sequence (\ref{eq17})
extends to a short exact sequence
\begin{equation}
    0\rightarrow\mathfrak{X}^\bullet_0(\mathcal{A})
    \rightarrow\mathfrak{X}^\bullet_t(\mathcal{A})
    \xrightarrow{\mathrm{pr}}
    \mathfrak{X}^\bullet_\mathcal{R}(\mathcal{A}/\mathcal{C})
    \rightarrow 0
\end{equation}
of braided Gerstenhaber algebras
by defining inductively $\mathrm{pr}(X\wedge_\mathcal{R}Y)
=(\mathrm{pr}(X))\wedge_\mathcal{R}(\mathrm{pr}(Y))$ for all
$X,Y\in\mathfrak{X}^\bullet_t(\mathcal{A})$,
where $\mathfrak{X}^\bullet_0(\mathcal{A})$ and
$\mathfrak{X}^\bullet_t(\mathcal{A})$ denote the braided exterior
algebras of $\mathfrak{X}^1_0(\mathcal{A})$ and $\mathfrak{X}^1_t(\mathcal{A})$,
respectively. In particular
$\mathrm{pr}(\llbracket X,Y\rrbracket_\mathcal{R})
=\llbracket\mathrm{pr}(X),\mathrm{pr}(Y)\rrbracket_\mathcal{R}$ holds
for all $X,Y\in\mathfrak{X}^\bullet_t(\mathcal{A})$.
For braided differential forms $\omega=a_0\cdot\mathrm{d}a_1\wedge_\mathcal{R}
\cdots\wedge_\mathcal{R}\mathrm{d}a_n
\in\Omega^\bullet_\mathcal{R}(\mathcal{A})$ one defines
\begin{equation}
    \mathrm{pr}(\omega)
    =\mathrm{pr}(a_0)\mathrm{d}(\mathrm{pr}(a_1))\wedge_\mathcal{R}\cdots
    \wedge_\mathcal{R}\mathrm{d}(\mathrm{pr}(a_n)),
\end{equation}
leading to a short exact sequence of differential graded algebras
\begin{equation}
    0\rightarrow\mathrm{ker}(\mathrm{pr})
    \rightarrow\Omega^\bullet_\mathcal{R}(\mathcal{A})
    \xrightarrow{\mathrm{pr}}\Omega^\bullet_\mathcal{R}(\mathcal{A}/\mathcal{C})
    \rightarrow 0,
\end{equation}
where $\mathrm{ker}(\mathrm{pr})
=\bigoplus_{k\geq 0}\mathrm{ker}(\mathrm{pr})^k$ is defined recursively
by $\mathrm{ker}(\mathrm{pr})^0=\mathcal{C}$ and
\begin{equation}
    \mathrm{ker}(\mathrm{pr})^{k+1}
    =\{\omega\in\Omega^{k+1}_\mathcal{R}(\mathcal{A})~|~
    \mathrm{i}^\mathcal{R}_X\omega\in\mathrm{ker}(\mathrm{pr})^k
    \text{ for all }X\in\mathfrak{X}^1_t(\mathcal{A})\}
\end{equation}
for $k\geq 0$.
As in the case of $\mathfrak{X}^1_\mathcal{R}(\mathcal{A}/\mathcal{C})$,
the projected actions, defined by intertwining the projections, structure
$\Omega^\bullet_\mathcal{R}(\mathcal{A}/\mathcal{C})$ as an object
in ${}^H_{\mathcal{A}/\mathcal{C}}
\mathcal{M}^\mathcal{R}_{\mathcal{A}/\mathcal{C}}$. 
\begin{theorem}\label{lemma03}
The braided Cartan calculus on $\mathcal{A}/\mathcal{C}$ is the projection
of the braided Cartan calculus on $\mathcal{A}$. Namely,
\begin{equation}\label{eq18}
    \mathscr{L}^\mathcal{R}_{\mathrm{pr}(X)}\mathrm{pr}(\omega)
    =\mathrm{pr}(\mathscr{L}^\mathcal{R}_X\omega),~
    \mathrm{i}^\mathcal{R}_{\mathrm{pr}(X)}\mathrm{pr}(\omega)
    =\mathrm{pr}(\mathrm{i}^\mathcal{R}_X\omega)
    \text{ and }
    \mathrm{d}(\mathrm{pr}(\omega))=\mathrm{pr}(\mathrm{d}\omega)
\end{equation}
for all $X\in\mathfrak{X}^\bullet_t(\mathcal{A})$ and 
$\omega\in\Omega^\bullet_\mathcal{R}(\mathcal{A})$.
\end{theorem}
\begin{proof}
Equations (\ref{eq18}) are easily verified on braided differential forms
of order $0$ and $1$. Since $\mathrm{pr}$ is a homomorphism of the braided
wedge product the claim follows.
\end{proof}
As a special case we recover that the Cartan calculus on a
closed embedded submanifold $\iota\colon N\rightarrow M$ of a smooth manifold
$M$ is obtained by the pullback 
$\iota^*\colon\Omega^\bullet(M)\rightarrow\Omega^\bullet(N)$ of differential forms
and restriction $\iota^*\colon\mathfrak{X}_t^\bullet(M)\rightarrow
\mathfrak{X}^\bullet(N)$ of tangent multivector fields to $N$.
The latter is defined for any $X\in\mathfrak{X}^1_t(M)$ as the unique vector field
$X|_N\in\mathfrak{X}^1(N)$, which is $\iota$-related to $X$, i.e.
$T_q\iota(X|_N)_q=X_{\iota(q)}$ for all $q\in N$, where
$T_q\iota\colon T_qN\rightarrow T_{\iota(q)}M$ denotes the tangent map
(c.f. \cite{Lee2003}~Lemma~5.39).
In particular,
\begin{align*}
    \mathscr{L}_{\iota^*(X)}\iota^*(\omega)
    =\iota^*(\mathscr{L}_X\omega),~~
    \mathrm{i}_{\iota^*(X)}\iota^*(\omega)
    =\iota^*(\mathrm{i}_X\omega)
    \text{ and }
    \mathrm{d}\iota^*(\omega)
    =\iota^*(\mathrm{d}\omega)
\end{align*}
for all $X\in\mathfrak{X}^\bullet(M)$ and $\omega\in\Omega^\bullet(M)$.

In the next proposition we prove that the gauge equivalence given by the Drinfel'd functor
is compatible with the notion of submanifold ideal, i.e. the projection
to submanifold algebras and twisting commute. In the particular case of a
cocommutative Hopf algebra with trivial triangular structure this means that
twist quantization and projection to the submanifold algebra commute
(see also \cite{GaetanoThomas19}).
\begin{proposition}
For any twist $\mathcal{F}$ on $H$,
the submanifold algebra projection of the twist deformation 
$(\mathfrak{X}_t^\bullet(\mathcal{A})_\mathcal{F},\wedge_\mathcal{F},
\llbracket\cdot,\cdot\rrbracket_\mathcal{F})$ of the braided
Gerstenhaber algebra of tangent multivector fields on $\mathcal{A}$
coincides with the twist deformation
$(\mathfrak{X}^\bullet_\mathcal{R}(\mathcal{A}/\mathcal{C})_\mathcal{F},
\wedge_\mathcal{F},
\llbracket\cdot,\cdot\rrbracket_\mathcal{F})$
of the braided Gerstenhaber algebra of braided multivector fields on
$\mathcal{A}/\mathcal{C}$.
Moreover, the twisted Cartan calculus on $\mathcal{A}/\mathcal{C}$
is given by the projection of the twisted Cartan calculus on $\mathcal{A}$.
Namely, $\Omega^\bullet_\mathcal{R}(\mathcal{A}/\mathcal{C})_\mathcal{F}
=\mathrm{pr}(\Omega^\bullet_\mathcal{R}(\mathcal{A})_\mathcal{F})$,
\begin{equation}
    \mathscr{L}^\mathcal{F}_{\mathrm{pr}(X)}\mathrm{pr}(\omega)
    =\mathrm{pr}(\mathscr{L}^\mathcal{F}_X\omega),~
    \mathrm{i}^\mathcal{F}_{\mathrm{pr}(X)}\mathrm{pr}(\omega)
    =\mathrm{pr}(\mathrm{i}^\mathcal{F}_X\omega)
    \text{ and }
    \mathrm{d}(\mathrm{pr}(\omega))
    =\mathrm{pr}(\mathrm{d}\omega)
\end{equation}
for all $X\in\mathfrak{X}^\bullet_t(\mathcal{A})_\mathcal{F}$ and
$\omega\in\Omega^\bullet_\mathcal{R}(\mathcal{A})_\mathcal{F}$.
\end{proposition}
\begin{proof}
Note that the twist deformation of $\mathfrak{X}^\bullet_t(\mathcal{A})$ is a
braided Gerstenhaber algebra since the braided multivector fields which are 
tangent to $\mathcal{C}$ are an
$H$-submodule and a braided symmetric $\mathcal{A}$-sub-bimodule of
$\mathfrak{X}^\bullet_\mathcal{R}(\mathcal{A})$.
We already noticed that $\mathrm{pr}\colon\mathfrak{X}^\bullet_t(\mathcal{A})
\rightarrow\mathfrak{X}^\bullet_\mathcal{R}(\mathcal{A}/\mathcal{C})$ is
surjective.
Let $X,Y\in\mathfrak{X}^\bullet_t(\mathcal{A})_\mathcal{F}$
and $a\in\mathcal{A}$. Then
$$
\mathrm{pr}(X)\wedge_\mathcal{F}\mathrm{pr}(Y)
=(\mathcal{F}_1^{-1}\rhd\mathrm{pr}(X))
\wedge_\mathcal{R}(\mathcal{F}_2^{-1}\rhd\mathrm{pr}(Y))
=\mathrm{pr}(X\wedge_\mathcal{F}Y),
$$
and similarly $\llbracket\mathrm{pr}(X),\mathrm{pr}(Y)\rrbracket_\mathcal{F}
=\mathrm{pr}(\llbracket X,Y\rrbracket_\mathcal{F})$
and $\mathrm{pr}(a)\cdot_\mathcal{F}\mathrm{pr}(X)
=\mathrm{pr}(a\cdot_\mathcal{F}X)$ follow.
Moreover,
$$
\mathscr{L}^\mathcal{F}_{\mathrm{pr}(X)}\mathrm{pr}(\omega)
=\mathscr{L}^\mathcal{R}_{\mathcal{F}_1^{-1}\rhd\mathrm{pr}(X)}
(\mathcal{F}_2^{-1}\rhd\mathrm{pr}(\omega))
=\mathrm{pr}(\mathscr{L}^\mathcal{F}_X\omega)
$$
and
$$
\mathrm{i}^\mathcal{F}_{\mathrm{pr}(X)}\mathrm{pr}(\omega)
=\mathrm{i}^\mathcal{R}_{\mathcal{F}_1^{-1}\rhd\mathrm{pr}(X)}
(\mathcal{F}_2^{-1}\rhd\mathrm{pr}(\omega))
=\mathrm{pr}(\mathrm{i}^\mathcal{F}_X\omega)
$$
for all $X\in\mathfrak{X}^\bullet_t(\mathcal{A})$ and 
$\omega\in\Omega^\bullet_\mathcal{R}(\mathcal{A})$ by Theorem~\ref{lemma03}.
This concludes the proof of the proposition.
\end{proof}

\subsection{Equivariant Covariant Derivatives on Submanifolds}

In this section we discuss equivariant covariant derivatives on submanifold algebras
and study under which conditions equivariant covariant derivatives and metrics
allow for projections. Accepting two mild axioms the latter is possible for a
given strongly non-degenerate equivariant metric.
Furthermore, the projection of the equivariant covariant derivative is compatible
with the notion of curvature, torsion and twist deformation. 

Fix a submanifold ideal $\mathcal{C}$
of $\mathcal{A}$ and a strongly non-degenerate equivariant metric ${\bf g}$ on 
$\mathcal{A}$. Then
there is a direct sum decomposition
\begin{equation}
    \mathfrak{X}^1_\mathcal{R}(\mathcal{A})
    =\mathfrak{X}^1_t(\mathcal{A})\oplus\mathfrak{X}^1_n(\mathcal{A}),
\end{equation}
where $\mathfrak{X}^1_n(\mathcal{A})$ are the so-called \textit{braided normal
vector fields} with respect to $\mathcal{C}$ and ${\bf g}$, defined to be the
subspace orthogonal to $\mathfrak{X}^1_t(\mathcal{A})$ with respect to ${\bf g}$.
Then, $\mathrm{pr}_{\bf g}\colon\mathfrak{X}^1_\mathcal{R}(\mathcal{A})
\rightarrow\mathfrak{X}^1_\mathcal{R}(\mathcal{A}/\mathcal{C})$ is the $\Bbbk$-linear map
which first projects to the first addend in the above decomposition and applies
$\mathrm{pr}\colon\mathfrak{X}^1_t(\mathcal{A})\rightarrow\mathfrak{X}^1_\mathcal{R}
(\mathcal{A}/\mathcal{C})$ afterwards. In particular $\mathrm{pr}_{\bf g}(X)
=\mathrm{pr}(X)$ for all $X\in\mathfrak{X}^1_t(\mathcal{A})$. 
In a next step we define a 
$\Bbbk$-linear map ${\bf g}_{\mathcal{A}/\mathcal{C}}\colon
\mathfrak{X}^1_\mathcal{R}(\mathcal{A}/\mathcal{C})
\otimes_{\mathcal{A}/\mathcal{C}}
\mathfrak{X}^1_\mathcal{R}(\mathcal{A}/\mathcal{C})
\rightarrow\mathcal{A}/\mathcal{C}$ by
\begin{equation}
    {\bf g}_{\mathcal{A}/\mathcal{C}}(\mathrm{pr}_{\bf g}(X),\mathrm{pr}_{\bf g}(Y))
    =\mathrm{pr}_{\bf g}({\bf g}(X,Y))
\end{equation}
for all $X,Y\in\mathfrak{X}^1_\mathcal{R}(\mathcal{A})$. It is well-defined if
$\mathfrak{X}^1_0(\mathcal{A})=\ker\mathrm{pr}$ has the following property.
\begin{equation*}
\begin{split}
    \textbf{Axiom 1: }&
    \text{for every $X\in\mathfrak{X}^1_0(\mathcal{A})$ there are finitely many
    $c_i\in\mathcal{C}$ and $X^i\in\mathfrak{X}^1_t(\mathcal{A})$}\\
    &\text{such that $X=\sum_{i}c_iX^i$.}
\end{split}
\end{equation*}
This is for example the case if $\mathfrak{X}^1_0(\mathcal{A})$ is finitely
generated as a $\mathcal{C}$-bimodule. If ${\bf g}$ is non-degenerate the
projection ${\bf g}_{\mathcal{A}/\mathcal{C}}$ is not non-degenerate in general.
However, if we assume the following property of ${\bf g}$, the projection
${\bf g}_{\mathcal{A}/\mathcal{C}}$ is strongly non-degenerate if ${\bf g}$ is.
\begin{equation*}
    \textbf{Axiom 2: }
    \text{if $X\in\mathfrak{X}^1_t(\mathcal{A})$, then ${\bf g}(X,X)\in\mathcal{C}$
    implies $X\in\mathfrak{X}^1_0(\mathcal{A})$.}
\end{equation*}
Note that in the case of closed embedded smooth manifolds both axiom 1 and 2 are
satisfied.
\begin{proposition}\label{lemma05}
For any strongly
non-degenerate equivariant metric ${\bf g}$ on $\mathcal{A}$ such that the axioms 1 and 2 
are satisfied, ${\bf g}_{\mathcal{A}/\mathcal{C}}$ is a well-defined strongly
non-degenerate equivariant metric on $\mathcal{A}/\mathcal{C}$.
The projection
\begin{equation}
    \nabla^{\mathcal{A}/\mathcal{C}}_{\mathrm{pr}(X)}\mathrm{pr}(Y)
    =\mathrm{pr}_{\bf g}(\nabla^\mathcal{R}_XY),
\end{equation}
of an equivariant covariant derivative $\nabla^\mathcal{R}\colon
\mathfrak{X}^1_\mathcal{R}(\mathcal{A})
\otimes\mathfrak{X}^1_\mathcal{R}(\mathcal{A})
\rightarrow\mathfrak{X}^1_\mathcal{R}(\mathcal{A})$ 
on $\mathcal{A}$,
where $X,Y\in\mathfrak{X}^1_t(\mathcal{A})$, is an equivariant covariant derivative
with respect to $\mathcal{R}$ on $\mathcal{A}/\mathcal{C}$.
If furthermore,
$\nabla^\mathcal{R}$ is the equivariant Levi-Civita covariant derivative with
respect to ${\bf g}$, $\nabla^{\mathcal{A}/\mathcal{C}}$ is the equivariant Levi-Civita
covariant derivative on $\mathcal{A}/\mathcal{C}$ with respect to
${\bf g}_{\mathcal{A}/\mathcal{C}}$.
\end{proposition}
\begin{proof}
Axiom~1 assures ${\bf g}_{\mathcal{A}/\mathcal{C}}$ to be
well-defined, since
\begin{align*}
    {\bf g}_{\mathcal{A}/\mathcal{C}}(\mathrm{pr}_{\bf g}(X),\mathrm{pr}_{\bf g}(Y))
    =&{\bf g}_{\mathcal{A}/\mathcal{C}}\bigg(
    \mathrm{pr}_{\bf g}\bigg(\sum_ic_i\cdot X^i\bigg),
    \mathrm{pr}_{\bf g}(Y)\bigg)
    =\mathrm{pr}\bigg({\bf g}\bigg(\sum_ic_i\cdot X^i,Y\bigg)\bigg)\\
    =&\mathrm{pr}\bigg(\underbrace{\sum_ic_i\cdot 
    {\bf g}(X^i,Y)}_{\in\mathcal{C}}\bigg)
    =0
\end{align*}
and similarly ${\bf g}_{\mathcal{A}/\mathcal{C}}
(\mathrm{pr}_{\bf g}(Y),\mathrm{pr}_{\bf g}(X))=0$
for all $X\in\mathfrak{X}^1_0(\mathcal{A})$ and
$Y\in\mathfrak{X}^1_\mathcal{R}(\mathcal{A})$.
Let $X\in\mathfrak{X}^1_\mathcal{R}(
\mathcal{A}/\mathcal{C})$ and choose $Y\in\mathfrak{X}^1_t(\mathcal{A})$ such 
that $\mathrm{pr}(Y)=X$. Then
\begin{align*}
    0
    ={\bf g}_{\mathcal{A}/\mathcal{C}}(X,X)
    =\mathrm{pr}({\bf g}(Y,Y))
\end{align*}
implies ${\bf g}(Y,Y)\in\mathcal{C}$, i.e. $Y\in\mathfrak{X}^1_0(\mathcal{A})$ by
Axiom~2. In other words ${\bf g}_{\mathcal{A}/\mathcal{C}}(X,X)=0$ implies
$X=0$, which is equivalent to the statement that $X\neq 0$ implies
${\bf g}_{\mathcal{A}/\mathcal{C}}(X,X)\neq 0$, i.e. strong non-degeneracy of
${\bf g}_{\mathcal{A}/\mathcal{C}}$. From Axiom~1 it follows that
$\nabla^{\mathcal{A}/\mathcal{C}}$ is well-defined. In fact, for
$X=\sum_ic_i\cdot X^i\in\mathfrak{X}^1_0(\mathcal{A})$ and
$Y\in\mathfrak{X}^1_t(\mathcal{A})$ we obtain
\begin{align*}
    \nabla^{\mathcal{A}/\mathcal{C}}_{\mathrm{pr}(X)}\mathrm{pr}(Y)
    =\mathrm{pr}_{\bf g}(\nabla^\mathcal{R}_XY)
    =\mathrm{pr}_{\bf g}\bigg(\underbrace{
    \sum_ic_i\cdot\nabla^\mathcal{R}_{X^i}Y}_{
    \in\mathfrak{X}^1_0(\mathcal{A})}\bigg)
    =0
\end{align*}
and
\begin{align*}
    \nabla^{\mathcal{A}/\mathcal{C}}_{\mathrm{pr}(Y)}\mathrm{pr}(X)
    =&\mathrm{pr}_{\bf g}\bigg(\sum_i\nabla^\mathcal{R}_Y(c_i\cdot X^i)\bigg)\\
    =&\sum_i\mathrm{pr}_{\bf g}(\underbrace{\overbrace{
    (\mathscr{L}^\mathcal{R}_Yc_i)}^{\in\mathcal{C}}\cdot X^i}_{
    \in\mathfrak{X}^1_0(\mathcal{A})}
    +\underbrace{\overbrace{(\mathcal{R}_1^{-1}\rhd c_i)}^{\in\mathcal{C}}\cdot
    \nabla^\mathcal{R}_{\mathcal{R}_2^{-1}\rhd X^i}Y}_{
    \in\mathfrak{X}^1_0(\mathcal{A})})
    =0,
\end{align*}
since $\nabla^\mathcal{R}$ is left $\mathcal{A}$-linear in the first
argument and satisfies a braided Leibniz rule in the second
argument. The remaining results are proven in \cite{ThomasPhDThesis}.
\end{proof}
We would like to stress that the assumptions of Proposition \ref{lemma05} are sufficient
to project strongly non-degenerate equivariant metrics and equivariant covariant derivatives
to submanifold algebras. Whether those conditions are also necessary is part
of further investigation.
Fix an equivariant covariant derivative
$\nabla^\mathcal{R}$ on $\mathcal{A}$ and
a strongly non-degenerate equivariant metric
${\bf g}$ such that axiom $1$ and $2$ hold.
The curvature and torsion of a projected equivariant covariant derivative
coincide with the projection of the curvature and torsion of $\nabla^\mathcal{R}$.
\begin{corollary}
The curvature $R^{\nabla^{\mathcal{A}/\mathcal{C}}}$ and the torsion
$\mathrm{Tor}^{\nabla^{\mathcal{A}/\mathcal{C}}}$ of the projected
equivariant covariant derivative
$\nabla^{\mathcal{A}/\mathcal{C}}$ are given by
\begin{equation}
    R^{\nabla^{\mathcal{A}/\mathcal{C}}}(\mathrm{pr}(X),\mathrm{pr}(Y))(\mathrm{pr}(Z))
    =\mathrm{pr}_{\bf g}(R^{\nabla^\mathcal{R}}(X,Y)Z)
\end{equation}
and
\begin{equation}
    \mathrm{Tor}^{\nabla^{\mathcal{A}/\mathcal{C}}}(\mathrm{pr}(X),\mathrm{pr}(Y))
    =\mathrm{pr}_{\bf g}(\mathrm{Tor}^{\nabla^\mathcal{R}}(X,Y))
\end{equation}
for all $X,Y,Z\in\mathfrak{X}^1_t(\mathcal{A})$.
\end{corollary}
One extends the projection $\mathrm{pr}_{\bf g}\colon
\mathfrak{X}^\bullet_\mathcal{R}(\mathcal{A})
\rightarrow\mathfrak{X}^\bullet_\mathcal{R}(\mathcal{A}/\mathcal{R})$
to braided multivector fields by defining it to coincide with $\mathrm{pr}$
on $\mathcal{A}$ and to be a homomorphism of the braided wedge product on
higher wedge powers. On braided differential forms we set
$\mathrm{pr}_{\bf g}=\mathrm{pr}$.
\begin{corollary}
The equivariant covariant derivatives
$$
\nabla^{\mathcal{A}/\mathcal{C}}\colon
\mathfrak{X}_\mathcal{R}^1(\mathcal{A}/\mathcal{C})
\otimes\mathfrak{X}_\mathcal{R}^\bullet(\mathcal{A}/\mathcal{C})
\rightarrow\mathfrak{X}_\mathcal{R}^\bullet(\mathcal{A}/\mathcal{C})
\text{ and }
\tilde{\nabla}^{\mathcal{A}/\mathcal{C}}
\colon\mathfrak{X}_\mathcal{R}^1(\mathcal{A}/\mathcal{C})
\otimes\Omega_\mathcal{R}^\bullet(\mathcal{A}/\mathcal{C})
\rightarrow\Omega_\mathcal{R}^\bullet(\mathcal{A}/\mathcal{C}),
$$
induced by the projected equivariant covariant derivative $\nabla^{\mathcal{A}/\mathcal{C}}$
on $\mathcal{A}/\mathcal{C}$ according to Proposition~\ref{prop02},
are projected from the covariant derivatives induced by $\nabla^\mathcal{R}$. Namely,
\begin{equation}
    \nabla^{\mathcal{A}/\mathcal{C}}_{\mathrm{pr}(X)}\mathrm{pr}(Y)
    =\mathrm{pr}_{\bf g}(\nabla^\mathcal{R}_XY)
    \text{ and }
    \tilde{\nabla}^{\mathcal{A}/\mathcal{C}}_{\mathrm{pr}(X)}\mathrm{pr}(\omega)
    =\mathrm{pr}_{\bf g}(\tilde{\nabla}^\mathcal{R}_X\omega)
\end{equation}
for all $X\in\mathfrak{X}^1_t(\mathcal{A})$, $Y\in\mathfrak{X}^\bullet_t(\mathcal{A})$ and
$\omega\in\Omega_\mathcal{R}^\bullet(\mathcal{A})$.
\end{corollary}
Furthermore, twisted equivariant covariant derivatives behave well under projection.
\begin{proposition}
For any twist $\mathcal{F}$ on $H$, the projection
of the twisted equivariant covariant derivative coincides with the twist deformation
of the projected equivariant covariant derivative, i.e.
$(\nabla^{\mathcal{A}/\mathcal{C}})^\mathcal{F}_{\mathrm{pr}(X)}\mathrm{pr}(Y)
=\mathrm{pr}_{\bf g}(\nabla^\mathcal{F}_XY)$
for all $X,Y\in\mathfrak{X}_t^1(\mathcal{A})_\mathcal{F}$. 
Its curvature and torsion are given by
\begin{equation}
\begin{split}
    R^{(\nabla^{\mathcal{A}/\mathcal{C}})^\mathcal{F}}&
    (\mathrm{pr}(X),\mathrm{pr}(Y))(\mathrm{pr}(Z))\\
    =&R^{\nabla^{\mathcal{A}/\mathcal{C}}}\bigg(
    (\mathcal{F}_{1(1)}^{-1}\mathcal{F}_1^{'-1})
    \rhd\mathrm{pr}(X),
    (\mathcal{F}_{1(2)}^{-1}\mathcal{F}_2^{'-1})\rhd\mathrm{pr}(Y)
    \bigg)(\mathcal{F}_2^{-1}\rhd\mathrm{pr}(Z))\\
    =&\mathrm{pr}\bigg(R^{\nabla^\mathcal{R}}\bigg(
    (\mathcal{F}_{1(1)}^{-1}\mathcal{F}_1^{'-1})
    \rhd X,
    (\mathcal{F}_{1(2)}^{-1}\mathcal{F}_2^{'-1})\rhd Y
    \bigg)(\mathcal{F}_2^{-1}\rhd Z)
    \bigg)
\end{split}
\end{equation}
and
\begin{equation}
\begin{split}
    \mathrm{Tor}^{(\nabla^{{\mathcal{A}/\mathcal{C}}})^\mathcal{F}}
    (\mathrm{pr}(X),\mathrm{pr}(Y))
    =&\mathrm{Tor}^{\nabla^{\mathcal{A}/\mathcal{C}}}
    (\mathcal{F}_1^{-1}\rhd\mathrm{pr}(X),
    \mathcal{F}_2^{-1}\rhd\mathrm{pr}(Y))\\
    =&\mathrm{pr}\bigg(\mathrm{Tor}^{\nabla^\mathcal{R}}(\mathcal{F}_1^{-1}\rhd X,
    \mathcal{F}_2^{-1}\rhd Y)\bigg)
\end{split}
\end{equation}
for all $X,Y,Z\in\mathfrak{X}^1_t(\mathcal{A})_\mathcal{F}$, respectively.
Similar statements hold for the
induced (twisted) equivariant covariant derivatives on braided differential forms and
braided multivector fields.
\end{proposition}
\begin{proof}
For all $X,Y\in\mathfrak{X}^1_t(\mathcal{A})_\mathcal{F}$ one obtains
\begin{align*}
    \mathrm{pr}_{\bf g}(\nabla^\mathcal{F}_XY)
    =&\mathrm{pr}_{\bf g}(\nabla^\mathcal{R}_{\mathcal{F}_1^{-1}\rhd X}
    (\mathcal{F}_2^{-1}\rhd Y))
    =\nabla^{\mathcal{A}/\mathcal{C}}_{\mathrm{pr}(\mathcal{F}_1^{-1}\rhd X)}
    (\mathrm{pr}(\mathcal{F}_2^{-1}\rhd Y))\\
    =&\nabla^{\mathcal{A}/\mathcal{C}}_{\mathcal{F}_1^{-1}\rhd\mathrm{pr}(X)}
    (\mathcal{F}_2^{-1}\rhd\mathrm{pr}(Y))
    =(\nabla^{\mathcal{A}/\mathcal{C}})^\mathcal{F}_{\mathrm{pr}(X)}\mathrm{pr}(Y)
\end{align*}
and similarly one proves the statements about the induced equivariant covariant derivatives.
\end{proof}

\section*{Acknowledgments}

The author is grateful to Francesco D'Andrea and Gaetano Fiore for their 
constant support. In particular, he wants to thank the latter for
introducing him to the concept of twisted Cartan calculus and proposing to prove
compatibility with projections to submanifold algebras. Special thanks go to
Paolo Aschieri and the referee for their valuable comments on the first version
of this paper. Furthermore, the author
wants to thank Stefan Waldmann for posing a question at the DQ seminar about the
existence of a braided Cartan calculus for every triangular structure.

\end{document}